\documentclass{amsart}

\usepackage{amsmath,amsthm,amssymb,amsfonts}
\usepackage{bbm}
\usepackage[mathscr]{eucal}
\usepackage[english]{babel}
\usepackage{graphicx}

\allowdisplaybreaks[1]

\newtheorem{thm}{Theorem}[section]
\newtheorem{cor}[thm]{Corollary}
\newtheorem{lem}[thm]{Lemma}
\newtheorem{prop}[thm]{Proposition}
\newtheorem{rem}[thm]{Remark}

\newtheorem{defn}[thm]{Definition}

\theoremstyle{remark}

\numberwithin{equation}{section}

\newcommand{\thmref}[1]{Theorem~\ref{#1}}
\newcommand{\lemref}[1]{Lemma \ref{#1}}
\newcommand{\propref}[1]{Proposition \ref{#1}}
\newcommand{\corref}[1]{Corollary \ref{#1}}
\newcommand{\remref}[1]{Remark \ref{#1}}

\def\BB{{\mathbb B}}
\def\CC{{\mathbb C}}
\def\NN{{\mathbb N}}
\def\RR{{\mathbb R}}
\def\SS{{{\mathbb S}^{d-1}}}

\def\ZZ{{\mathbb Z}}

\def\cB{\mathcal{B}}
\def\cC{\mathcal{C}}

\def\cF{\mathcal{F}}

\def\cH{\mathcal{H}}

\def\cN{\mathcal{N}}
\def\cP{\mathcal{P}}
\def\cQ{\mathcal{Q}}

\def\cS{\mathcal{S}}
\def\cU{\mathcal{U}}

\def\cX{\mathcal{X}}
\def\cY{\mathcal{Y}}
\def\cZ{\mathcal{Z}}

\def\fb{{\mathfrak b}}
\def\ff{{\mathfrak f}}
\def\fg{{\mathfrak g}}

\def\fD{{\mathfrak D}}

\def\nA{\mathscr{D}}

\def\HHH{\mathscr{H}}

\def\sA{\mathscr{A}}

\newcommand{\PP}{Z}
\def\eps{{\varepsilon}}
\def\u*{*}

\def\bbb{b}
\def\kk{k}

\def\supp{\operatorname{supp}}

\def\QQ{{\cQ}}
\def\NNN{\mathscr{N}}
\def\ONE{{\mathbbm 1}}
\def\aa{\varphi}
\def\LL{\Psi}
\def\YY{\mathscr Y}

\def\nA{A}

\def\ww{\widetilde{w}}
\def\betaw{\bar{\beta}}

\newcommand{\sepTL}[3]{\overset{\!\!\!\!\circ}{#1_\infty^{#2#3}}}
\newcommand{\sepTLb}[3]{\overset{\!\!\!\!\!\!\circ}{#1_\infty^{#2#3}}}

\def\ups{\upsilon}
\def\Ups{\Upsilon}
\def\bb{\beta}
\def\BMO{\operatorname{BMO}}
\def\VMO{\operatorname{VMO}}
\def\avg{\operatorname{Avg}}
\def\BMOH{\operatorname{BMOH}}
\def\VMOH{\operatorname{VMOH}}

\begin{document}

\title[Nonlinear approximation of harmonic functions in BMO]
{Nonlinear approximation of harmonic functions from shifts of the Newtonian kernel in BMO}

\author[K. G. Ivanov]{Kamen G. Ivanov}
\address{Institute of Mathematics and Informatics\\
Bulgarian Academy of Sciences\\
Sofia, Bulgaria}
\email{kamen@math.bas.bg}

\author[P. Petrushev]{Pencho Petrushev}
\address{Department of Mathematics\\
University of South Carolina\\
Columbia, SC}
\email{pencho@math.sc.edu}

\thanks{Both authors have been supported by Grant KP-06-N62/4 of the Fund for Scientific Research of the Bulgarian Ministry of Education and Science.}

\begin{abstract}
We study nonlinear $n$-term approximation of harmonic functions on the unit ball in $\RR^d$ from
linear combinations of shifts of the Newtonian kernel (fundamental solution of the Laplace equation) in BMO.
A sharp Jackson estimate is established that naturally involves certain Besov spaces.
The method for obtaining this result is based on the construction of highly localized frames for
Besov spaces and VMO on the sphere whose elements are linear combinations of
a fixed number of shifts of the Newtonian kernel.
\end{abstract}


\maketitle

\tableofcontents

\section{Introduction}\label{s1}

The shifts of the Newtonian kernel $\frac{1}{|x|^{d-2}}$ in dimension $d>2$ or $\ln\frac{1}{|x|}$ if $d=2$
are principle building blocks in Potential theory.
Our main goal is to obtain optimal rates of nonlinear $n$-term approximation of harmonic functions
from finite linear combinations of shifts of the Newtonian kernel in the harmonic $\BMO$ space on the unit ball $B^d$ in $\RR^d$.

To be more specific, let  $\NNN_n$ be the set of all harmonic functions $G$ on $B^d$ that can be represented in the form
\begin{equation}\label{lin-comb}
G(x)=a_0+\sum_{j=1}^n \frac{a_j}{|x-y_j|^{d-2}}\;\;\hbox{if} \;\; d>2
\;\;\hbox{or}\;\;
G(x)=a_0+\sum_{j=1}^n a_j\ln \frac{1}{|x-y_j|}\;\;\hbox{if} \;\; d=2,
\end{equation}
where the poles $\{y_j\}$ are in $\RR^d\setminus \overline{B^d}$ and the coefficients $a_j\in \CC$.
The points $\{y_j\}$ are allowed to vary with the function $G$ and hence $\NNN_n$ is nonlinear.

Denote by $\BMOH$ the set of all harmonic function $U$ on $B^d$
with boundary value function $f_U$ in $\BMO$
and let $\|\cdot\|_{\BMOH}$ be the induced norm (see Definition~\ref{def:BMOH}).
Since the space $\BMOH$ is non-separable it is natural to approximate
harmonic functions in the space $\VMOH$ of all harmonic function $U$ on $B^d$
with boundary value function $f_U$ in $\VMO$ which is a separable subspace of $\BMOH$.

Given $U\in \VMOH$ we define
\begin{equation}\label{def-En}
E_n(U)_{\BMOH}:=\inf_{G\in\NNN_n}\|U-G\|_{\BMOH}.
\end{equation}
Our goal is to establish sharp estimates on $\{E_n(U)\}$ for functions
in the ``natural" harmonic Besov type spaces on $B^d$.
Naturally, this approximation is equivalent to approximation of functions $f\in \VMO(\SS)$
on the unit sphere $\SS$ in $\RR^d$ from $\NNN_n$ in the $\BMO(\SS)$ norm.

This is a followup paper to \cite{IP2}, where we treated the problem
for nonlinear $n$-term approximation of functions in the harmonic Hardy spaces $\HHH^p(B^d)$, $0<p<\infty$, 
from shifts of the Newtonian kernel. 

Why we consider approximation of harmonic functions in the $\BMO$ norm rather than in the $L^\infty$ norm?
The main reason for this is that the space $\BMO$ fits well in the Littlewood-Paley theory,
while $L^\infty$ is completely outside this theory.

As can be expected the harmonic Besov spaces
\begin{equation*}
B^{s\tau}_\tau(\HHH)
\quad\hbox{with}\quad
1/\tau=s/(d-1), \; s>0,
\end{equation*}
play a major role in the approximation process here.
The principle result in this article (Theorem~\ref{thm:Jackson_VMOH}) asserts that
if $U\in B^{s\tau}_\tau(\HHH)$,
then $U\in\VMOH$ and
\begin{equation}\label{Jack}
E_n(U)_{\BMOH} \le cn^{-s/(d-1)}\|U\|_{B^{s\tau}_\tau(\HHH)}.
\end{equation}

The main difficulty in approximating functions using shift of the Newtonian kernel
is rooted in its poor localization.
The highly localized summability kernels in terms of shifts of the Newtonian kernel, constructed in \cite{IP1},
come to the rescue.
To obtain our approximation result we follow the approach from \cite{IP2}.
Namely, we first use the kernels from \cite{IP1} to construct a pair of dual frames $\{\theta_\xi\}$, $\{\tilde\theta_\xi\}$
for the respective Besov spaces and for VMO on $\SS$ whose elements $\{\theta_\xi\}$  are linear combinations of
a fixed number of shifts of the Newtonian kernel and are well localized.
Second, we apply an intermediate nonlinear $n$-term approximation from $\{\theta_\xi\}$
to the boundary value function $f_U$ of the harmonic function $U$ to be approximated.
Finally, using the fact that each $\theta_\xi$ is a finite linear combination of shifts of the Newtonian kernel
we obtain the desired estimate by harmonic extension to $B^d$ of the approxiamant.

The realization of the above program
depends heavily on the fact that
$\BMO$ can be identified with the Triebel-Lizorkin space $\cF^{02}_\infty$.
In a natural progression we are led to the problem of approximation of
functions in the more general Triebel-Lizorkin spaces $\cF^{sq}_\infty(\SS)$ on the sphere
from shifts of the Newtonian kernel.
However, the spaces $\cF^{sq}_\infty(\SS)$ are non-separable and, therefore, not appropriate for approximation.
To remedy this inconvenience, we introduce the \emph{separable} Triebel-Lizorkin spaces $\sepTL{\cF}{s}{q}(\SS)$,
study their properties and create a pair of dual frames for them from shifts of the Newtonian kernel.
It turns out that, in particular, we have $\sepTL{\cF}{0}{2}(\SS)=\VMO(\SS)$.
The gist of our approach is that the approximation in $\cF^{sq}_\infty(\SS)$
reduces to nonlinear $n$-term approximation in the respective sequence spaces.
In \thmref{thm:Jackson_TL} we establish the following estimate
\begin{equation}\label{Jack3}
E_n(f)_{\cF^{0q}_\infty(\SS)} \le cn^{-s/(d-1)}\|f\|_{\cB^{s\tau}_\tau(\SS)},\quad 0<q<\infty,
\end{equation}
which leads to (\thmref{thm:Jackson_VMO})
\begin{equation}\label{Jack2}
E_n(f)_{\BMO(\SS)} \le cn^{-s/(d-1)}\|f\|_{\cB^{s\tau}_\tau(\SS)}.
\end{equation}
In turn, this estimate implies our main estimate \eqref{Jack}.

\medskip

\noindent
{\em Conjecture: Bernstein inequality.}
We conjecture that the following Bernstein type inequality is valid:
Let $s>0$, and $1/\tau=s/(d-1)$. Then
\begin{equation}\label{bernstein}
\|G\|_{B_\tau^{s\tau}(\HHH)} \le cn^{s/(d-1)}\|G\|_{\BMOH}, \quad \forall G\in \NNN_n.
\end{equation}
If valid this estimate along with the Jackson estimate (\ref{Jack}) would lead to a complete characterization
of the rates of approximation (approximation spaces) of nonlinear $n$-term approximation
of harmonic functions in $\VMOH(B^d)$ from shifts of the Newtonian kernel.
Theorems \ref{thm:Jackson} and \ref{thm:Bernstein}, which are sequence space analogs to our Jackson and Bernstein inequalities, 
support this conjecture.

\smallskip

\noindent
{\em Outline.}
This article is organized as follows.
In Section~\ref{sec:background}, we set some basic notations and assemble background material
about spherical harmonics, highly localized kernels, maximal $\delta$-nets, cubature formulas, and nested structures on the sphere.
In Section~\ref{s2}, we present some basic facts about Besov and Triebel-Lizorkin spaces
on $\SS$ developed in \cite{IP};
we also recall the construction of frames for Besov and Triebel-Lizorkin spaces on $\SS$ from \cite{NPW2} and \cite{IP2}.
In Section~\ref{sec:F-spaces}, the Triebel-Lizorkin spaces $\cF^{sq}_\infty$ and the separable Triebel-Lizorkin spaces $\sepTL{\cF}{s}{q}$ are developed,
including the decomposition of $\sepTL{\cF}{s}{q}$ via the needlet frame from \cite{NPW2}
and a frame whose elements are linear combinations of finitely many shifts of the Newtonian kernel.
The identifications of $\BMO(\SS)$ with $\cF^{02}_\infty(\SS)$ and of $\VMO(\SS)$ with $\sepTL{\cF}{0}{2}(\SS)$ are established in Section~\ref{sec:BMO-F-spaces}.
The harmonic $\BMO$ and $\VMO$ spaces are defined and discussed in Section~\ref{sec:harmonic-BMO}.
In Section~\ref{sec:approximation}, we present our main results on nonlinear $n$-term approximation of functions in
the Triebel-Lizorkin spaces $\cF^{sq}_\infty$, the $\BMO$ space on $\SS$ and the harmonic $\BMO$
space on $B^d$ from linear combinations of shifts of the Newtonian kernels.

\smallskip

\noindent
{\em Notation.}
Throughout this article we use standard notation introduced in \S \ref{s1_2}.
Positive constants will be denoted by $c$ and they may vary at every occurrence.
The notation $a\sim b$ will stand for $c^{-1}a\le b \le ca$.

\bigskip

\section{Background and technical ground work}\label{sec:background}

\subsection{Basic notation and simple inequalities}\label{s1_2}

In this article we use standard notation.
Thus $\RR^d$ stands for the $d$-dimensional Euclidean space.
The inner product of $x,y\in\RR^d$ is denoted by $x\cdot y =\sum_{k=1}^d x_ky_k$ and
the Euclidean norm of $x$ by $|x|=\sqrt{x\cdot x}$.
We write $\BB(x_0,r):=\{x : |x-x_0|< r\}$
and set $B^d:=\BB(0, 1)$, the open unit ball in $\RR^d$.

As usual $\NN_0$ stands for the set of non-negative integers.
For $\beta=(\beta_1,\dots,\beta_d)\in\NN_0^d$ the monomial $x^\beta$
is defined by $x^\beta:=x_1^{\beta_1}\dots x_d^{\beta_d}$ and $|\beta|:=\beta_1+\dots+\beta_d$ is its degree.
The set of all polynomials in $x\in\RR^d$ of total degree $n$ is denoted by $\cP_n^d$.
We denote $\partial_k:=\partial/\partial x_k$
and then $\partial^\beta:=\partial_1^{\beta_1}\dots \partial_d^{\beta_d}$ is a differential operator of order $|\beta|$,
the gradient operator is $\nabla:=(\partial_1,\dots, \partial_d)$,
and $\Delta:=\partial_1^2+\dots +\partial_d^2$ stands for the Laplacian.
When necessary we indicate the variable of differentiation by a subscript, e.g. $\partial^\beta_x$.

The unit sphere in $\RR^d$ is denoted by $\SS:=\{x : |x|=1\}$.
We denote by $\rho(x, y)$ the geodesic distance between $x, y\in\SS$, that is,
$\rho(x,y):=\arccos (x\cdot y)$.
The open spherical cap (ball on the sphere) centred at $\eta\in\SS$ of radius $r$ is denoted by $B(\eta,r)=\{x\in\SS : \rho(\eta,x)<r\}$.
For $B=B(\eta,r)$ the notation $\lambda B$ means the concentric cap $B(\eta,\lambda r)$.

The Lebesgue measure on $\SS$ is denoted by $\sigma$ and
we set $|E|:=\sigma(E)$ for a~measurable set $E\subset \SS$.
Thus, $\omega_d:=|\SS|=2\pi^{d/2}/\Gamma(d/2)$.

The area/volume of a spherical cap $B(x,r)$ on $\SS$, $d\ge 2$, is given by
\begin{equation*}
|B(x,r)|=\omega_{d-1}\int_0^r \sin^{d-2}v\,dv.
\end{equation*}
Hence
\begin{equation}\label{sph_cap2}
|B(x_1,r_1)|/|B(x_2,r_2)|\le (r_1/r_2)^{d-1},\quad 0<r_2\le r_1\le\pi,\quad x_1,x_2\in\SS,
\end{equation}
\begin{equation}\label{sph_cap}
1/c\le |B(x,r)|/r^{d-1}\le c,\quad 0<r\le\pi,\quad x\in\SS,
\end{equation}
where $c$ is a constant depending only on $d$.

The inner product of $f,g\in L^2(\SS)$ is given by
\begin{equation*}
	\left\langle f,g\right\rangle := \int_{\SS} f(y)\overline{g(y)}\,d\sigma(y).
\end{equation*}
The nonstandard convolution of functions
$F\in L^\infty[-1,1]$ and $g\in L(\SS)$ is defined by
\begin{equation}\label{def-conv}
F*g(x):=\left\langle F(x\cdot\bullet),\overline{g}\right\rangle=\int_{\SS}F(x\cdot y)g(y) d\sigma(y),\quad x\in\SS.
\end{equation}

We say that a function $f$ defined on $\SS$
is \emph{localized around} $\eta\in\SS$ with dilation factor $N$ and decay rate $M>0$ if the estimate
\begin{equation}\label{eq:local_1}
|f(x)|\le \kappa N^{d-1}(1+N\rho(\eta,x))^{-M},\quad x\in\SS,
\end{equation}
holds for some constant $\kappa >0$ independent of $N, x, \eta$.
The multiplier $N^{d-1}$ is used as part of the decay function in \eqref{eq:local_1} in order to have $\|f\|_{L^1(\SS)}\le c$.
Namely, for $M> d-1$ we have
\begin{equation}\label{eq:conv_1}
\int_{\SS} \frac{N^{d-1}}{(1+N\rho(\eta,y))^M}d\sigma(y)\le c,\quad \forall \eta\in\SS,~\forall N\ge 1,
\end{equation}
where $c=c(d)/(M-d+1)$ depends only on $d$ and $M$. The good localization is also demonstrated by
\begin{equation}\label{eq:conv_2}
\int_{\SS\backslash B(\eta,r)} \frac{N^{d-1}}{(1+N\rho(\eta,y))^M}d\sigma(y)\le c(Nr)^{-(M+d-1)},\quad \forall r\in(0,\pi].
\end{equation}
for all $\eta\in\SS$, $N\ge 1$, where $c=c(d)/(M-d+1)$ depends only on $d$ and $M$.

The weight function in the right-hand side of \eqref{eq:local_1} also has the property:
For any $\eta_1,\eta_2\in\SS$  with
$\rho(\eta_1,\eta_2)\le N^{-1}$
\begin{equation}\label{comp_local}
(1+N\rho(\eta_2, x))^{-M}\le 2^M(1+N\rho(\eta_1, x))^{-M},\quad \forall x\in\SS.
\end{equation}
Indeed, $1+N\rho(\eta_1, x)\le 1+N(\rho(\eta_1,\eta_2)+\rho(\eta_2, x))\le 2+N\rho(\eta_2, x)$, which implies \eqref{comp_local}.

\subsection{Spherical harmonics}\label{s1_3}

The spherical harmonics will be our main tool in dealing with harmonic functions
on the unit ball $B^d$ in $\RR^d$.

Denote by $\cH_\kk$ the space of all spherical harmonics of degree $\kk$ on $\SS$,
i.e. the restriction on $\SS$ of all harmonic  
homogeneous polynomials of degree $k$.
As is well known the dimension of $\cH_\kk$ is
$N(\kk, d)= \frac{2\kk+d-2}{\kk}\binom{\kk+d-3}{\kk-1} \sim \kk^{d-1}$.
Furthermore, the spaces $\cH_\kk$, $\kk=0, 1, \dots$, are orthogonal
and $L^2(\SS) =\bigoplus_{k\ge 0} \cH_k$.

Let $\{Y_{\kk \nu}: \nu=1, \dots, N(\kk, d)\}$ be a real-valued orthonormal basis for $\cH_\kk$.
Then the kernel of the orthogonal projector onto $\cH_\kk$ is given by
\begin{equation}\label{Pk}
Z_\kk(x\cdot y) = \sum_{\nu=1}^{N(\kk, d)} Y_{\kk \nu}(x)Y_{\kk \nu}(y),
\quad x, y\in \SS.
\end{equation}
As is well known (see e.g. \cite[Theorem 1.2.6]{DX})
\begin{equation}\label{def-Pk}
Z_\kk(x\cdot y)= \frac{\kk+\mu}{\mu\omega_d}\,
C^{\mu}_\kk(x\cdot y),
\quad x, y\in \SS,
\quad \mu:=\frac{d-2}{2},~d>2.
\end{equation}
Here $C_\kk^{\mu}$ is the Gegenbauer (ultraspherical) polynomial of degree
$\kk$ normalized by
$C_\kk^{\mu}(1)= \binom{\kk + 2\mu-1}{\kk}$.
The Gegenbauer polynomials are usually defined by the following generating function
\begin{equation*}
(1-2uz+z^2)^{-\mu}=\sum_{\kk=0}^\infty C^{\mu}_\kk(u)z^\kk,\quad |z|<1,~|u|<1.
\end{equation*}
The polynomials $C_\kk^{\mu}$, $k=0, 1, \dots$, are orthogonal in the space $L^2([-1, 1], w)$
with weight $w(u):= (1-u^2)^{\mu-1/2}$,
see \cite[p.~80, (4.7.1)]{Sz} or \cite[Table 18.3.1]{OLBC}.
In the case $d=2$ the kernel of the orthogonal projector onto $\cH_\kk$ takes the form
\begin{equation*}
Z_0(x\cdot y)=\frac{1}{2\pi},\quad 	Z_k(x\cdot y)=\frac{1}{\pi}T_k(x\cdot y),\quad k\ge 1,
\end{equation*}
where $T_k(u):=\cos(n\arccos u)$ is the $k$-th degree Chebyshev polynomial of the first kind.
We refer the reader to \cite{DX, Muller, SW} for the basics of spherical harmonics.

As is well known (see e.g. \cite[Theorem 1.4.5]{DX}) the spherical harmonics are
eigenfunctions of the Laplace-Beltrami operator $\Delta_0$ on $\SS$, namely,
\begin{equation}\label{eq:LB1}
-\Delta_0Y(x)=\kk(\kk+d-2)Y(x),\quad~x\in\SS,~\forall Y\in \cH_\kk.
\end{equation}

The set of all band-limited functions (i.e. spherical polynomials) on $\SS$ of degree $\le N$ will be denoted by $\Pi_N$,
i.e. $\Pi_N :=\bigoplus_{\kk=0}^N \cH_\kk$.

The Poisson kernel on the unit ball $B^d$ is given by
\begin{equation}\label{Poisson}
P(y,x):=\sum_{k= 0}^\infty |x|^k\PP_\kk\Big(\frac{x}{|x|}\cdot y\Big)
=\frac{1}{\omega_d}\frac{1-|x|^2}{|x-y|^d},
\quad |x| <1,\; y\in \SS.
\end{equation}

Kernels on the sphere $\SS$ of the form
\begin{equation}\label{def-Lam}
\Lambda_N(x\cdot y) := \sum_{\kk=0}^\infty
\lambda(\kk/N)\PP_\kk(x\cdot y),
\quad x, y\in \SS,
\quad N\ge 1,
\end{equation}
where $\lambda\in C^\infty[0,\infty)$ is compactly supported,
will play a key role in this article.
Observe that in this case
\begin{equation}\label{def-Lam-2}
\Lambda_N(u) := \sum_{\kk=0}^\infty
\lambda(\kk/N)\PP_\kk(u), \quad u\in[-1, 1],
\end{equation}
is simply a polynomial kernel.
The localization of this kernel is given in the following

\begin{thm}\label{thm:localization}
Let $\nu\ge 0$ and $M\in\NN$.
Assume $\lambda\in C^\infty[0, \infty)$,
$\|\lambda^{(m)}\|_\infty \le \kappa$ for $0\le m\le M$ and for some $b>1$
either $\supp \lambda \subset [(2b)^{-1}, 2b]$ or $\supp \lambda \subset [0, b]$ and $\lambda(t)=1$ for $t\in [0, 1]$.
Then there exists a constant $c>0$ depending only on $M$, $\nu$, b, and $d$ such that
for any $N\ge 1$ the kernel $\Lambda_N$ from $(\ref{def-Lam})-(\ref{def-Lam-2})$
obeys
\begin{equation}\label{local-Lam}
|\Lambda_N^{(\nu)}(\cos \theta)| \le c \kappa \frac{N^{d-1+2\nu}}{(1+N|\theta|)^{M}},
\quad |\theta|\le \pi,
\end{equation}
and hence
\begin{equation}\label{local-Lam-2}
|\Lambda_N^{(\nu)}(x\cdot y)| \le c \kappa \frac{N^{d-1+2\nu}}{(1+N\rho(x, y))^{M}},
\quad x, y\in\SS.
\end{equation}
Furthermore, for $x, y, z\in \SS$
\begin{equation}\label{local-Lip-1}
|\Lambda_N(x\cdot z)-\Lambda_N(y\cdot z)| \le c\kappa \frac{\rho(x, y) N^d}{(1+N\rho(x, z))^{M}},
\quad\hbox{if}\quad  \rho(x, y) \le N^{-1}.
\end{equation}
\end{thm}

For a proof of this theorem in the case $b=2$, see  \cite[Theorem 3.5]{NPW1} and \cite[Lemmas 2.4, 2.6]{NPW2}, also \cite{IPX}.
The proof for a general $b>1$ is similar.

\subsection{Maximal $\delta$-nets, cubature formulas on the sphere and nested sets}\label{s5_3}

For discretization of integrals and construction of frames on $\SS$ we shall need cubature formulas,
which are naturally constructed using maximal $\delta$-nets on $\SS$.

\smallskip

\noindent
{\bf Definition.}
Given $\delta>0$ we say that a finite set $\cZ\subset\SS$ is a maximal $\delta$-net on $\SS$
if
\begin{equation*}
\hbox{(i) $\rho(\zeta_1,\zeta_2)\ge\delta$ for all $\zeta_1,\zeta_2\in\cZ$, $\zeta_1\ne\zeta_2$, \quad and \quad
(ii) $\cup_{\zeta\in\cZ}B(\zeta,\delta)=\SS$.}
\end{equation*}

Clearly, a maximal $\delta$-net $\cZ$ on $\SS$ exists for any $\delta>0$ and $\# \cZ\sim \delta^{-(d-1)}$ for $0<\delta\le\pi$.
For every maximal $\delta$-net $\cZ$
it is easy to construct (see \cite[Proposition 2.5]{CKP}) a corresponding disjoint partition $\{\nA_\zeta\}_{\zeta\in\cZ}$ of $\SS$
consisting of measurable sets such that
\begin{equation}\label{disjoint}
B(\zeta,\delta/2)\subset \nA_\zeta\subset B(\zeta,\delta),\quad \zeta\in\cZ.
\end{equation}
Then the simple cubature formula $\int_\SS f(x)d\sigma(x)\approx \sum_{\zeta\in\cZ} |A_\zeta| f(\zeta)$
will be exact for all constant functions $f$ and $|A_\zeta|\sim \delta^{d-1}$.

\smallskip

\noindent
{\bf Cubature formula on $\SS$.}
In \cite[Theorem 4.3]{NPW1} it is shown that there exists a constant $\bar{\gamma}$ ($0<\bar{\gamma} <1$), depending only on $d$,
such that for any $0<\gamma\le\bar{\gamma}$, $0<\delta\le 1$, and a maximal $\gamma\delta$-net $\cZ\subset\SS$
there exist weights $\{\ww_\xi\}_{\xi\in\cZ}$, satisfying
\begin{equation}\label{cubature_w}
c^{-1} (\gamma\delta)^{d-1}\le \ww_\xi\le c (\gamma\delta)^{d-1}, \quad \xi\in\cZ,
\end{equation}
with constant $c>0$ depending only on $d$, such that the cubature formula
\begin{equation}\label{cubature}
\int_\SS f(x)d\sigma(x)\approx \sum_{\xi\in\cZ} \ww_\xi f(\xi)
\end{equation}
is exact for all spherical harmonics $f$ of degree $\le \delta^{-1}$.
The number of nodes in $\cZ$ is $\le c(d)(\gamma\delta)^{-(d-1)}$.

\smallskip

\noindent
{\bf Nested structure on $\SS$.}
As will be seen
any selection of a disjoint partition $\{\nA_\zeta\}$ with property \eqref{disjoint}  works well for
the construction of frames from shifts of the Newtonian kernel in the Triebel-Lizorkin spaces $\cF^{sq}_p(\SS)$ with $p<\infty$.
The construction of frames in the Triebel-Lizorkin spaces $\cF^{sq}_p(\SS)$ with $p=\infty$ will in addition require
the family of partitions of $\SS$ to form a \emph{nested structure} similar to the system of dyadic cubes in $\RR^d$.
It is given in the following

\begin{thm}\label{nested_sets}
Let $b>3$, $\betaw>0$ satisfy
\begin{equation}\label{dyadic_condition}
\frac{1}{b-1}+2\betaw\le\frac{1}{2}.
\end{equation}
Let $d\ge 2$, $\gamma>0$, and assume that $\{\cX_n\}_{n\ge 1}$ is a sequence of maximal $\delta_n$-nets on $\SS$ with $\delta_n:=\gamma b^{-n}$.
Then for every $n\in\NN$ there exists a collection of open subsets $\{Q_\xi\subset\SS : \xi\in\cX_n\}$ with the following properties:
\begin{align}
&\big|\SS\backslash\bigcup_{\xi\in\cX_n}Q_\xi\big|=0,\quad\forall n\in\NN;\label{dyadic_eq:1}\\
&\mbox{If}~\eta\in\cX_m, \xi\in\cX_n,~\mbox{and}~m\ge n,~\mbox{then either}~Q_\eta\subset Q_\xi~\mbox{or}~Q_\eta\cap Q_\xi=\emptyset;\label{dyadic_eq:2}\\
&\mbox{If}~\xi\in\cX_n~\hbox{and}~m< n,~\mbox{then there is a unique}~\eta\in\cX_m~\mbox{such that}~Q_\xi\subset Q_\eta;\label{dyadic_eq:3}\\
&B(\xi,\betaw\gamma b^{-n})\subset Q_\xi\subset B\Big(\xi,\frac{b}{b-1} \gamma b^{-n}\Big)~
\mbox{for all}~\xi\in\cX_n, n\in\NN.\label{dyadic_eq:4}
\end{align}
\end{thm}

The proof of \thmref{nested_sets} follows the proof of \cite[Theorem 11, (3.1)--(3.5)]{Ch}
with the trivial modification required by replacement of the maximal $b^{-n}$-net with a maximal $\gamma b^{-n}$-net.
Also, \cite[Theorem 11]{Ch} is stated for sufficiently small $b^{-1}$ and $\betaw$,
while condition \eqref{dyadic_condition} is derived from its proof.
This condition should be strengthened if one would require
an additional property of ``small boundary'' for the sets $Q_\xi$ (see \cite[Theorem 11, (3.6)]{Ch}).
We shall not need the ``small boundary'' property here.
Note that \cite[Theorem 11]{Ch} is stated for general spaces of homogeneous type
of which the sphere $\SS$ equipped with the geodesic distance and Lebesgue measure is a particular case.

\begin{proof}[Sketch of proof.]
The proof is carried out by completing the following steps.

\smallskip

$(1)$
Partial ordering.
For every $n\in\NN$ the set $\cX_n$ is a maximal $\delta$-net, $\delta:=\gamma b^{-n}$, on $\SS$, 
which implies $\cup_{\xi\in\cX_n}B(\xi,\gamma b^{-n})=\SS$ and the sets $\{B(\xi,\gamma b^{-n}/2) : \xi\in\cX_n\}$ are mutually disjoint.
A partial ordering of the points in $\cup_{n\ge 1}\cX_n$ is constructed as follows:
\begin{enumerate}
\renewcommand{\labelenumi}{(\alph{enumi})}
\item If $\eta\in\cX_{n+1}$ and $\rho(\eta,\xi)<\gamma b^{-n}/2$ for some $\xi\in\cX_n$ then we set $\eta\prec\xi$. In this cases $\xi$ is unique.
\item If $\eta\in\cX_{n+1}$, $\eta\notin\cup_{\zeta\in\cX_n}B(\zeta,\gamma b^{-n}/2)$, and $\rho(\eta,\xi)<\gamma b^{-n}$ for some $\xi\in\cX_n$
then we set $\eta\prec\xi$. The choice of $\xi$ in this cases may not be unique.
\item Finally, $\prec$ is extended by transitivity.
\end{enumerate}

\smallskip

$(2)$
The open sets $Q_\xi$, $\xi\in\cX$, are defined by
\begin{equation*}
Q_\xi:=\cup_{\eta\prec\xi}B(\eta,\betaw\gamma /N_\eta),
\end{equation*}
where $\betaw$ is from \eqref{dyadic_condition} and $N_\eta=b^{k}$ for $\eta\in\cX_k$.

\smallskip

$(3)$
If $\eta\prec\xi$ then $\rho(\xi,\eta)\le\frac{b}{b-1}\gamma /N_\xi$.

\smallskip

$(4)$
If $\xi,\eta\in\cX_n$ and $Q_\xi\cap Q_\eta\ne\emptyset$ then $\xi=\eta$. Condition \eqref{dyadic_condition} is used in this step.
\end{proof}

\begin{rem}\label{rem:1}
Condition \eqref{dyadic_condition} is satisfied by $b=4$ and $\betaw=1/12$ -- a choice we shall employ in the paper.
Note that neither \eqref{dyadic_condition} nor the conclusions in \eqref{dyadic_eq:1}--\eqref{dyadic_eq:3} depends on $\gamma$
$($after the choosing of the maximal $\delta_n$-nets $\cX_n$$)$.
\end{rem}

\begin{rem}\label{rem:2}
A partial ordering in $\cX$ is implied by \thmref{nested_sets}, namely,
\begin{equation}\label{partial_order}
\eta\prec\xi\Leftrightarrow Q_\eta\subset Q_\xi.
\end{equation}
Given the collection of nodes $\cX=\cup_{n=1}^{\infty}\cX_n$ the partial ordering is not unique for a fixed $\betaw$.
But all possible partial orderings obey \thmref{nested_sets} and \propref{prop:nested_sets} bellow.
\end{rem}

\begin{rem}\label{rem:3}
For the rest of the paper we fix the numbers $b, \betaw$ satisfying \eqref{dyadic_condition} and $\bar{\gamma}$
from the nontrivial cubature formula \eqref{cubature}. Then we choose
\begin{equation}\label{index_set}
\begin{split}
\cX_j ~\text{to be a maximal}~ \delta_j\text{-net on}~ \SS,~ j\in\NN,\; \delta_j:=\bar{\gamma}2^{-1} b^{-j},
\\
\cX_0=\{e_1\},~e_1=(1,0,\dots,0);\quad \cX=\cup_{j=0}^\infty\cX_j.
\end{split}
\end{equation}
For the set $\cX$ we assume that equal points from different sets $\cX_j$
are distinct points in $\cX$ so that $\cX$ can be used as an \emph{index set}.

The partial ordering is trivially extended from $\cup_{j=1}^\infty\cX_j$ to $\cX$ by setting
$Q_{e_1}=\SS$ and, hence, $\xi\prec e_1$ for all $\xi\in\cX_1$.

Note that the choice of $\bar{\gamma}$ allows every $\cX_j$, $j\in\NN$, to be a nodal set
for a cubature formula on $\SS$ that is exact for the spherical polynomials of degree $\le 2b^j$.
Similarly, $\cX_0$ will be a nodal set
for a cubature formula on $\SS$ exact for the spherical polynomials of degree $0$
if we set $\ww_{e_1}:=\omega_d$.
\end{rem}

\begin{prop}\label{prop:nested_sets}
Under the assumption of \thmref{nested_sets} we have:

$(1)$
For any $m\in\NN$ we have
\begin{equation}\label{number_of_children}
\cN_m := \max_{n\in \NN,~\xi\in\cX_n} \#\{Q_\eta\subset Q_\xi : \eta\in Q_{n+m}\}
\le\left(\frac{b}{\betaw(b-1)}\right)^{d-1} b^{m(d-1)}.
\end{equation}

\smallskip

$(2)$
If $b\ge 4$ then every node has at least two children.

\smallskip

$(3)$
For any $\xi\in\cX_n$ and $k\in\NN$
\begin{equation}\label{dyadic_eq:5}
|Q_\xi|=\sum_{\eta\in\cX_{n+k},~Q_\eta\subset Q_\xi}|Q_\eta|.
\end{equation}

$(4)$
For any $\xi\in\cX_n$ and $n\in\NN_0$
\begin{equation}\label{dyadic_eq:6}
|Q_\xi|\sim b^{-n(d-1)}
\end{equation}
with constants of equivalence depending only on $d, b, \betaw, \gamma$.

\end{prop}

\begin{proof}
Inequality \eqref{number_of_children} follows from \eqref{dyadic_eq:2}, \eqref{dyadic_eq:4}, and \eqref{sph_cap2} because
\begin{equation*}
\cN_m \le \Big|B\Big(\xi,\frac{b}{b-1} \gamma b^{-n}\Big)\Big|\Big/|B(\eta,\betaw \gamma b^{-n-m})|.
\end{equation*}

For the proof of part (2) we fix $\xi\in\cX_n$ and set $\delta:=\gamma b^{-n}$.
Then $\cX_{n}$ is a maximal $\delta$-net and $\cX_{n+1}$ is a maximal $\delta/b$-net.
In view of Step (1) from the proof of \thmref{nested_sets} every $\eta\in\cX_{n+1}$ satisfying $\eta\in B(\xi,\delta/2)$ is by definition a child of $\xi$.
If $B(\eta,\delta/b)\cap\overline{B(\xi,\delta/4)}\ne\emptyset$ for some $\eta\in\SS$, then $\eta\in B(\xi,\delta/2)$.
Finally, observing that the set $\overline{B(\xi,\delta/4)}$ cannot be covered by a single ball of the type $B(y,\delta/b)$
and $\cup_{\eta\in\cX_{n+1}}B(\eta,\delta/b)=\SS$
we conclude that $\xi$ has at least two children.

Identity \eqref{dyadic_eq:5} follows from \eqref{dyadic_eq:1}--\eqref{dyadic_eq:3},
while \eqref{dyadic_eq:6} follows from \eqref{dyadic_eq:4} and \eqref{sph_cap}.
\end{proof}

For future references we state the following simple lemma with the notation from above.
It is proved in \cite[Lemma 9.2]{IP2} for $b=2$.
The proof for a general $b>1$ is the same.
\begin{lem}\label{lem:Besov_1}
Let $j,m\ge 0$, $0<\kappa\le 1$. Then
\begin{equation}\label{eq:besov_0}
\sum_{\eta\in\cX_{j+m}} \frac{1}{(1+b^j \rho(x,\eta))^{d-1+\kappa}}\le c b^{m(d-1)},
\quad x\in\SS,
\end{equation}
with $c=c(d,b)\kappa^{-1}$.
\end{lem}

\section{Spaces of functions and distributions on the sphere}\label{s2}

Besov and Triebel-Lizorkin spaces on the sphere are developed in \cite{NPW2}
using the approach in the classical case on $\RR^d$ from \cite{Peetre, Triebel-1, FJ1, FJ2, FJW}.
Here we recall the basics of this theory.
In particular, we recall the frame characterization of these spaces and
the construction of frames in terms of shifts of Newtonian kernels from \cite{IP2}.

\subsection{Besov and Triebel-Lizorkin spaces on $\SS$}\label{s2_2}

The Besov and Triebel-Lizorkin spaces on $\SS$ in general are spaces of distributions.
The class of test functions on $\SS$ is defined by $\cS:= C^\infty(\SS)$.
To be more specific recall the definition of $C^\infty(\SS)$.
Given a function $\phi$ on $\SS$ we denote by $E\phi$ its extension to $\RR^d\setminus\{0\}$ 
defined by $E\phi(x):=\phi(x/|x|)$. 
Then $\partial^\beta \phi := \partial^\beta(E\phi)|_{\SS}$
and $C^\infty(\SS)$ is the set of all functions $\phi$ on $\SS$ such that
$\|\partial^\beta \phi\|_{L^\infty(\SS)}<\infty$ for all $\beta\in\NN_0^d$.  
The topology on $\cS$ is standardly defined by the semi-norms
\begin{equation*}
\tilde{P}_m(\phi) := \sum_{|\beta|=m} \|\partial^\beta \phi\|_{L^\infty(\SS)}, \quad m=0,1, \dots.
\end{equation*}
As is well known the topology on $\cS$ can be equivalently defined by the semi-norms
\begin{equation*}
\tilde{\tilde{P}}_m(\phi) := \|\Delta_0^m \phi\|_{L^\infty}, \quad m=0,1, \dots,
\end{equation*}
where $\Delta_0$ is the Laplace-Beltrami operator on $\SS$.
In turn this readily implies that a function $\phi\in\cS$ if and only if
\begin{equation*}
\|\PP_\kk*\phi\|_2 \le c(\phi, m)(1+\kk)^{-m}, \quad \forall \kk, m\ge 0.
\end{equation*}
Recall that the convolution $\PP_\kk*\phi$ is defined in \eqref{def-conv}.
Therefore, the topology on $\cS$ is also equivalently defined by the norms
\begin{equation}\label{test_norms}
P_m(\phi):= \sum_{\kk=0}^\infty (\kk+1)^m \|\PP_\kk*\phi\|_2
= \sum_{\kk=0}^\infty (\kk+1)^m \Big(\sum_{\nu=1}^{N(\kk, d)} |\langle \phi, Y_{\kk \nu}\rangle|^2\Big)^{1/2}.
\end{equation}
Note that $\cS$ is complete in this topology.

Observe also that all $Y_{\kk \nu}\in\cS$ and hence by (\ref{Pk}) $\PP_\kk(x\cdot y)\in \cS$
as a function of $x$ for every fixed $y$ and as function of $y$ for every fixed $x$.

The space $\cS':=\cS'(\SS)$ of distributions on $\SS$ is defined as the space of
all continuous linear functionals on $\cS$.
The pairing of $f\in \cS'$ and $\phi\in\cS$ will be denoted by
$\langle f, \phi\rangle := f(\overline{\phi})$, which is consistent with the inner product
on $L^2(\SS)$.
More precisely, $\cS'$ consists of all linear functionals $f$ on $\cS$ for which
there exist constants $c>0$ and $m\in\NN_0$ such that
\begin{equation}\label{def-distr}
|\langle f, \phi\rangle| \le c P_m(\phi),\quad \forall \phi\in\cS.
\end{equation}
For any $f\in\cS'$ we define $\PP_\kk*f$ by
\begin{equation}\label{def-P-f}
\PP_\kk*f(x) :=\langle f, \overline{\PP_\kk(x\cdot\bullet)}\rangle
=\langle f, \PP_\kk(x\cdot\bullet)\rangle,
\end{equation}
where on the right $f$ is acting on $\overline{\PP_\kk(x\cdot y)}=\PP_\kk(x\cdot y)$ as a function of $y$
($\PP_k$ is real-valued).

Observe that the representation
\begin{equation}\label{represent-f}
f=\sum_{\kk=0}^\infty \PP_\kk*f, \quad \forall f\in\cS'
\end{equation}
holds with convergence in distributional sense.

\begin{defn}\label{def:B-F-spaces}
Let $s\in\RR$, $0<q\le \infty$, $b>1$, and
$\varphi$ satisfy the conditions:
\begin{equation}\label{adm}
\varphi \in C^\infty(\RR_+),~~
\supp \varphi \subset [b^{-1},b],~~
|\varphi(u)|\ge c>0~\mbox{for}~u\in [b^{-2/3},b^{2/3}].
\end{equation}
For a distribution $f\in \cS'$ set
\begin{equation}\label{rep-Phi-j-f}
\Phi_0*f := \PP_0*f,\quad
\Phi_j*f := \sum_{\kk=0}^\infty \varphi\Big(\frac{\kk}{b^{j-1}}\Big)\PP_\kk*f, ~ j\ge 1,
\end{equation}
where $\PP_k*f$ is defined in \eqref{def-P-f}.

$(a)$
The Besov space $\cB^{sq}_p:=\cB^{sq}_p(\SS)$, $0<p\le \infty$, is defined as the set of all distributions
$f\in \cS'$ such that
\begin{equation}\label{B-norm}
\|f\|_{\cB^{s q}_p} :=
\Big(\sum_{j=0}^\infty \Big(b^{sj}\|\Phi_j*f\|_{L^p(\SS)}\Big)^q\Big)^{1/q}
< \infty,
\end{equation}
where the $\ell^q$-norm is replaced by the sup-norm if $q=\infty$.

$(b)$
The Triebel-Lizorkin space $\cF^{sq}_p:=\cF^{sq}_p(\SS)$, $0<p<\infty$, is defined as the set of all distributions
$f\in \cS'$ such that
\begin{equation}\label{F-norm}
\|f\|_{\cF^{s q}_p} :=
\Big\|\Big(\sum_{j=0}^\infty \big(b^{sj}|\Phi_j*f(\cdot)|\big)^q\Big)^{1/q}\Big\|_{L^p(\SS)}
< \infty,
\end{equation}
where the $\ell^q$-norm is replaced by the sup-norm if $q=\infty$.
\end{defn}

Note that the definitions of the Besov and Triebel-Lizorkin spaces above are independent of the choice of $b>1$
and of the particular selection of the function $\varphi$ with the required properties, that is,
different $b$'s and $\varphi$'s produce equivalent quasi-norms (see \propref{prop:independ2}).

\subsection{Needlet frame decomposition of spaces of distributions on $\SS$}\label{subsec:frame-SS}

We next recall the construction of the frame (needlets) on $\SS$
from \cite{NPW2}.
Note that in dimension $d=2$ the Meyer's periodic wavelets (see \cite{Meyer})
form a basis with the desired properties.

Let $b>1$ be fixed.
The first step in the construction of needlets on $\SS$, $d\ge 2$, is the selection of
a real-valued function $\aa\in C^\infty(\RR_+)$ with the properties:
$\supp \aa\subset [b^{-1},b]$, $0\le \aa \le 1$, $\aa(u)\ge c>0$ for $u\in[b^{-2/3},b^{2/3}]$,
$\aa^2(u)+\aa^2(u/b)=1$ for $u\in [1, b]$,
and hence
$\sum_{\nu=0}^\infty \aa^2(b^{-\nu}u) =1$ for $u\in [1, \infty)$.
Set
\begin{equation}\label{def-Psi-j}
\LL_0:=Z_0
\quad\hbox{and}\quad
\LL_j:=\sum_{\kk=0}^\infty \aa\Big(\frac{\kk}{b^{j-1}}\Big)Z_\kk, \quad j\ge 1.
\end{equation}
It is easy to see that
$f=\sum_{j=0}^\infty \LL_j*\LL_j*f$ for every $f\in\cS'$ (convergence in $\cS'$).

The next step is to discretize $\LL_j*\LL_j$ for $j\ge 1$ by using the cubature formula on $\SS$
from \eqref{cubature} which is exact for all spherical harmonics of degree $\le 2b^{j}$,
 where $\cZ=\cX_j$ is a maximal $\delta_j$-net with $\delta_j:=\bar{\gamma}b^{-j}/2$ (see \eqref{index_set}).

Since the cubature formula \eqref{cubature} is exact for spherical harmonics of degree $\le 2b^{j}$, $j\in\NN$, (or of degree $0$ for $j=0$)
we have
$$
\LL_j*\LL_j(x\cdot y) =\int_\SS \LL_j(x\cdot\eta)\LL_j(\eta\cdot y)d\sigma(\eta)
= \sum_{\xi\in\cX_j} \ww_\xi \LL_j(x\cdot \xi)\LL_j(\xi\cdot y),
$$
which allows to discretize $f=\sum_{j=0}^\infty \LL_j*\LL_j*f$ and obtain
\begin{equation}\label{discretize}
f=\sum_{j=0}^\infty \sum_{\xi\in\cX_j}\langle f,\psi_\xi \rangle \psi_\xi,
\quad \forall f\in\cS' \quad\hbox{(convergence in $\cS'$)},
\end{equation}
\begin{equation}\label{def-psi-xi}
\psi_\xi(x):=\ww_\xi^{1/2} \LL_j(\xi\cdot x), \quad \xi\in\cX_j, \;j\ge 0.
\end{equation}

This completes the construction of the \emph{needlet system}
$\Psi=\{\psi_\xi\}_{\xi\in\cX}$.

Observe that the functions $\{\psi_\xi\}$ are not only band limited,
but also have an excellent space localization on $\SS$.
From the properties of $\aa$, Theorem~\ref{thm:localization}, and \eqref{cubature_w}
it follows that (see also \cite{NPW1, NPW2}) for any $M>0$
\begin{equation}\label{local-needlet-0}
|\psi_\xi(x)| \le c b^{j(d-1)/2}(1+b^j\rho(x, \xi))^{-M},
\quad x\in\SS,\; \xi\in\cX_j, \; \;j\ge 0,
\end{equation}
where $c>0$ is a constant depending only on $d$, $M$, $b$, and $\aa$.
Moreover, the localization of $\psi_\xi$ can be improved to sub-exponential
as shown in \cite[Theorem~5.1]{IPX}.
The normalization factor $\ww_\xi^{1/2}$ in \eqref{def-psi-xi} makes all $\psi_\xi$ essentially normalized in $L^2(\SS)$,
i.e. $\|\psi_\xi\|_{L^2(\SS)}\sim 1$.

\smallskip

We next define the Besov and Triebel-Lizorkin sequence spaces $\fb_p^{sq}$ and $\ff_p^{sq}$ associated to the $\delta_j$-nets  $\cX_j$ from $\cX$.

\begin{defn}\label{def:B}
Let $b>1$, $s\in \RR$, $0<p,q\le\infty$. Then $\fb_p^{sq}:=\fb_p^{sq}(\cX)$
is defined as the space of all sequences of complex numbers
$h:=\{h_{\xi}\}_{\xi\in \cX}$ such that
\begin{equation}\label{def-b-space}
\|h\|_{\fb_p^{sq}} :=
\Big(\sum_{j=0}^\infty
\Big[b^{j[s+(d-1)(1/2-1/p)]}
\Big(\sum_{\xi \in \cX_j}
|h_\xi|^p\Big)^{1/p}\Big]^q\Big)^{1/q}<\infty
\end{equation}
with the usual modification when $p=\infty$ or  $q=\infty$.
\end{defn}

\begin{defn}\label{def:TL}
Let $b>1$, $s\in \RR$, $0<p<\infty$, and $0<q\le\infty$. Then $\ff_p^{sq}:=\ff_p^{sq}(\cX)$
is defined as the space of all sequences of complex numbers
$h:=\{h_{\xi}\}_{\xi\in \cX}$ such that
\begin{equation}\label{def-f-space}
\|h\|_{\ff_p^{sq}} :=\Big\|\Big(\sum_{\xi\in\cX}
\big[|B_\xi|^{-s/(d-1)-1/2}
|h_{\xi}|\ONE_{B_\xi}(\cdot)\big]^q\Big)^{1/q}\Big\|_{L^p} <\infty
\end{equation}
with the usual modification for $q=\infty$.
Here $B_\xi:=B(\xi,\bar{\gamma} b^{-j+1})$, $\xi\in \cX_j$, where $\bar{\gamma}$ is used in the selection of $\cX_j$,
$|B_\xi|$ is the measure of $B_\xi$
and
$\ONE_{B_\xi}$ is the characteristic function of $B_\xi$.
\end{defn}

The main result in this subsection asserts that $\{\psi_\xi\}_{\xi\in\cX}$ is a self-dual real-valued frame for
the Besov and Triebel-Lizorkin spaces on the sphere
with parameters $(s, p, q)$ in the range
\begin{equation}\label{indices-1}
\QQ=\QQ(A):=\big\{(s, p, q): |s| \le A,\; A^{-1}\le p \le A, \;\hbox{and}\; A^{-1}\le q<\infty\big\},
\end{equation}
for a given constant $A>1$.
To state this result we introduce the following
{\em analysis} and {\em synthesis} operators:
\begin{equation}\label{def-oper-S-T}
S_{\psi}: f\mapsto \{\langle f, \psi_\xi\rangle\}_{\xi\in\cX},
\quad
T_\psi: \{h_\xi\}_{\xi\in\cX} \mapsto \sum_{\xi\in\cX}h_\xi\psi_\xi.
\end{equation}

\begin{thm}\label{thm:F-Bnorm-equiv}
Let $b>1$ and $A>1$.
\smallskip

\noindent
$(a)$
The operators
$S_{\psi}: \cB_p^{s q} \to \fb_p^{s q}$ and
$T_\psi: \fb_p^{s q} \to \cB_p^{s q}$ are uniformly bounded for $(s, p, q)\in\QQ(A)$, and
$T_\psi\circ S_{\psi}= I$ on $\cB_p^{s q}$.
Hence,
if $f\in \cS'$, then $f\in \cB_p^{sq}$ if and only if
$\{\langle f,\psi_\xi\rangle\}_{\xi \in \cX}\in \fb_p^{sq}$,
and
\begin{equation}\label{B-disc-calderon}
f =\sum_{\xi\in \cX}\langle f, \psi_\xi\rangle \psi_\xi
\quad\mbox{and}\quad
\|f\|_{\cB_p^{sq}}
\sim  \|\{\langle f,\psi_\xi\rangle\}\|_{\fb_p^{sq}}.
\end{equation}

\smallskip

\noindent
$(b)$
The operators
$S_{\psi}: \cF_p^{s q} \to \ff_p^{s q}$ and
$T_\psi: \ff_p^{s q} \to \cF_p^{s q}$ are  uniformly bounded for $(s, p, q)\in\QQ(A)$, and
$T_\psi\circ S_{\psi}= I$ on $\cF_p^{s q}$.
Hence,
if $f\in \cS'$, then $f\in \cF_p^{sq}$ if and only if
$\{\langle f,\psi_\xi\rangle\}_{\xi \in \cX}\in \ff_p^{sq}$, and
\begin{equation}\label{F-disc-calderon}
f =\sum_{\xi\in \cX}\langle f, \psi_\xi\rangle \psi_\xi
\quad \mbox{and}\quad
\|f\|_{\cF_p^{sq}}
\sim  \|\{\langle f,\psi_\xi\rangle\}\|_{\ff_p^{sq}}.
\end{equation}
The convergence in $(\ref{B-disc-calderon})$ and $(\ref{F-disc-calderon})$ is unconditional
in $\cB_p^{sq}$ and $\cF_p^{sq}$, respectively.
The constants of equivalence in $(\ref{B-disc-calderon})$ and $(\ref{F-disc-calderon})$ depend only on $A$, $b$, $d$, and $\varphi$.
\end{thm}
For details and proofs in the case $b=2$, see \cite[Theorems 4.5 and 5.5]{NPW2} and \cite[Remark 3.13]{IP2}.
The proof for general $b>1$ is similar.

Some embeddings between Besov or Triebel-Lizorkin spaces are given in \cite[Proposition 3.15]{IP2}.
Here we state only those that will be needed later on.

\begin{prop}\label{embed}
Assume $s,s_0,s_1\in\RR$ and let $0<p,p_0,p_1,q,q_0,q_1\le\infty$ in the case of Besov spaces
and $0<p,p_0,p_1<\infty$, $0<q,q_1,q_2\le\infty$ in the case of Triebel-Lizorkin spaces.
The following continuous embeddings are valid:
\begin{equation}\label{eq:2}
\cB^{s_0q_0}_p\subset \cB^{s_1q_1}_p,~~\cF^{s_0q_0}_p\subset \cF^{s_1q_1}_p,
\quad \hbox{if}\;\;
s_0=s_1,~q_0\le q_1~~\mbox{or}~~s_0>s_1,~\forall q_0, q_1;
\end{equation}
\begin{equation}\label{eq:8b}
\cB^{sq}_{p}\subset \cF^{sq}_{p}\subset \cF^{sp}_{p}=\cB^{sp}_{p},
\quad  \hbox{if}\;\;
q<p;
\end{equation}
\begin{equation}\label{eq:8c}
\cB^{sp}_{p}=\cF^{sp}_{p}\subset \cF^{sq}_{p}\subset \cB^{sq}_{p},
\quad  \hbox{if}\;\;
p<q.
\end{equation}
\end{prop}

\subsection{Frames in terms of shifts of the Newtonian kernel on $\SS$}

In \cite{IP2} we constructed a system $\Theta=\{\theta_\xi\}_{\xi\in\cX}$
where each element $\theta_\xi$ is \emph{a linear combination of a~fixed number of shifts of the Newtonian kernel};
$\Theta$ and its dual system $\tilde{\Theta}$ are frames for
the Besov spaces $\cB_p^{sq}$ and the Triebel-Lizorkin spaces $\cF_p^{sq}$ with $p<\infty$.
With $\QQ(A)$ defined in \eqref{indices-1} we have

\begin{thm}\label{thm:prop-frame}
Assume $d\ge 2$, $b>1$, $A>1$, and
let $\Theta=\{\theta_\xi\}_{\xi\in\cX}$ be the real-valued system constructed in \cite[(6.36)--(6.37)]{IP2} with parameters
$K \ge \left\lceil Ad\right\rceil$, $K\in 2\NN$, and $M = K+d$.
If the constant $\gamma_0$ in the construction of $\{\theta_\xi\}_{\xi\in\cX}$ is sufficiently small, then:

$(a)$ The synthesis operator $T_\theta$ defined by $T_\theta h:= \sum_{\xi\in \cX}h_\xi\theta_\xi$
on sequences of complex numbers $h=\{h_\xi\}_{\xi\in\cX}$ is bounded as a map
$T_\theta: \fb_p^{sq} \mapsto \cB_p^{sq}$, uniformly with respect to to $(s, p, q)\in\QQ(A)$.

$(b)$ The operator
\begin{equation*}
Tf:=\sum_{\xi\in\cX} \langle f, \psi_\xi\rangle \theta_\xi=T_{\theta}S_{\psi}f,
\end{equation*}
is invertible on $\cB_p^{sq}$ and $T$, $T^{-1}$ are bounded on $\cB_p^{sq}$,
uniformly with respect to $(s, p, q)\in\QQ(A)$.

$(c)$ For $(s, p, q)\in \QQ(A)$ the dual system $\tilde\Theta=\{\tilde\theta_\xi\}_{\xi\in\cX}$ consists of bounded linear functionals
on $\cB_p^{sq}$ defined by
\begin{equation*}
\tilde\theta_\xi(f)=
\langle f, \tilde\theta_\xi\rangle
:= \sum_{\eta\in\cX}
\langle T^{-1}\psi_\eta, \psi_\xi\rangle \langle f, \psi_\eta\rangle
\quad \hbox{for}\;\; f\in \cB_p^{sq},
\end{equation*}
with the series converging absolutely.
Also, the analysis operator
$$
S_{\tilde\theta}: \cB_p^{sq} \mapsto \fb_p^{sq},\;
S_{\tilde\theta} = S_{\psi}T^{-1}T_{\psi}S_{\psi}=S_{\psi}T^{-1},
$$
is uniformly bounded with respect to $(s, p, q)\in\QQ(A)$ and
$T_\theta\circ S_{\tilde{\theta}}= I$ on $\cB_p^{s q}$.
Moreover, $\{\theta_\xi\}_{\xi\in\cX}$, $\{\tilde\theta_\xi\}_{\xi\in\cX}$
form a pair of dual frames for $\cB_p^{sq}$ in the following sense:
For any $f\in \cB_p^{sq}$
\begin{equation*}
f=\sum_{\xi\in\cX} \langle f, \tilde\theta_\xi\rangle \theta_\xi
\quad\hbox{and}\quad
\|f\|_{\cB_p^{sq}} \sim \|\{\langle f, \tilde\theta_\xi\rangle\}\|_{\fb_p^{sq}},
\end{equation*}
where the convergence is unconditional in $\cB_p^{sq}$.

Furthermore, $(a)$, $(b)$, and $(c)$ hold true when
$\cB_p^{sq}$, $\fb_p^{sq}$ are replaced by $\cF_p^{sq}$, $\ff_p^{sq}$, respectively.
\end{thm}

\thmref{thm:prop-frame} is proved in the case $b=2$ in \cite[Theorem 6.9]{IP2}.
The proof for general $b>1$ is similar.

\section{The Triebel-Lizorkin spaces $\cF^{sq}_\infty$ on $\SS$}\label{sec:F-spaces}

As is well known (see Triebel's remark in \cite[p. 46]{Triebel-1}) 
if one replaces the $L^p$-norm by the $L^\infty$-norm in the definition of 
the Triebel-Lizorkin spaces $F^s_{pq}$ in the classical case on $\RR^d$ one does not obtain a satisfactory
definition of $F^s_{\infty q}$. 
A good definition of $F^s_{\infty q}$ in that case was made by Frazier and Jawerth in \cite[see (5.1)]{FJ2}.
In this section we are concerned with the Triebel-Lizorkin spaces $\cF^{sq}_\infty$ on $\SS$
and their separable subspaces $\sepTL{\cF}{s}{q}$.
In particular, we study their properties and construct two types of frames for them.

\subsection{Definition and basic properties of the spaces $\cF^{sq}_\infty$ on $\SS$}

\begin{defn}\label{def:F-infty}
Let $b>1$, $s\in\RR$ and $0<q\le\infty$.
Let the functions $\Phi_n$, $n\ge 0$, be as in Definition~\ref{def:B-F-spaces}.
The Triebel-Lizorkin space $\cF^{sq}_\infty:=\cF^{sq}_\infty(\SS)$ is defined as the set of all distributions
$f\in \cS'$ such that
\begin{equation}\label{F-norm-infty0}
\|f\|_{\cF^{s q}_\infty} := \sup_{\substack{y\in\SS\\0<r\le\pi}}
\bigg(\frac{1}{|B(y,r)|}\int_{B(y,r)}\sum_{n\ge j} \big(b^{sn}|\Phi_n*f(x)|\big)^q d\sigma(x)\bigg)^{1/q}
< \infty,
\end{equation}
if $q<\infty$, where $j=\left\lfloor \log_{b}(\pi/r)\right\rfloor$,
and
\begin{equation}\label{F-infty-infty}
\|f\|_{\cF^{s\infty}_\infty} := \sup_{n\ge 0} b^{sn}\|\Phi_n*f\|_{L^\infty(\SS)} < \infty.
\end{equation}
\end{defn}

\begin{defn}\label{def:VF-infty}
Let $b>1$, $s\in\RR$ and $0<q\le\infty$.
The \emph{Separable Triebel-Lizorkin space} $\sepTL{\cF}{s}{q}:=\sepTL{\cF}{s}{q}(\SS)$
is defined as the set of all distributions
$f\in \cF^{sq}_\infty(\SS)$ satisfying for $q<\infty$
\begin{equation}\label{VF-norm-infty0}
\lim_{\delta\to 0} \sup_{\substack{y\in\SS\\0<r\le\delta}}
\bigg(\frac{1}{|B(y,r)|}\int_{B(y,r)}\sum_{n\ge \left\lfloor \log_{b}(\pi/r)\right\rfloor} \big(b^{sn}|\Phi_n*f(x)|\big)^q d\sigma(x)\bigg)^{1/q}=0
\end{equation}
and for $q=\infty$
\begin{equation}\label{VF-infty-infty}
\lim_{n\to\infty} b^{sn}\|\Phi_n*f\|_{L^\infty(\SS)} =0.
\end{equation}
Also, $\sepTL{\cF}{s}{q}(\SS)$ inherits the norm from $\cF^{sq}_\infty(\SS)$.
\end{defn}

From the above definitions we immediately get the following properties.

\begin{prop}\label{prop:TL-infty}
Let $s\in\RR$, $0<q\le\infty$. Then

$(a)$
$\cF^{sq}_\infty(\SS)$ and $\sepTL{\cF}{s}{q}(\SS)$ are (quasi-)Banach spaces.

$(b)$
$\sepTL{\cF}{s}{q}(\SS)$ is the closure in the $\cF^{sq}_\infty(\SS)$ norm of $C^\infty(\SS)$.

$(c)$
$\sepTL{\cF}{s}{q}(\SS)$ is a \emph{separable space}, while $\cF^{sq}_\infty(\SS)$ is \emph{non-separable}.

\end{prop}

\begin{prop}\label{prop:independ}
The definition of the Triebel-Lizorkin spaces $\cF^{sq}_\infty$, $s\in\RR$, $0<q\le\infty$, $($Definition~\ref{def:F-infty}$)$
is independent of the particular selection of $b>1$ and $\varphi$ satisfying \eqref{adm}.
\end{prop}

\begin{proof}
Let $a>1$ and $\ups\in C^\infty(\RR_+)$ be such that $\supp \ups \subset [a^{-1}, a]$
and $|\ups(u)|\ge c>0$ for $u\in [a^{-2/3}, a^{2/3}]$.
For any $f\in \cS'$ we set
\begin{equation*}
\Ups_0*f:=\cZ_0*f,
\quad
\Ups_j*f:= \sum_{k\ge 0}\ups\Big(\frac{k}{a^{j-1}}\Big) \cZ_k*f,\quad j\ge 1.
\end{equation*}
Denote
\begin{multline}\label{F-norm-a}
\|f\|_{\cF^{s q}_\infty(a,\Ups)}\\
:= 
\sup_{\substack{y\in\SS\\0<r\le\pi}}
\bigg(\frac{1}{|B(y,r)|}\int_{B(y,r)}\sum_{\nu\ge \left\lfloor \log_{a}\pi/r\right\rfloor} \big(a^{s\nu}|\Ups_\nu*f(x)|\big)^q d\sigma(x)\bigg)^{1/q}.
\end{multline}
It is easy to see (e.g. \cite[Lemma 6.9]{FJW}) that there exists a function $\tilde\ups$
with the properties of $\ups$
such that
\begin{equation*}
\sum_{\nu\ge 1}\tilde\ups\Big(\frac{u}{a^{\nu-1}}\Big)\ups\Big(\frac{u}{a^{\nu-1}}\Big) = 1,
\quad u\ge 1.
\end{equation*}
Therefore,
\begin{equation}\label{rep-f-Ups}
f=\sum_{\nu\ge 0} \widetilde\Ups_\nu*\Ups_\nu *f,
\end{equation}
where
\begin{equation*}
\widetilde\Ups_0*f:=\cZ_0*f,
\quad
\widetilde\Ups_\nu*f:= \sum_{k\ge 0}\tilde\ups\Big(\frac{k}{a^{\nu-1}}\Big) \cZ_k*f,
\quad \nu\ge 1.
\end{equation*}

Let $\bb>0$ be such that $b=a^\bb$.
Let $\Phi_n$ be from the definition of Besov and Triebel-Lizorkin space (Definition~\ref{def:B-F-spaces}).
Then from \eqref{rep-f-Ups}
\begin{equation*}
\Phi_n*f=\sum_{\beta n-2\beta \le \nu\le \beta n+2}\Phi_n*\widetilde\Ups_\nu*\Ups_\nu*f
\end{equation*}
because $\Phi_n*\widetilde\Ups_\nu\ne 0$ only for $\beta (n-2)\le \nu\le \beta n+2$.
From Theorem~\ref{thm:localization} it follows that
\begin{equation*}
\Phi_n(x\cdot y) \le cb^{n(d-1)}\big(1+b^n\rho(x, y)\big)^{-M},
\quad
\widetilde\Ups_\nu(x\cdot y) \le ca^{\nu(d-1)}\big(1+a^\nu\rho(x, y)\big)^{-M}.
\end{equation*}
Just as in \cite[Proposition 2.5]{IP2}
these two inequalities imply
\begin{equation*}
\Phi_n*\widetilde\Ups_\nu(x\cdot y)
\le ca^{\nu(d-1)}\big(1+a^\nu\rho(x, y)\big)^{-M},
\quad \beta (n-2)\le \nu\le \beta n+2.
\end{equation*}
Here $M>0$ can be arbitrarily large. We choose $M=d/q_*$ with $q_*=\min\{q,1\}$.
Therefore,
\begin{equation}\label{est-conv}
|\Phi_n*f(x)| \le c\sum_{\beta (n-2)\le \nu\le \beta n+2}
\int_\SS\frac{a^{\nu(d-1)}|\Ups_\nu*f(y)|}{\big(1+a^\nu\rho(x, y)\big)^M} d\sigma(y).
\end{equation}

Fix $z\in\SS$ and $0<r\le \pi$, and denote
\begin{equation*}
\cN_{z,r}:=\bigg(\frac{1}{|B(z,r)|}\int_{B(z,r)}\sum_{n\ge j} \big(b^{sn}|\Phi_n*f(x)|\big)^q d\sigma(x)\bigg)^{1/q},
\end{equation*}
where $j=\left\lfloor \log_{b}\pi/r\right\rfloor$ and $q<\infty$.

Let $\cZ_\nu$, $\nu\in\ZZ$, be a maximal $\delta$-net  on $\SS$, $\delta=a^{-\nu}$,
and let $\{A_\zeta\}_{\zeta\in \cZ_\nu}$ be the sets from \eqref{disjoint} with $\delta=a^{-\nu}$.
For each $\eta\in\cZ_\nu$, we set $H_\eta:=\max_{y\in B(\eta,a^{-\nu})}|\Ups_\nu*f(y)|$
and
$$
h_\eta:=\max\{\min_{z\in B(\zeta,a^{-\nu-\kappa})}|\Ups_\nu*f(z)|
: \zeta\in\cZ_{\nu+\kappa}, B(\zeta,a^{-\nu-\kappa})\cap B(\eta,a^{-\nu})\ne\emptyset\},
$$
where $\kappa \in \NN$ is chosen as follows.
We invoke Lemma~4.7 from \cite{NPW2} to conclude that
for a sufficiently large $\kappa \in \NN$ we have for any $\xi\in\cX_\nu$
\begin{equation}\label{H-h}
\sum_{\eta\in\cZ_\nu}
\frac{H_\eta}{\big(1+a^\nu\rho(\xi, \eta)\big)^M}
\le c\sum_{\eta\in\cZ_\nu}
\frac{h_\eta}{\big(1+a^\nu\rho(\xi, \eta)\big)^M},
\end{equation}
where the constant $c$ is independent of $f$, $\xi$, and $\nu$.
In \eqref{H-h} we can choose $\kappa$ to depend only on $d$, $a$, and $M$.
Note that all $B(\zeta,a^{-\nu-\kappa})$ from the definition of $h_\eta$
satisfy $B(\zeta,a^{-\nu-\kappa})\subset B(\eta,2a^{-\nu})$.

We shall evaluate the integrals in the right-hand side of \eqref{est-conv}
for $x\in A_\xi$, $\xi\in\cZ_\ell$, $\ell=\left\lceil \beta n\right\rceil$,
and $\beta (n-2)\le \nu\le \beta n+2$, $n\ge j$.
Note that \eqref{disjoint} and \eqref{sph_cap} imply $|A_\eta|\sim a^{-\nu(d-1)}$ for $\eta\in\cZ_\nu$.
Using \eqref{H-h} we obtain
\begin{multline}\label{est-conv2}
\int_\SS\frac{a^{\nu(d-1)}|\Ups_\nu*f(y)|}{\big(1+a^\nu\rho(x, y)\big)^M} d\sigma(y)
=\sum_{\eta\in\cZ_\nu}
\int_{A_\eta}\frac{a^{\nu(d-1)}|\Ups_\nu*f(y)|}{\big(1+a^\nu\rho(x, y)\big)^M} d\sigma(y)\\
\le c\sum_{\eta\in\cZ_\nu}
\frac{H_\eta}{\big(1+a^\nu\rho(\xi, \eta)\big)^M}
\le c\sum_{\eta\in\cZ_\nu}
\frac{h_\eta}{\big(1+a^\nu\rho(\xi, \eta)\big)^M}.
\end{multline}
We claim that
\begin{equation}\label{est-conv21}
\bigg(\int_\SS\frac{a^{\nu(d-1)}|\Ups_\nu*f(y)|}{\big(1+a^\nu\rho(x, y)\big)^M} d\sigma(y)\bigg)^q
\le c\sum_{\eta\in\cZ_\nu}
\frac{h_\eta^q}{\big(1+a^\nu\rho(\xi, \eta)\big)^{Mq_*}}.
\end{equation}
Indeed, for $0<q\le 1$  \eqref{est-conv21} follows from \eqref{est-conv2} and the $q$-inequality.
To prove \eqref{est-conv21} for $1<q<\infty$ we apply  H\"{o}lder's inequality to obtain
\begin{equation*}
\bigg(\sum_{\eta\in\cZ_\nu} \frac{h_\eta}{\big(1+a^\nu\rho(\xi, \eta)\big)^{M}}\bigg)^q
\le \bigg(\sum_{\eta\in\cZ_\nu} \frac{1}{\big(1+a^\nu\rho(\xi, \eta)\big)^{M}}\bigg)^{q-1}
\sum_{\eta\in\cZ_\nu} \frac{h_\eta^q}{\big(1+a^\nu\rho(\xi, \eta)\big)^{M}}.
\end{equation*}
Now, \eqref{est-conv21} follows from this, \eqref{est-conv2}, and \lemref{lem:Besov_1} with $m=0$.

Using \eqref{est-conv} and \eqref{est-conv21} we get for $\xi\in\cZ_\ell$ and $x\in A_\xi$
\begin{multline*}
|\Phi_n*f(x)|^q
\le c\sum_{\beta (n-2)\le \nu\le \beta n+2} \sum_{\eta\in\cZ_\nu}\frac{h_\eta^q}{\big(1+a^\nu\rho(\xi, \eta)\big)^{Mq_*}} \\
\le c\sum_{\beta (n-2)\le \nu\le \beta n+2} \sum_{\eta\in\cZ_\nu}\frac{|B(\eta,2a^{-\nu})|^{-1}}{\big(1+a^\nu\rho(\xi, \eta)\big)^{Mq_*}}
\int_{B(\eta,2a^{-\nu})}|\Ups_\nu*f(y)|^q d\sigma(y).
\end{multline*}
Above we also used that if $h_\eta$ is attained for $\zeta=\zeta_0$, then
\begin{align*}
h_\eta^q &\le \frac{1}{|B(\zeta_0,a^{-\nu-\kappa})|}\int_{B(\zeta_0,a^{-\nu-\kappa})}|\Ups_\nu*f(y)|^q d\sigma(y)
\\
&\le \frac{(2a^\kappa)^{d-1}}{|B(\eta,2a^{-\nu})|}\int_{B(\eta,2a^{-\nu})}|\Ups_\nu*f(y)|^q d\sigma(y).
\end{align*}
Hence, for any $\xi\in\cZ_\ell$
\begin{multline}\label{est-conv3}
\int_{A_\xi} \big(b^{sn}|\Phi_n*f(x)|\big)^q d\sigma(x)\\
\le c\sum_{\beta (n-2)\le \nu\le \beta n+2} \sum_{\eta\in\cZ_\nu}\frac{1}{\big(1+a^\nu\rho(\xi, \eta)\big)^{Mq_*}}
\int_{B(\eta,2a^{-\nu})}(a^{s\nu}|\Ups_\nu*f(y)|)^q d\sigma(y).
\end{multline}

Set $k:=\left\lceil \beta(j-2)\right\rceil$ and
note that $\ell,\nu\ge k$ for $n\ge j$.
Fix $\omega, \zeta\in\cZ_k$.
We claim that for any $\eta\in\cZ_\nu$ such that  $B(\eta,2a^{-\nu})\cap A_\zeta\ne\emptyset$ we have
\begin{equation}\label{est-conv31}
\sum_{\substack{\xi\in\cZ_\ell\\ A_\xi\cap A_\omega\ne\emptyset}}\frac{1}{\big(1+a^\nu\rho(\xi, \eta)\big)^{Mq_*}}
\le \frac{c}{\big(1+a^k\rho(\omega, \zeta)\big)^{Mq_*}}.
\end{equation}
Indeed, if $\rho(\omega,\zeta)\le 6a^{-k}$ then \lemref{lem:Besov_1} with $m=0$ and $a^\nu\sim a^\ell$ give
\begin{equation*}
\sum_{\substack{\xi\in\cZ_\ell\\ A_\xi\cap A_\omega\ne\emptyset}}\frac{1}{\big(1+a^\nu\rho(\xi, \eta)\big)^{Mq_*}} \\
\le\sum_{\xi\in\cZ_\ell}\frac{c}{\big(1+a^\ell\rho(\xi, \eta)\big)^{Mq_*}}
\le c,
\end{equation*}
which implies \eqref{est-conv31}.
On the other hand, if $\rho(\omega,\zeta)> 6a^{-k}$ then $a^{k}\rho(\xi, \eta)> 1$
and $\rho(\xi, \eta)> \rho(\omega,\zeta)/6$ for every $\xi\in\cZ_\ell$ such that $A_\xi\cap A_\omega\ne\emptyset$.
Note that \eqref{disjoint} and \eqref{sph_cap2} imply
$\#\{\xi\in\cZ_\ell : A_\xi\cap A_\omega\ne\emptyset\}\le c a^{(\ell-k)(d-1)}\le c a^{(\nu-k)(d-1)}$. Then
\begin{align*}
\sum_{\substack{\xi\in\cZ_\ell\\ A_\xi\cap A_\omega\ne\emptyset}} &\frac{1}{\big(1+a^\nu\rho(\xi, \eta)\big)^{Mq_*}}
\\
&\le\sum_{\substack{\xi\in\cZ_\ell\\ A_\xi\cap A_\omega\ne\emptyset}}
\frac{1}{\big(a^\nu\rho(\xi, \eta)\big)^{d-1}}\frac{1}{\big(1+a^\nu\rho(\xi, \eta)\big)^{Mq_*-d+1}}
\\
&\le\frac{c a^{(\nu-k)(d-1)}}{\big(a^\nu\rho(\omega, \zeta)\big)^{d-1}}\frac{1}{\big(1+a^k\rho(\omega, \zeta)\big)^{Mq_*-d+1}}
\le \frac{c}{\big(1+a^k\rho(\omega, \zeta)\big)^{Mq_*}},
\end{align*}
which completes the proof of \eqref{est-conv31}.

Using \eqref{est-conv3} and \eqref{est-conv31} we get, for any $\omega\in\cZ_k$,
\begin{align*}
\int_{A_\omega} &\big(b^{sn}|\Phi_n*f(x)|\big)^q d\sigma(x)
\le \sum_{\substack{\xi\in\cZ_\ell\\ A_\xi\cap A_\omega\ne\emptyset}}\int_{A_\xi} \big(b^{sn}|\Phi_n*f(x)|\big)^q d\sigma(x)
\\
&\le c\sum_{\beta (n-2)\le \nu\le \beta n+2} \sum_{\eta\in\cZ_\nu}\sum_{\substack{\xi\in\cZ_\ell\\ A_\xi\cap A_\omega\ne\emptyset}}
\frac{\int_{B(\eta,2a^{-\nu})}(a^{s\nu}|\Ups_\nu*f(y)|)^q d\sigma(y)}{\big(1+a^\nu\rho(\xi, \eta)\big)^{Mq_*}}
\\
&\le c\sum_{\beta (n-2)\le \nu\le \beta n+2} \sum_{\zeta\in\cZ_k}\sum_{\substack{\eta\in\cZ_\nu\\ B(\eta,2a^{-\nu})\cap A_\zeta\ne\emptyset}}
\frac{\int_{B(\eta,2a^{-\nu})}(a^{s\nu}|\Ups_\nu*f(y)|)^q d\sigma(y)}{\big(1+a^k\rho(\omega, \zeta)\big)^{Mq_*}}
\\
&\le c\sum_{\beta (n-2)\le \nu\le \beta n+2} \sum_{\zeta\in\cZ_k}\frac{1}{\big(1+a^k\rho(\omega, \zeta)\big)^{Mq_*}}
\int_{B(\zeta,5a^{-k})}(a^{s\nu}|\Ups_\nu*f(y)|)^q d\sigma(y).
\end{align*}
In the last inequality we use that for every $\zeta\in\cZ_k$ and $\eta\in\cZ_\nu$ such that $B(\eta,2a^{-\nu})\cap A_\zeta\ne\emptyset$
we have $B(\eta,2a^{-\nu})\subset B(\zeta,5a^{-k})$ and $\|\sum_{\eta\in\cZ_\nu}\ONE_{B(\eta,2a^{-\nu})}\|_{L^\infty(\SS)}\le c$.
Hence
\begin{multline}\label{est-conv33}
\sum_{n\ge j}\int_{A_\omega} \big(b^{sn}|\Phi_n*f(x)|\big)^q d\sigma(x)\\
\le c\sum_{\zeta\in\cZ_k}\frac{1}{\big(1+a^k\rho(\omega, \zeta)\big)^{Mq_*}}
\sum_{\nu\ge k} \int_{B(\zeta,5a^{-k})}(a^{s\nu}|\Ups_\nu*f(y)|)^q d\sigma(y).
\end{multline}

Using that $B(z,r)$ is covered by at most a fixed number of sets $A_\omega$, $\omega\in\cZ_k$,
and $|B(z,r)|\sim |B(\zeta,5a^{-k})|$, $\zeta\in\cZ_k$,
we obtain from \eqref{est-conv33}, \eqref{F-norm-a},
and \lemref{lem:Besov_1} with $m=0$ that
\begin{align}\label{est-conv5}
\cN_{z,r}^q&=\frac{1}{|B(z,r)|}\int_{B(z,r)} \sum_{n\ge j}\big(b^{sn}|\Phi_n*f(x)|\big)^q d\sigma(x) 
\\
&\le c\sum_{\substack{\omega\in\cZ_k\\ B(z,r)\cap A_\omega\ne\emptyset}}
\sum_{\zeta\in\cZ_k}\frac{|B(\zeta,5a^{-k})|^{-1}}{\big(1+a^k\rho(\omega, \zeta)\big)^{Mq_*}}
\int_{B(\zeta,5a^{-k})}\sum_{\nu\ge k} (a^{s\nu}|\Ups_\nu*f(y)|)^q d\sigma(y) \nonumber
\\
&\le c \|f\|_{\cF^{s q}_\infty(a,\Ups)}^q. \nonumber
\end{align}
In the last inequality of \eqref{est-conv5} we also use that the summation in \eqref{F-norm-a}
for the spherical cap $B(\zeta,5a^{-k})$ starts with $\left\lfloor \log_{a}\pi/(5a^{-k})\right\rfloor=k-1$.
Inequality \eqref{est-conv5} and definition \eqref{F-norm-infty0} lead to
$\|f\|_{\cF^{sq}_\infty}\le c \|f\|_{\cF^{sq}_\infty(a,\Ups)}$ for $q<\infty$.
For $q=\infty$ the same inequality follows directly from \eqref{est-conv} using \eqref{F-infty-infty} and \eqref{eq:conv_1}.
Interchanging the roles of $\Ups$, $\Phi$ and $a$, $b$, respectively, we get from above
$\|f\|_{\cF^{sq}_\infty(a,\Ups)} \le c\|f\|_{\cF^{sq}_\infty}$, which completes the proof.
\end{proof}

\propref{prop:independ} and \propref{prop:TL-infty}~b) immediately imply the following

\begin{cor}\label{cor:independ3}
The definition of the Separable Triebel-Lizorkin spaces $\sepTL{\cF}{s}{q}(\SS)$,
$s\in\RR$, $0< q\le\infty$, $($Definition~\ref{def:VF-infty}$)$ is independent of the choice of $b>1$ and $\varphi$.
\end{cor}

In the case $p<\infty$ we have
\begin{prop}\label{prop:independ2}
The definition of the Besov spaces $\cB^{sq}_p$ $(s\in\RR$, $0<p, q\le\infty)$ and Triebel-Lizorkin spaces $\cF^{sq}_p$
$(s\in\RR$, $0<p<\infty$, $0<q\le\infty)$ $($see Definition~\ref{def:B-F-spaces}$)$ is independent of the choice of $b>1$ and $\varphi$.
\end{prop}
\begin{proof}
The proof is similar to the proof of \propref{prop:independ} but simpler; we omit it.
\end{proof}

\medskip

The embeddings of the Triebel-Lizorkin spaces $\cF^{sq}_\infty$ and the Besov spaces $\cB^{sq}_\infty$
(see \eqref{B-norm}) are as follows (cf. \eqref{eq:2}, \eqref{eq:8b})
\begin{equation}\label{eq:2-infty}
\cF^{sq_0}_\infty\subset \cF^{sq_1}_\infty,\quad \hbox{if}\;\;q_0\le q_1;
\end{equation}

\begin{equation}\label{eq:8b-infty}
\cB^{sq}_{\infty}\subset \cF^{sq}_{\infty}\subset \cF^{s\infty}_{\infty}=\cB^{s\infty}_{\infty},\quad  \hbox{if}\;\;q<\infty;
\end{equation}

The definition of the Triebel-Lizorkin spaces
$\cF^{sq}_\infty$, $s\in\RR$, $0<q\le\infty$, (Definition~\ref{def:F-infty}) on the sphere $\SS$ with $b=2$
is similar to the definition of the corresponding Triebel-Lizorkin spaces on $\RR^d$, see \cite[\S5 and \S12]{FJ2}.
In order to obtain a decomposition of $\cF^{sq}_\infty(\SS)$ similar to the decomposition of $\cF^{sq}_\infty(\RR^d)$ from \cite{FJ2}
we need a nested structure of open sets on $\SS$ with the main properties of the dyadic cubs on $\RR^d$.
Such a structure is provided by \thmref{nested_sets} under the requirement $b>3$.
For this reason we define the Triebel-Lizorkin spaces in \eqref{F-norm-infty0}
with a general base $b$ instead of using the standard base $2$.

We consider the Triebel-Lizorkin spaces $\cF^{sq}_\infty$
with parameters $(s, q)$  in the range
\begin{equation}\label{indices-infty}
\QQ_\infty=\QQ_\infty(A):=\big\{(s, q): |s| \le A,\; \;\hbox{and}\;A^{-1}\le q\le\infty\big\},
\end{equation}
where $A >1$ is a fixed constant.
Condition \eqref{indices-infty} is the analog for $p=\infty$ of condition \eqref{indices-1}.

\begin{prop}\label{def:equiv_F-infty}
Let $(s,q)\in \QQ_\infty(A)$, $A>1$.
Under the assumptions on $\varphi$ of Definition~\ref{def:B-F-spaces} with $b>3$ we have
\begin{equation}\label{F-norm-infty}
\|f\|_{\cF^{s q}_\infty} \sim
\sup_{j\ge 0}\sup_{\xi\in\cX_j}\bigg(\frac{1}{|Q_\xi|}\int_{Q_\xi}\sum_{n\ge j} \big(b^{sn}|\Phi_n*f(x)|\big)^q d\sigma(x)\bigg)^{1/q},
\end{equation}
with constants of equivalence depending on $d, b, \betaw, A$.
Here $\cX_j$ is from \eqref{index_set}
and $Q_\xi$, $\xi\in\cX_j$, are defined in \thmref{nested_sets} and Remark~\ref{rem:3}.
\end{prop}

\begin{proof}
For $q=\infty$ equivalence \eqref{F-norm-infty} is fulfilled as an equality. So, assume $q<\infty$.
Let $\xi\in\cX_j$, $j\in\NN$. Set $r=\pi b^{-j}$. Using \eqref{dyadic_eq:4} and \eqref{sph_cap2} we get
\begin{equation*}
B\Big(\xi,\frac{\bar{\beta} \bar{\gamma}}{2} b^{-j}\Big)\subset Q_\xi\subset
B\Big(\xi,\frac{b \bar{\gamma}}{2(b-1)} b^{-j}\Big)\subset B(\xi,r),\quad |Q_\xi|\sim |B(\xi,r)|.
\end{equation*}
Hence
\begin{multline*}
\frac{1}{|Q_\xi|}\int_{Q_\xi}\sum_{n\ge j} \big(b^{sn}|\Phi_n*f(x)|\big)^q d\sigma(x)\\
\le c \frac{1}{|B(\xi,r)|}\int_{B(\xi,r)}\sum_{n\ge j} \big(b^{sn}|\Phi_n*f(x)|\big)^q d\sigma(x)
\le c \|f\|_{\cF^{s q}_\infty}^q.
\end{multline*}
If $\xi\in\cX_0$, then we have $Q_\xi=B(\xi,r)$, $r=\pi$, and we can repeat the same arguments.
Hence, $\|f\|_\star\le c \|f\|_{\cF^{s q}_\infty}$, where $\|f\|_\star$ stands for the right-hand side norm in \eqref{F-norm-infty}.

For the proof of the inequality $\|f\|_{\cF^{s q}_\infty}\le c\|f\|_\star$ we take an arbitrary spherical cap $B(y,r)$, $y\in\SS$,
and choose $j\in\NN_0$ such that $\pi b^{-j-1}<r\le \pi b^{-j}$.
From \thmref{nested_sets}, \eqref{sph_cap}, and \eqref{dyadic_eq:6} we get
$\#\{\xi\in\cX_j : Q_\xi\cap B(y,r)\ne\emptyset\}\le c$ and $|Q_\xi|\sim |B(y,r)|$ for $\xi\in\cX_j$. Hence
\begin{multline*}
\frac{1}{|B(\xi,R)|}\int_{B(\xi,R)}\sum_{n\ge j} \big(b^{sn}|\Phi_n*f(x)|\big)^q d\sigma(x)\\
\le \sum_{\substack{\xi\in\cX_j\\ Q_\xi\cap B(y,r)\ne\emptyset}} \frac{|Q_\xi|}{|B(\xi,R)|}\frac{1}{|Q_\xi|}\int_{Q_\xi}\sum_{n\ge j} \big(b^{sn}|\Phi_n*f(x)|\big)^q d\sigma(x)
\le c\|f\|_\star^q,
\end{multline*}
which completes the proof.
\end{proof}

We are now ready to define the respective Triebel-Lizorkin \emph{sequence} spaces.

\begin{defn}\label{def:TL_infty}
Let $b>3$, $s\in \RR$, and $0<q\le\infty$. Then  the Triebel-Lizorkin sequence space $\ff_\infty^{sq}:=\ff_\infty^{sq}(\cX)$
is defined as the space of all sequences of complex numbers
$h:=\{h_{\xi}\}_{\xi\in \cX}$ such that
\begin{equation}\label{def-f-space-infty}
\|h\|_{\ff_\infty^{sq}} :=\sup_{\xi\in\cX}\bigg(\frac{1}{|Q_\xi|}\int_{Q_\xi}\sum_{Q_\eta\subset Q_\xi}
\big[|Q_\eta|^{-s/(d-1)-1/2}
|h_{\eta}|\ONE_{Q_\eta}(x)\big]^q d\sigma(x)\bigg)^{1/q} <\infty
\end{equation}
if $q<\infty$ and
\begin{equation}\label{def-f-space-infty1}
\|h\|_{\ff_\infty^{s\infty}} :=\sup_{\xi\in\cX}
|Q_\xi|^{-s/(d-1)-1/2} |h_{\xi}| <\infty.
\end{equation}
Here $\cX$ is from \eqref{index_set} and $Q_\xi$, $\xi\in\cX$, are defined in \thmref{nested_sets} $($see also Remark~\ref{rem:2}$)$.
\end{defn}

For $0<q<\infty$ one can carry out the integration in \eqref{def-f-space-infty} to obtain
\begin{equation}\label{def-f-space-infty2}
\|h\|_{\ff_\infty^{sq}}=\sup_{\xi\in\cX}\bigg(\frac{1}{|Q_\xi|}\sum_{Q_\eta\subset Q_\xi}
\big[|Q_\eta|^{-s/(d-1)-1/2}
|h_{\eta}|\big]^q |Q_\eta|\bigg)^{1/q},
\end{equation}

From \eqref{def-f-space-infty2} it follows that
\begin{equation}\label{def-f-space-infty3}
\|h\|_{\ff_\infty^{sq}}
=\|\{|Q_\eta|^{-s/(d-1)}h_\eta\}_{\eta\in\cX}\|_{\ff_\infty^{0q}},\quad s\in\RR,
\end{equation}
\begin{equation}\label{def-f-space-infty4}
\|h\|_{\ff_\infty^{sq}}
=\|\{|Q_\eta|^{-(s-\bar{s})/(d-1)}h_\eta\}_{\eta\in\cX}\|_{\ff_\infty^{\bar{s}q}}, \quad s,\bar{s}\in\RR.
\end{equation}

Using \eqref{dyadic_eq:6}
in \eqref{def-f-space-infty2} we obtain for every $(s, q)\in\QQ_\infty(A)$
\begin{equation}\label{def-f-space-infty5}
c^{-1}\|h\|_{\ff_\infty^{sq}}
\le \sup_{j\in\NN_0}\sup_{\xi\in\cX_j}\bigg(\sum_{k=0}^\infty\sum_{\substack{\eta\prec\xi\\ \eta\in\cX_{j+k}}}
\big[b^{(j+k)(s+(d-1)/2)}|h_{\eta}|\big]^q b^{-k(d-1)}\bigg)^{1/q}
\le c\|h\|_{\ff_\infty^{sq}}
\end{equation}
with constants $c$ depending only on $b, \gamma, \beta$ from \thmref{nested_sets}, on $A$ from \eqref{indices-infty} and on $d$.
The middle quantity of \eqref{def-f-space-infty5} does not depend of the particular choice of the sets $\{Q_\xi\}$ in \thmref{nested_sets}
but still requires the partial ordering of the elements of $\cX$.

\begin{defn}\label{def:VTL_infty}
The the separable Triebel-Lizorkin sequence space $\sepTL{\ff}{s}{q}:=\sepTL{\ff}{s}{q}(\cX)$
is defined as the space of all sequences of complex numbers
$h:=\{h_{\xi}\}_{\xi\in \cX}$ such that $h\in\ff_\infty^{sq}$ and
\begin{equation*}
\lim_{j\to\infty}\sup_{\xi\in\cX_j}\bigg(\frac{1}{|Q_\xi|}\sum_{Q_\eta\subset Q_\xi}
\big[|Q_\eta|^{-s/(d-1)-1/2}|h_{\eta}|\big]^q |Q_\eta|\bigg)^{1/q}=0,\quad q<\infty,
\end{equation*}
\begin{equation*}
\lim_{j\to\infty}\sup_{\xi\in\cX_j} 
|Q_\xi|^{-s/(d-1)-1/2} |h_{\xi}| = 0,
\quad q=\infty.
\end{equation*}
Also, $\sepTL{\ff}{s}{q}$ inherits the norm from $\ff^{sq}_\infty$.
\end{defn}

From Definitions~\ref{def:TL_infty} and \ref{def:VTL_infty} we immediately get the following properties.

\begin{prop}\label{prop:sep-TL-infty}
Let $s\in\RR$, $0<q\le\infty$. Then:

$(a)$
$\ff_\infty^{sq}$ and $\sepTL{\ff}{s}{q}$ are $($quasi-$)$Banach spaces.

$(b)$
$\sepTL{\ff}{s}{q}$ is the closure of the compactly supported sequences in the $\ff_\infty^{sq}$ norm.

$(c)$
$\sepTL{\ff}{s}{q}$ is a \emph{separable space}, while the space $\ff_\infty^{sq}$ is \emph{non-separable}.

\end{prop}

\subsection{Almost diagonal operators}

The almost diagonal matrices we shall use are

\begin{defn}\label{def:star-star-seq}
Let $M,K>0$. Define $N_\xi=b^j$ for $\xi\in\cX_j$, $j\in\NN_0$.
Set $\Omega^{(K,M)}=\{\omega_{\xi,\eta}^{(K,M)} : \xi,\eta\in\cX\}$, where
\begin{equation}\label{eq:omega_xi_eta}
\omega_{\xi,\eta}^{(K,M)}=\omega_{\xi,\eta}
:=\left(\frac{\min\{N_\xi,N_\eta\}}{\max\{N_\xi,N_\eta\}}\right)^{K+(d-1)/2}
\frac{1}{\left(1+\min\{N_\xi,N_\eta\}\rho(\xi,\eta)\right)^M}.
\end{equation}
\begin{equation}\label{omega_infty_xi_eta}
\omega_{\xi,\eta}^{(\infty,M)}
:=\left\{\begin{array}{l l}
 \left(1+N_\xi\rho(\xi,\eta)\right)^{-M},& N_\eta=N_\xi;\\
0,&N_\eta\ne N_\xi.
\end{array}\right.
\end{equation}
\end{defn}

Note that for any $K, M, \xi, \eta$ we have $\omega_{\xi,\eta}^{(\infty,M)}\le\omega_{\xi,\eta}^{(K,M)}$.

From Definition~\ref{def:star-star-seq} and \lemref{lem:Besov_1} we obtain
\begin{lem}\label{lem:ad}
Let $K>(d-1)/2$ and $M>d-1$. Set $\bar{\omega}_{\xi}=\sum_{\zeta\in\cX}\omega_{\xi,\zeta}^{(K,M)}$ for $\xi\in\cX$.
Then $\bar{\omega}_{\xi}\le c$, where $c$ is a constant depending on $d$, $\lambda$, $b$, $K$, and $M$.
\end{lem}
\begin{proof}
Let $\xi\in\cX_j$, $j\in\NN_0$. Using \eqref{eq:omega_xi_eta} and \lemref{lem:Besov_1} we obtain
\begin{align*}
\bar{\omega}_{\xi}
&=\sum_{n=0}^{j-1}\sum_{\zeta\in\cX_n}\frac{b^{(n-j)(K+(d-1)/2)}}{(1+b^n\rho(\xi,\zeta))^M}
+\sum_{n=j}^\infty\sum_{\zeta\in\cX_n}\frac{b^{(j-n)(K+(d-1)/2)}}{(1+b^j\rho(\xi,\zeta))^M}\\
&\le c\sum_{n=0}^{j-1}b^{(n-j)(K+(d-1)/2)}
+c\sum_{n=j}^\infty b^{(j-n)(K+(d-1)/2)}b^{(n-j)(d-1)}
\le c,
\end{align*}
which completes the proof.
\end{proof}

In the following two statements we establish the fact that
the almost diagonal operators bounded element-wise by $\Omega^{(K,M)}$ are bounded on $\ff_\infty^{sq}$.
The proof of the boundedness for $q=\infty$ is simpler.

\begin{prop}\label{almost_diag_infty}
Let  $s\in\RR$, $K>|s|$ and $M>d-1$. If $h\in\ff_\infty^{s\infty}$ then
\begin{equation*}
\|\Omega^{(K,M)}h\|_{\ff_\infty^{s\infty}}\le c\|h\|_{\ff_\infty^{s\infty}}.
\end{equation*}
Moreover, if $h\in\sepTLb{\ff}{s}{\infty}$ then $\Omega^{(K,M)}h\in\sepTLb{\ff}{s}{\infty}$.
\end{prop}
\begin{proof}
Recall that $\|h\|_{\ff_\infty^{s\infty}}=\sup_{\xi\in\cX}|Q_\xi|^{-s/(d-1)-1/2}|h_\xi|$. Fix $\xi\in\cX_j$, $j\in\NN_0$, and write
\begin{equation}\label{almost_diag_infty_1}
|Q_\xi|^{-s/(d-1)-1/2}\sum_{\eta\in\cX}\omega_{\xi,\eta}|h_\eta|
\le\|h\|_{\ff_\infty^{s\infty}}\sum_{n=0}^\infty\sum_{\eta\in\cX_n}\omega_{\xi,\eta}\frac{|Q_\xi|^{-s/(d-1)-1/2}}{|Q_\eta|^{-s/(d-1)-1/2}}.
\end{equation}
Then using \eqref{eq:omega_xi_eta} and \lemref{lem:Besov_1} we get
\begin{align}
\sum_{n=0}^\infty\sum_{\eta\in\cX_n}\omega_{\xi,\eta}&\frac{|Q_\xi|^{-s/(d-1)-1/2}}{|Q_\eta|^{-s/(d-1)-1/2}}\label{almost_diag_infty_2}\\
&\le c\sum_{n=0}^{j-1}\sum_{\eta\in\cX_n}\frac{b^{(n-j)(K+(d-1)/2)}}{(1+b^n\rho(\xi,\eta))^M}b^{(j-n)(s+(d-1)/2)}\nonumber\\
&+c\sum_{n=j}^\infty \sum_{\eta\in\cX_n}\frac{b^{(j-n)(K+(d-1)/2)}}{(1+b^j\rho(\xi,\eta))^M}b^{(j-n)(s+(d-1)/2)}\nonumber\\
&\le c\sum_{n=0}^{j-1}b^{(n-j)(K-s)}+c\sum_{n=j}^\infty b^{(j-n)(K+s+d-1)} b^{(n-j)(d-1)}\le c.\nonumber
\end{align}
Now \eqref{almost_diag_infty_1} and \eqref{almost_diag_infty_2} prove the proposition for every $h\in\ff_\infty^{s\infty}$.

In order to prove this proposition in the separable case pick $\eps>0$ and for $h\in\sepTLb{\ff}{s}{\infty}$ choose:

\noindent
(a) $n_0$ such that $|Q_\eta|^{-s/(d-1)-1/2}|h_\eta|\le \eps \|h\|_{\ff_\infty^{s\infty}}$ for every $\eta\in\cX_n$, $n\ge n_0$;

\noindent
(b) $n_1$ such that $b^{-n_1(K-s)}\le\eps(1-b^{-(K-s)})$.

Now for $\xi\in\cX_j$, $j\ge n_0+n_1$ we repeat the same proof with the modifications:
(i) We separate the sum on $n$ in \eqref{almost_diag_infty_1} to $n<n_0$ and $n\ge n_0$ and use (a) in the second sum.
(ii) We estimate the sum on $n<n_0$ in \eqref{almost_diag_infty_2} by $c\eps$ by applying (b) because $j-n_0\ge n_1$.
(iii) We estimate the sum on $n\ge n_0$ in \eqref{almost_diag_infty_2} by $c$ as in the non-separable case.
Thus, we show that for all $\xi\in\cX_j$, $j\ge n_0+n_1$,
\begin{equation*}
|Q_\xi|^{-s/(d-1)-1/2}\sum_{\eta\in\cX}\omega_{\xi,\eta}|h_\eta|
\le c\eps \|h\|_{\ff_\infty^{s\infty}},
\end{equation*}
which completes the proof.
\end{proof}

The proof of the boundedness for $q<\infty$ is more involved.

\begin{thm}\label{thm:almost_diag}
Let $s\in\RR$, $0<q<\infty$, $q_*=\min\{q,1\}$ and $K$, $M$ satisfy
\begin{equation}\label{almost_diag:2}
K>\max\Big\{sq-(d-1)\Big(\frac{q_*}{2}-\frac{q}{2}\Big),-sq-(d-1)\Big(\frac{q_*}{2}+\frac{q}{2}-1\Big),(d-1)\frac{q_*}{2}\Big\}\frac{1}{q_*},
\end{equation}
\begin{equation}\label{almost_diag:3}
M>\frac{d-1}{q_*}.
\end{equation}
Then
\begin{equation}\label{almost_diag:1}
\|\Omega^{(K,M)}h\|_{\ff_\infty^{sq}}\le c\|h\|_{\ff_\infty^{sq}}.
\end{equation}
Moreover, if $h\in\sepTL{\ff}{s}{q}$, then $\Omega^{(K,M)}h\in\sepTL{\ff}{s}{q}$.
\end{thm}

\begin{proof}
We shall use the abbreviated notation $\Omega:=\Omega^{(K,M)}$, $\omega_{\xi,\eta}:=\omega_{\xi,\eta}^{(K,M)}$.
For $1\le q<\infty$ using H\"{o}lder's inequality and \lemref{lem:ad} we obtain
\begin{equation}\label{ad:eq01}
|(\Omega h)_{\eta}|^q\le\big[\sum_{\zeta\in\cX}\omega_{\eta,\zeta}|h_{\zeta}|\big]^q
\le \bar{\omega}_\eta^{q-1}\sum_{\zeta\in\cX}\omega_{\eta,\zeta}|h_{\zeta}|^q
\le c^{q-1}\sum_{\zeta\in\cX}\omega_{\eta,\zeta}|h_{\zeta}|^q.
\end{equation}
For $0< q<1$ using the $q$-inequality we have
\begin{equation}\label{ad:eq02}
|(\Omega h)_{\eta}|^q\le\big[\sum_{\zeta\in\cX}\omega_{\eta,\zeta}|h_{\zeta}|\big]^q
\le \sum_{\zeta\in\cX}\omega_{\eta,\zeta}^q|h_{\zeta}|^q
\end{equation}
In order to evaluate for $\Omega h$ the expression in the right-hand side of \eqref{def-f-space-infty5} we fix $\xi\in\cX$. Assume $\xi\in\cX_j$.
We shall establish the estimate
\begin{align}\label{ad:eq13}
S(\xi)&:=b^{j(d-1)}\sum_{k=0}^\infty\sum_{\substack{\eta\prec\xi\\ \eta\in\cX_{j+k}}}
\big[b^{(j+k)(s+(d-1)/2)}|(\Omega h)_{\eta}|\big]^q b^{-(j+k)(d-1)}\\
&\le c\sum_{y\in\cX_j}\frac{b^{j(d-1)}}{(1+b^j\rho(\xi,y))^{Mq_*}} \sum_{n=j}^\infty\sum_{\substack{\zeta\in\cX_n\\ \zeta\prec y}}
b^{n(sq+(d-1)(q/2-1))}|h_{\zeta}|^q\nonumber \\
&+c\sum_{n=0}^{j-1}b^{(n-j)L}\sum_{\zeta\in\cX_n} \frac{b^{j(d-1)}}{(1+b^n\rho(\xi,\zeta))^{Mq_*}}
b^{n(sq+(d-1)(q/2-1))}|h_\zeta|^q\nonumber
\end{align}
with $L=Kq_*-sq+(d-1)(q_*/2-q/2+1)>d-1$ (cf. \eqref{almost_diag:2}).
If $j=0$ we assume that the second summand in the right-hand side of \eqref{ad:eq13} is zero.

Using \eqref{ad:eq01}--\eqref{ad:eq02} we get
\begin{align}\label{ad:eq03}
S(\xi)&=b^{j(d-1)}\sum_{k=0}^\infty\sum_{\substack{\eta\prec\xi\\ \eta\in\cX_{j+k}}}
\big[b^{(j+k)(s+(d-1)/2)}|(\Omega h)_{\eta}|\big]^q b^{-(j+k)(d-1)}\\
&\le c b^{j(d-1)}\sum_{k=0}^\infty\sum_{\substack{\eta\prec\xi\\ \eta\in\cX_{j+k}}}
\big[b^{(j+k)(s+(d-1)/2)q}\sum_{\zeta\in\cX}\omega_{\eta,\zeta}^{q_*}|h_{\zeta}|^q\big] b^{-(j+k)(d-1)} \nonumber \\
&=cb^{j(d-1)}\sum_{n=0}^\infty\sum_{\zeta\in\cX_n} G(\xi,\zeta) b^{n(sq+(d-1)(q/2-1))}|h_{\zeta}|^q\nonumber
\end{align}
with
\begin{equation}\label{ad:eq05}
G(\xi,\zeta):=\sum_{k=0}^\infty\sum_{\substack{\eta\prec\xi\\ \eta\in\cX_{j+k}}}\omega_{\eta,\zeta}^{q_*}b^{(j+k-n)(sq+(d-1)(q/2-1))}.
\end{equation}
In evaluating $G(\xi,\zeta)$ we consider two cases depending on whether the index $\zeta$ belongs to higher or lower level than the level of $\xi$.

Case 1: $n\ge j$, $\zeta\in\cX_n$. According to \thmref{nested_sets} there exists a unique $y\in \cX_j$ such that $\zeta\prec y$.
In this case we shall establish the estimate
\begin{equation}\label{ad:eq06}
G(\xi,\zeta)\le c (1+b^j\rho(y,\xi))^{-Mq_*},\quad \zeta\prec y,~y\in \cX_j.
\end{equation}

Set $\delta:=\gamma b^{-j}$. If  $\rho(\xi,y)< 3\delta$ using \lemref{lem:Besov_1} and \eqref{almost_diag:2}--\eqref{almost_diag:3}  we obtain
\begin{align*}
G(\xi,\zeta)
&=\sum_{k=0}^{n-j-1}\sum_{\substack{\eta\prec\xi\\ \eta\in\cX_{j+k}}}\frac{b^{(k+j-n)(Kq_*+sq+(d-1)(q_*/2+q/2-1))}}{(1+b^{k+j}\rho(\eta,\zeta))^{Mq_*}}\\
&+\sum_{k=n-j}^\infty\sum_{\substack{\eta\prec\xi\\ \eta\in\cX_{j+k}}}\frac{b^{(n-j-k)(Kq_*-sq+(d-1)(q_*/2-q/2+1))}}{(1+b^n\rho(\eta,\zeta))^{Mq_*}}\\
&\le \sum_{k=0}^{n-j-1} \sum_{\eta\in\cX_{j+k}}\frac{b^{(k+j-n)(Kq_*+sq+(d-1)(q_*/2+q/2-1))}}{(1+b^{k+j}\rho(\eta,\zeta))^{Mq_*}}\\
&+ \sum_{k=n-j}^\infty \sum_{\eta\in\cX_{j+k}}\frac{b^{(n-j-k)(Kq_*-sq+(d-1)(q_*/2-q/2+1))}}{(1+b^n\rho(\eta,\zeta))^{Mq_*}}\\
&\le c\sum_{k=0}^{n-j-1} b^{(k+j-n)(Kq_*+sq+(d-1)(q_*/2+q/2-1))}\\
&+c\sum_{k=n-j}^\infty b^{(n-j-k)(Kq_*-sq+(d-1)(q_*/2-q/2+1))} b^{(j+k-n)(d-1)}\\
&\le c\le \frac{c}{(1+3\gamma)^{Mq_*}}\frac{1}{(1+b^j\rho(y,\xi))^{Mq_*}}.
\end{align*}
This proves \eqref{ad:eq06} whenever $\rho(\xi,y)< 3\delta$.

Let $\rho(\xi,y)\ge 3\delta$.
From \thmref{nested_sets} we have $\rho(\zeta,y)<\frac{b}{b-1}\delta$ for $\zeta\prec y$ and $\rho(\eta,\xi)<\frac{b}{b-1}\delta$ for $\eta\prec\xi$. Hence
\begin{equation*}
\rho(\zeta,\eta)\ge\rho(\xi,y)-\rho(\zeta,y)-\rho(\eta,\xi)\ge\rho(\xi,y)-\frac{2b}{b-1}\delta\ge\frac{b-3}{3b-3}\rho(\xi,y).
\end{equation*}
Using this inequality, the estimate from \eqref{number_of_children}, and \eqref{almost_diag:2}--\eqref{almost_diag:3} we obtain
\begin{align*}
G(\xi,\zeta)
&=\sum_{k=0}^{n-j-1}\sum_{\substack{\eta\prec\xi\\ \eta\in\cX_{j+k}}}\frac{b^{(k+j-n)(Kq_*+sq+(d-1)(q_*/2+q/2-1))}}{(1+b^{k+j}\rho(\eta,\zeta))^{Mq_*}}\\
&+\sum_{k=n-j}^\infty\sum_{\substack{\eta\prec\xi\\ \eta\in\cX_{j+k}}}\frac{b^{(n-j-k)(Kq_*-sq+(d-1)(q_*/2-q/2+1))}}{(1+b^n\rho(\eta,\zeta))^{Mq_*}}\\
&\le c \sum_{k=0}^{n-j-1} b^{k(d-1)}\frac{b^{(k+j-n)(Kq_*+sq+(d-1)(q_*/2+q/2-1))}}{(b^{k+j}\rho(y,\xi))^{Mq_*}}\\
&+c \sum_{k=n-j}^\infty  b^{k(d-1)}\frac{b^{(n-j-k)(Kq_*-sq+(d-1)(q_*/2-q/2+1))}}{(b^n\rho(y,\xi))^{Mq_*}}\\
&\le \frac{c}{(b^j\rho(y,\xi))^{Mq_*}}\sum_{k=0}^{n-j-1}b^{(k+j-n)(Kq_*+sq+(d-1)(q_*/2+q/2-1))}\\
&+\frac{c}{(b^j\rho(y,\xi))^{Mq_*}}\sum_{k=n-j}^\infty b^{(n-j-k)(Kq_*-sq+(d-1)(q_*/2-q/2))}\\
&\le \frac{c}{(1+b^j\rho(y,\xi))^{Mq_*}}
\end{align*}
with $b^j\rho(y,\xi)\ge 3\gamma$ used in the last inequality. The above completes the proof of \eqref{ad:eq06}.

Case 2: $n<j$, $\zeta\in\cX_n$.
For every $\zeta\in\cX_n$ we use the maximality of the net $\cX_n$ and \eqref{dyadic_eq:4} to obtain
\begin{equation*}
(1+b^n\rho(\eta,\zeta))^{-1} \le c(1+b^n\rho(\xi,\zeta))^{-1}\quad \forall\eta\prec\xi.
\end{equation*}
Using the above inequality, the estimate from \eqref{number_of_children}  and \eqref{almost_diag:2}--\eqref{almost_diag:3} we obtain
\begin{align}\label{ad:eq08}
G(\xi,\zeta)&=\sum_{k=0}^\infty\sum_{\substack{\eta\prec\xi\\ \eta\in\cX_{j+k}}}\frac{b^{(n-j-k)(Kq_*-sq+(d-1)(q_*/2-q/2+1))}}{(1+b^n\rho(\eta,\zeta))^{Mq_*}}\\
&\le c\sum_{k=0}^\infty \frac{b^{(n-j-k)(Kq_*-sq+(d-1)(q_*/2-q/2+1))}}{(1+b^n\rho(\xi,\zeta))^{Mq_*}}b^{k(d-1)}\nonumber\\
&=c\frac{b^{(n-j)(Kq_*-sq+(d-1)(q_*/2-q/2+1))}}{(1+b^n\rho(\xi,\zeta))^{Mq_*}}.\nonumber
\end{align}
Inserting \eqref{ad:eq06} and \eqref{ad:eq08} in \eqref{ad:eq03}--\eqref{ad:eq05} we obtain \eqref{ad:eq13}.

Using \eqref{ad:eq13}, \lemref{lem:Besov_1}, \eqref{def-f-space-infty5}, and its trivial consequence
\begin{equation*}
b^{n(s+(d-1)/2)}|h_\zeta|\le C_5\|h\|_{\ff_\infty^{sq}},\quad \zeta\in\cX_n,
\end{equation*}
for $n<j$ we obtain
\begin{align}\label{ad:eq11}
S(\xi)&\le c\sum_{y\in\cX_j}\frac{1}{(1+b^j\rho(\xi,y))^{Mq_*}}\|h\|_{\ff_\infty^{sq}}^q\\
&+c\sum_{n=0}^{j-1}b^{(n-j)(L-d+1)}\sum_{\zeta\in\cX_n} \frac{1}{(1+b^n\rho(\xi,\zeta))^{Mq_*}}\|h\|_{\ff_\infty^{sq}}^q \nonumber \\
&\le c\left(1+\frac{b^{-L+d-1}}{1-b^{-L+d-1}}\right) \|h\|_{\ff_\infty^{sq}}^q. \nonumber
\end{align}
Now, inequality \eqref{almost_diag:1} follows from \eqref{ad:eq11} and \eqref{def-f-space-infty5}.

Finally, assume that $h\in\sepTL{\ff}{s}{q}$ and fix $\eps>0$. Let $N_1$ be such that
\begin{equation*}
\sum_{Q_\eta\subset Q_\zeta}
\big[|Q_\eta|^{-s/(d-1)-1/2}|h_{\eta}|\big]^q \frac{|Q_\eta|}{|Q_\zeta|}\le\eps^q,\quad \zeta\in\cX_n,~n\ge N_1.
\end{equation*}
Also, let $N_2$ be such that
\begin{equation*}
b^{N_2(-L+d-1)}\|h\|_{\ff_\infty^{sq}}^q\le\eps^q.
\end{equation*}
Using again \eqref{ad:eq13}, \lemref{lem:Besov_1}, and \eqref{def-f-space-infty5} we obtain for all $\xi\in\cX_j$, $j\ge N_1+N_2$,
\begin{align*}
S(\xi)&\le c\sum_{y\in\cX_j}\frac{1}{(1+b^j\rho(\xi,y))^{Mq_*}}\eps^q\\
&+c\sum_{n=N_1}^{j-1}b^{(n-j)(L-d+1)}\sum_{\zeta\in\cX_n} \frac{1}{(1+b^n\rho(\xi,\zeta))^{Mq_*}}\eps^q\\
&+c\sum_{n=0}^{N_1-1}b^{(n-j)(L-d+1)}\sum_{\zeta\in\cX_n} \frac{1}{(1+b^n\rho(\xi,\zeta))^{Mq_*}}\|h\|_{\ff_\infty^{sq}}^q\\
&\le c\left(1+2\frac{b^{-L+d-1}}{1-b^{-L+d-1}}\right) \eps^q=c\eps^q.
\end{align*}
Hence $\Omega^{(K,M)}h\in\sepTL{\ff}{s}{q}$.
This completes the proof.
\end{proof}

The next lemma follows readily from \cite[Lemma 4.7]{NPW2} and \eqref{dyadic_eq:4} of \thmref{nested_sets}.
Note that \lemref{lem:sup_inf} will coincide with \cite[Lemma 4.7]{NPW2} if we replace $Q_\xi$ with $B(\xi,\gamma b^{-j})$
and $\zeta\prec\xi$ with the nonempty intersection of the respective balls.

\begin{lem}\label{lem:sup_inf}
Let $M\ge d$.
There exists $\kappa\in\NN$ depending only on $d$ 
and a constant $c>0$ depending on $d$ and $M$ such that for all $j\in\NN_0$ and $g\in\Pi_{b^j}(\SS)$ we have
\begin{equation*}
(\Omega^{(\infty,M)}h(g))_\eta\le (\Omega^{(\infty,M)}H(g))_\eta\le c(\Omega^{(\infty,M)}h(g))_\eta,\quad \forall \eta\in\cX,
\end{equation*}
where 
$H(g)_\xi:=\max_{x\in Q_\xi}|g(x)|$
and
\begin{equation*}
h(g)_\xi :=\max\{\min_{x\in Q_\zeta}|g(x)| : \zeta\in\cX_{i+\kappa}, \zeta\prec \xi\}, \quad \xi\in\cX_i, \; i\in\NN_0.
\end{equation*}
\end{lem}

\subsection{Frames}

The needlet system $\Psi=\{\psi_\xi : \xi\in\cX\}$ introduced in Subsection~\ref{subsec:frame-SS}
is used as a backbone for the construction of the frame in terms of shifts of the Newtonian kernel from \cite{IP2}.
In the next theorem we show that $\Psi$ is a self-dual frame
for $\sepTL{\cF}{s}{q}$ the same way it is a self-dual frame for $\cB_p^{s q}$ and $\cF_p^{s q}$ with $p<\infty$.

Recall the {\em analysis} and {\em synthesis} operators, introduced in \eqref{def-oper-S-T}:
\begin{equation*}
S_{\psi}: f\mapsto \{\langle f, \psi_\xi\rangle\}_{\xi\in\cX},
\quad
T_\psi: \{h_\xi\}_{\xi\in\cX} \mapsto \sum_{\xi\in\cX}h_\xi\psi_\xi.
\end{equation*}

\begin{thm}\label{F-norm-infty-equiv}
Let $s\in \RR$ and $0< q\le \infty$.
The operators
$S_{\psi}: \sepTL{\cF}{s}{q} \to \sepTL{\ff}{s}{q}$ and
$T_\psi: \sepTL{\ff}{s}{q} \to \sepTL{\cF}{s}{q}$
are bounded, and
$T_\psi\circ S_{\psi}= I$ on $\sepTL{\cF}{s}{q}$.
Hence,
if $f\in \cS'$, then $f\in \sepTL{\cF}{s}{q}$ if and only if
$\{\langle f,\psi_\xi\rangle\}_{\xi \in \cX}\in \sepTL{\ff}{s}{q}$, and
\begin{equation}\label{F-infty-calderon}
f =\sum_{\xi\in \cX}\langle f, \psi_\xi\rangle \psi_\xi
\quad \mbox{and}\quad
\|f\|_{\cF_\infty^{sq}}
\sim  \|\{\langle f,\psi_\xi\rangle\}\|_{\ff_\infty^{sq}}.
\end{equation}
The convergence in \eqref{F-infty-calderon} is unconditional in the $\cF_\infty^{sq}$-norm
and hence in $\cS'$.
All norm bounds above are uniform for $(s, q)\in\QQ_\infty(A)$, $A>1$,
with $\QQ_\infty(A)$ from \eqref{indices-infty}. 
\end{thm}

\begin{proof}
First, let $q<\infty$.
To prove that $T_\psi: \sepTL{\ff}{s}{q} \to \sepTL{\cF}{s}{q}$ is bounded,
suppose $h\in \sepTL{\ff}{s}{q}$.
We shall show that
$f:=T_\psi h=\sum_{\eta\in\cX}h_\eta\psi_\eta$
converges in the $\cF_\infty^{sq}$-norm.
Assume for a moment that $h$ is finitely supported.
Using that $\Psi_n*\Psi_k(x)=0$ for $|k-n|>1$ we have
\begin{equation*}
|\Psi_n*f(x)|\le\sum_{k=n-1}^{n+1}\sum_{\eta\in\cX_k}|h_\eta||\Psi_n*\psi_\eta(x)|
\le c \sum_{k=n-1}^{n+1}\sum_{\eta\in\cX_k}\frac{b^{k(d-1)/2}|h_\eta|}{(1+b^k\rho(x,\eta))^{M}}
\end{equation*}
and
\begin{equation}\label{thm:equiv01}
(b^{sn}|\Psi_n*f(x)|)^q
\le c \sum_{k=n-1}^{n+1}b^{k(s+(d-1)/2)q}\Bigg(\sum_{\eta\in\cX_k}\frac{|h_\eta|}{(1+b^k\rho(x,\eta))^{M}}\Bigg)^q.
\end{equation}

Let $\xi\in\cX_j$. Using that
\begin{equation*}
(1+N_\zeta\rho(x,\eta))^{-M}
\le c(1+N_\zeta\rho(\zeta,\eta))^{-M},\quad \forall x\in Q_\zeta,~\eta\in\SS,
\end{equation*}
and \eqref{def-f-space-infty5} we get from \eqref{thm:equiv01}
\begin{align}\label{thm:equiv02}
\frac{1}{|Q_\xi|}&\int_{Q_\xi}\sum_{n=j}^\infty (b^{sn}|\Psi_n*f(x)|)^qd\sigma(x)
\\
&\le c \sum_{n=j}^\infty\sum_{k=n-1}^{n+1}\sum_{\substack{\zeta\prec\xi\\ \zeta\in\cX_k}}b^{k(s+(d-1)/2)q}
\Bigg(\sum_{\eta\in\cX_k}\frac{|h_\eta|}{(1+b^k\rho(\zeta,\eta))^{M}}\Bigg)^q\frac{|Q_\zeta|}{|Q_\xi|} \nonumber
\\
&\le c \|\Omega^{(\infty,M)}|h|\|_{\ff_\infty^{sq}}^q. \nonumber
\end{align}
From \eqref{F-norm-infty}, \eqref{thm:equiv02}, and \thmref{thm:almost_diag} it follows that
\begin{equation}\label{T-bound}
\|T_\psi h\|_{\cF_\infty^{s q}}\le c \|\Omega^{(\infty,M)}|h|\|_{\ff_\infty^{sq}}\le c \|h\|_{\ff_\infty^{sq}}
\end{equation}
for any finitely sequences $h$.

Consider now an arbitrary $h\in \sepTL{\ff}{s}{q}$.
Let $\{h_{\eta_j}\}_{j\ge 1}$ be an arbitrary ordering of the coordinates $\{h_\eta\}$ of $h$
and set $h_k:= (h_{\eta_1}, \dots, h_{\eta_k}, 0, \dots)$, $k\in\NN$.
From the definition of $\sepTL{\ff}{s}{q}$ (Definition~\ref{def:VTL_infty}) it redily follows that
$\|h-h_k\|_{\ff_\infty^{sq}} \to 0$.
We now apply \eqref{T-bound} to obtain
\begin{equation*}
\|T_\psi h_k - T_\psi h_m\|_{\cF_\infty^{s q}}\le c \|h_k-h_m\|_{\ff_\infty^{sq}} \to 0
\quad\hbox{as}\quad k, m\to \infty.
\end{equation*}
Therefore, $\{T_\psi h_k\}$ is a Cauchy sequence in $\cF_\infty^{s q}$
and since $\cF_\infty^{s q}$ is complete $\{T_\psi h_k\}$ is convergent.
Furthermore, using \eqref{T-bound}
\begin{equation*}
\|T_\psi h_k\|_{\cF_\infty^{s q}}\le c \|h_k\|_{\ff_\infty^{sq}}, \quad \forall k\in\NN,
\quad\hbox{implying}\quad
\|T_\psi h\|_{\cF_\infty^{s q}}\le c \|h\|_{\ff_\infty^{sq}}.
\end{equation*}
i.e. the operator $T_\psi$ is bounded.

The unconditional convergence of the series $\sum_{\eta\in\cX}h_\eta\psi_\eta$ in the $\cF_\infty^{s q}$-norm
follow from the fact that if $h\in \sepTL{\ff}{s}{q}$ is as above and
$\tilde{h}_k:=(0, \dots, 0, h_{\eta_{k+1}}, h_{\eta_{k+2}}, \dots)$, then
$
\|T_\psi \tilde{h}_k\|_{\cF_\infty^{s q}}\le c \|\tilde{h}_k\|_{\ff_\infty^{sq}} \to 0
$
as $k\to \infty$.
Therefore, the operator $T_\psi: \sepTL{\ff}{s}{q} \to \sepTL{\cF}{s}{q}$ is well defined and bounded.

\smallskip

We next show that $S_{\psi}: \cF_\infty^{s q} \to \ff_\infty^{s q}$ is bounded.
Assume that $f\in\cF_\infty^{s q}$. For $\eta\in\cX_n$ set $H_\eta:=b^{-n(d-1)/2}\max_{x\in Q_\eta}|\Psi_n*f(x)|$
and
$$h_\eta:=b^{-n(d-1)/2}\max\{\min_{x\in Q_\zeta}|\Psi_n*f(x)| : \zeta\in\cX_{n+\kappa}, \zeta\prec\eta\},$$
where $\kappa$ is from \lemref{lem:sup_inf}.
Assume that for a given $\eta\in\cX$ the maximum in the definition of $h_\eta$ is attained for $\zeta=\zeta_0$.
Then
\begin{multline}\label{thm:equiv03}
|Q_\eta|^{-q/2}|h_\eta|^q\le c\min_{x\in Q_{\zeta_0}}|\Psi_n*f(x)|^q
\le \frac{c}{|Q_{\zeta_0}|}\int_{Q_{\zeta_0}}|\Psi_n*f(x)|^q d\sigma(x)\\
\le c\frac{|Q_\eta|}{|Q_{\zeta_0}|}\frac{1}{|Q_\eta|}\int_{Q_\eta}|\Psi_n*f(x)|^q d\sigma(x)
\le \frac{c}{|Q_\eta|}\int_{Q_\eta}|\Psi_n*f(x)|^q d\sigma(x).
\end{multline}
Using \eqref{dyadic_eq:6} and \eqref{thm:equiv03} we obtain, for any $\xi\in\cX_j$ and $j\in\NN_0$,
\begin{align}\label{thm:equiv06}
\sum_{Q_\eta\subset Q_\xi}&\big[|Q_\eta|^{-s/(d-1)-1/2}|h_{\eta}|\big]^q \frac{|Q_\eta|}{|Q_\xi|}
\\
&=\sum_{n=j}^\infty \sum_{\substack{\eta\prec\xi\\ \eta\in\cX_n}}
|Q_\eta|^{-sq/(d-1)}|Q_\eta|^{-q/2}|h_{\eta}|^q \frac{|Q_\eta|}{|Q_\xi|} \nonumber
\\
&\le c\sum_{n=j}^\infty\sum_{\substack{\eta\prec\xi\\ \eta\in\cX_n}}
b^{nsq}\frac{1}{|Q_\eta|}\int_{Q_\eta}|\Psi_n*f(x)|^q d\sigma(x)\frac{|Q_\eta|}{|Q_\xi|} \nonumber
\\
&=\frac{c}{|Q_\xi|}\int_{Q_\xi}\sum_{n=j}^\infty (b^{ns}|\Psi_n*f(x)|)^q d\sigma(x). \nonumber
\end{align}
In view of \eqref{def-f-space-infty2} and \eqref{F-norm-infty} inequality \eqref{thm:equiv06} implies
\begin{equation}\label{thm:equiv04}
\|h\|_{\ff_\infty^{s q}}\le c \|f\|_{\cF_\infty^{s q}},
\end{equation}
which, in particular, gives that $h\in \ff_\infty^{s q}$.
From \eqref{def-psi-xi}, \eqref{cubature_w}, the definitions of $H(g)$ and $h(g)$ and \lemref{lem:sup_inf}
we get for $\eta\in\cX_n$, $n\in\NN_0$,
\begin{align}\label{thm:equiv05}
|\left\langle f,\psi_\eta\right\rangle|
&=\widetilde{w}_\eta^{1/2}|\Psi_n*f(\eta)|
\le c b^{-n(d-1)/2}|\Psi_n*f(\eta)|
\\
&\le c H_\eta
\le c (\Omega^{(\infty,M)}H)_\eta\le c (\Omega^{(\infty,M)}h)_\eta. \nonumber
\end{align}
Combining \eqref{thm:equiv05}, \thmref{thm:almost_diag} and \eqref{thm:equiv04} we obtain
\begin{equation*}
\|S_{\psi} f\|_{\ff_\infty^{s q}}=\|\{\left\langle f,\psi_\eta\right\rangle\}\|_{\ff_\infty^{s q}}
\le c\|\Omega^{(\infty,M)}h\|_{\ff_\infty^{s q}}
\le c\|h\|_{\ff_\infty^{s q}}\le c \|f\|_{\cF_\infty^{s q}},
\end{equation*}
which proves the boundedness of $S_{\psi}$.
Further, if $f$ satisfies \eqref{VF-norm-infty0}, then \eqref{thm:equiv06} implies that $S_{\psi}f$ satisfies Definition~\ref{def:VTL_infty},
i.e. $S_{\psi}: \sepTL{\cF}{s}{q} \to \sepTL{\ff}{s}{q}$.

Finally, the identity $T_\psi\circ S_{\psi}= I$ on $\sepTL{\cF}{s}{q}$ is immediate from \eqref{discretize}.

In the case $q=\infty$ we establish the theorem with the same arguments replacing \thmref{thm:almost_diag} by \propref{almost_diag_infty}.
\end{proof}

\begin{rem}
If we replace $\sepTL{\cF}{s}{q}$ and $\sepTL{\ff}{s}{q}$ with $\cF_\infty^{sq}$ and $\ff_\infty^{sq}$, respectively, then
\thmref{F-norm-infty-equiv} will remain valid except for the type of convergence in \eqref{F-infty-calderon}.
Note that \thmref{thm:prop-frame} implies that for any $f\in\cF_\infty^{sq}$
the series in $f=\sum_{\xi\in \cX}\langle f, \psi_\xi\rangle \psi_\xi$ converges unconditionally in the $\cF_p^{sq}$-norm for every $p<\infty$.
The proof is carried out along the same lines if we start it with $h\in\ff_\infty^{sq}$.
\end{rem}

In \cite{IP2} we constructed a system $\Theta=\{\theta_\xi : \xi\in\cX\}$
whose elements $\theta_\xi$ are linear combinations of a fixed number of shifts of the Newtonian kernel.
Also, it is shown in \cite{IP2} that $\Theta$ and its dual system $\tilde{\Theta}$
form a pair of dual frames for the spaces $\cB_p^{s q}$ and $\cF_p^{sq}$ when $p<\infty$.
In the following theorem we show that $\Theta$ and $\tilde{\Theta}$ also form
a pair of dual frames for the separable Triebel-Lizorkin spaces $\sepTL{\cF}{s}{q}$.

\begin{thm}\label{thm:frame-infty}
Assume $d\ge 2$, $A>1$, and
let $\Theta=\{\theta_\xi\}_{\xi\in\cX}$ be the real-valued system constructed in \cite[(6.36)--(6.37)]{IP2} with parameters
\begin{equation}\label{cond-K}
K \ge \left\lceil Ad\right\rceil, K\in 2\NN, \quad M = K+d.
\end{equation}
If the constant $\gamma_0$ in the construction of $\{\theta_\xi\}_{\xi\in\cX}$ is sufficiently small, then:

$(a)$ The synthesis operator $T_\theta$ defined by $T_\theta h:= \sum_{\xi\in \cX}h_\xi\theta_\xi$
on sequences of complex numbers $h=\{h_\xi\}_{\xi\in\cX}$ is bounded as a map
$T_\theta: \sepTL{\ff}{s}{q} \mapsto \sepTL{\cF}{s}{q}$
uniformly with respect to $(s, q)\in\QQ_\infty(A)$ with $\QQ_\infty(A)$ from \eqref{indices-infty}.

$(b)$ The operator
\begin{equation}\label{def:operator-T}
Tf:=\sum_{\xi\in\cX} \langle f, \psi_\xi\rangle \theta_\xi=T_{\theta}S_{\psi}f,
\end{equation}
is invertible on $\sepTL{\cF}{s}{q}$ and $T$, $T^{-1}$ are bounded on $\sepTL{\cF}{s}{q}$
uniformly with respect to $(s, q)\in\QQ_\infty(A)$.

$(c)$ For $(s, q)\in \QQ_\infty(A)$ the dual system $\tilde\Theta=\{\tilde\theta_\xi\}_{\xi\in\cX}$ consists of bounded linear functionals
on $\sepTL{\cF}{s}{q}$ defined by
\begin{equation}\label{def-f-dual-f}
\tilde\theta_\xi(f)=
\langle f, \tilde\theta_\xi\rangle
:= \sum_{\eta\in\cX}
\langle T^{-1}\psi_\eta, \psi_\xi\rangle \langle f, \psi_\eta\rangle
\quad \hbox{for}\;\; f\in \sepTL{\cF}{s}{q},
\end{equation}
with the series converging absolutely.
Also, the analysis operator
$$
S_{\tilde\theta}: \sepTL{\cF}{s}{q} \mapsto \sepTL{\ff}{s}{q},\quad
S_{\tilde\theta} = S_{\psi}T^{-1}T_{\psi}S_{\psi}=S_{\psi}T^{-1},
$$
is uniformly bounded with respect to $(s, q)\in\QQ_\infty(A)$ and
$T_\theta\circ S_{\tilde{\theta}}= I$ on $\sepTL{\cF}{s}{q}$.
Moreover, $\{\theta_\xi\}_{\xi\in\cX}$, $\{\tilde\theta_\xi\}_{\xi\in\cX}$
form a pair of dual frames for $\sepTL{\cF}{s}{q}$ in the following sense:
For any $f\in \sepTL{\cF}{s}{q}$
\begin{equation}\label{frame-B}
f=\sum_{\xi\in\cX} \langle f, \tilde\theta_\xi\rangle \theta_\xi
\quad\hbox{and}\quad
\|f\|_{\cF_\infty^{sq}} \sim \|\{\langle f, \tilde\theta_\xi\rangle\}\|_{\ff_\infty^{sq}},
\end{equation}
where the convergence is unconditional in $\sepTL{\cF}{s}{q}$.
\end{thm}

\begin{proof}
We shall derive this theorem as a consequence of \cite[Theorem 4.4]{IP2}.
We define the class $\YY$ of (quasi-)Banach spaces by $\YY=\{\sepTL{\cF}{s}{q}(\SS) : (s,q)\in\QQ_\infty(A)$\}
(see \propref{prop:TL-infty}~(a))
and the class $\YY_d$ of (quasi-) Banach sequence spaces by $\YY_d=\{\sepTL{\ff}{s}{q}(\cX) : (s,q)\in\QQ_\infty(A)$\}
(see \propref{prop:sep-TL-infty}~a)).
The respective (quasi-) triangle inequalities are satisfied with a constant $2^{1/q_*-1}$, $q_*=\min\{1,q\}$.
The old frame is $\{\psi\}$ and it is a self-dual frame for $\sepTL{\cF}{s}{q}\in\YY$.
Conditions {\bf A1} and {\bf A2} from \cite{IP2} are established in \thmref{F-norm-infty-equiv}.
Conditions {\bf A3}(a) and {\bf A3}(b) from \cite{IP2} are immediate from Definition~\ref{def:TL_infty}.

For $\sepTL{\cF}{s}{q}$ and $\sepTL{\ff}{s}{q}$ we also have:
\begin{itemize}
\item The set of all test functions, i.e. $C^\infty(\SS)$, is a dense subset of each $\sepTL{\cF}{s}{q}$ in view of \propref{prop:TL-infty}~(b);
\item  Compactly supported sequences belong to $\sepTL{\ff}{s}{q}$ and are dense in it (see \propref{prop:sep-TL-infty}~(b)), i.e. condition {\bf A3}(c) is fulfilled.
\end{itemize}

Finally, condition \cite[(4.6)]{IP2} is a consequence of \thmref{thm:almost_diag} and inequality \cite[(6.41)]{IP2}.
Thus, all assumptions of \cite[Theorem 4.4]{IP2} are satisfied by $\YY, \YY_d$, which completes the proof.
\end{proof}

\section{Identification of $\BMO(\SS)$ with $\cF_\infty^{02}(\SS)$}\label{sec:BMO-F-spaces}

The intimate relation between the spaces $\BMO$ with $\cF_\infty^{02}$ on $\SS$
will play a critical role in this study.

\begin{defn}\label{def_BMO}
The \emph{bounded mean oscillation} space $\BMO(\SS)$ is defined as the set of all functions $f\in L^1(\SS)$
such that
\begin{equation*}
\sup_{\substack{B=B(y,r)\\y\in\SS,~0<r\le \pi}}\frac{1}{|B|}\int_B|f(x)-\avg_{B} f| d\sigma(x) <\infty,
\quad \avg_{B} f:=\frac{1}{|B|}\int_B f(x) d\sigma(x).
\end{equation*}
The norm in $\BMO$ is given by
\begin{equation*}
\|f\|_{\BMO}:=|\avg_{\SS} f|+\sup_{\substack{B=B(y,r)\\y\in\SS,~0<r\le \pi}}\frac{1}{|B|}\int_B|f(x)-\avg_{B} f| d\sigma(x).
\end{equation*}
\end{defn}

\begin{defn}\label{def_VMO}
The \emph{vanishing mean oscillation} space $\VMO(\SS)$ is defined as the set of all functions $f\in\BMO(\SS)$
such that
\begin{equation*}
\lim_{\delta\to 0}\sup_{\substack{B=B(y,r)\\y\in\SS,~0<r\le \delta}}
\frac{1}{|B|}\int_B|f(x)-\avg_{B} f| d\sigma(x) =0,
\end{equation*}
$\VMO$ is equipped with the $\BMO$ norm.
\end{defn}

Here we state some well known properties of the spaces $\BMO$ and $\VMO$.

\begin{prop}\label{prop:BMO}
We have

$(a)$
$\BMO(\SS)$ and $\VMO(\SS)$ are Banach spaces.

$(b)$
$\VMO(\SS)$ is the closure of $C^\infty(\SS)$ in the $\BMO(\SS)$ norm.

$(c)$
$\VMO(\SS)$ is a \emph{separable space}, while $\BMO(\SS)$ is \emph{non-separable}.

\end{prop}

\begin{prop}\label{prop_BMO}
Let $0<q<\infty$. Then:

$(a)$ The condition
\begin{equation*}
\sup_{\substack{B=B(y,r)\\y\in\SS,~0<r\le \pi}}\frac{1}{|B|}\int_B|f(x)-\avg_B f|^q d\sigma(x)<\infty
\end{equation*}
is equivalent to $f\in\BMO(\SS)$ and
\begin{equation*}
\|f\|_{\BMO}\sim |\avg_{\SS} f|+\sup_{\substack{B=B(y,r)\\y\in\SS,~0<r\le \pi}}\bigg(\frac{1}{|B|}\int_B|f(x)-\avg_B f|^q d\sigma(x)\bigg)^{1/q}.
\end{equation*}

$(b)$ For $f\in\BMO(\SS)$ the condition
\begin{equation*}
\lim_{\delta\to 0}\sup_{\substack{B=B(y,r)\\y\in\SS,~0<r\le \delta}}\frac{1}{|B|}\int_B|f(x)-\avg_B f|^q d\sigma(x)=0
\end{equation*}
is equivalent to $f\in\VMO(\SS)$.
\end{prop}

\begin{proof}
Part (a) follows from the John-Nirenberg inequality \cite{JN}.

We observe that H\"{o}lder's inequality implies that $\big(|B|^{-1}\int_B|\cdot|^q \big)^{1/q}$
is monotone increasing function of $q$ and
\begin{equation*}
\|\cdot\|_{L^q}\le\|\cdot\|_{L^1}^{1/(2q-1)}\|\cdot\|_{L^{2q}}^{(2q-2)/(2q-1)},\quad q>1.
\end{equation*}
These two properties imply part (b).
\end{proof}

We refer the reader to \cite{FS, Stein} as general references for $\BMO$.

The purpose of this section is to clarify the relationship between $\BMO(\SS)$ and $\cF_\infty^{02}(\SS)$.
To this end we shall use that
\begin{equation}\label{rep-F}
\|f\|_{\cF_\infty^{02}(\SS)}
\sim \sup_{\xi\in\cX}\bigg(\frac{1}{|Q_\xi|}\sum_{Q_\eta\subset Q_\xi}|\langle f, \psi_\eta \rangle|^2\bigg)^{1/2},
\end{equation}
which follows by Theorem~\ref{F-norm-infty-equiv} and \eqref{def-f-space-infty2}.
We shall also use that
\begin{equation}\label{def-BMO}
\|f\|_{\BMO} \sim |\avg_{\SS} f|+\sup_{\substack{B=B(y,r)\\y\in\SS,~0<r\le \pi}}\bigg(\frac{1}{|B|}\int_B|f(x)-\avg_B f|^2d\sigma(x)\bigg)^{1/2},
\end{equation}
which follows from \propref{prop_BMO}~(a) with $q=2$.

\begin{thm}\label{thm:BMO-F}
$f\in \BMO(\SS)$ if and only if $f\in\cF_\infty^{02}(\SS)$ and
\begin{equation}\label{equiv-BMO-F}
\|f\|_{\BMO} \sim \|f\|_{\cF_\infty^{02}(\SS)}.
\end{equation}
Furthermore, $f\in \VMO(\SS)$ if and only if $f\in\sepTL{\cF}{0}{2}(\SS)$.
\end{thm}

\begin{proof}
(a)
Let $f\in \BMO(\SS)$.
We next show that $f\in\cF_\infty^{02}(\SS)$ and
\begin{equation}\label{est-F-BMO}
\|f\|_{\cF_\infty^{02}(\SS)} \le c\|f\|_{\BMO}.
\end{equation}

To simplify our notation set $\gamma:=\bar{\gamma}/2$ and
\begin{equation*}
R_\xi:= \frac{b}{b-1}\gamma b^{-n},
\quad\hbox{for}\quad \xi\in\cX_n,~n\in\NN,
\end{equation*}
see Theorem~\ref{nested_sets} and \remref{rem:3}.

First,
let $\xi\in \cX_n$, $n\ge 1$.
Using that $\bar{\gamma}\le 1$ and $b>3$ we get $R_\xi \le \pi/b$.
We also introduce the abbreviated notation
$B_\xi:= B(\xi, bR_\xi)$.
By \eqref{dyadic_eq:4} in Theorem~\ref{nested_sets}, \eqref{sph_cap}, and \eqref{dyadic_eq:6} we have
\begin{equation}\label{Q-B}
Q_\xi\subset B(\xi, R_\xi) \subset B_\xi
\quad\hbox{and} \quad
|Q_\xi|\sim |B(\xi, R_\xi)| \sim |B_\xi|\sim b^{-n(d-1)}.
\end{equation}
The fact that $\{\psi_\xi\}$ is a tight frame (see \cite[Theorem 5.2]{NPW1}) implies
\begin{align}\label{f-avg-B}
\|(f-\avg_{B_\xi}f)\ONE_{B_\xi}\|_{L^2}^2
&= \sum_{\eta\in\cX}|\langle (f-\avg_{B_\xi}f)\ONE_{B_\xi}, \psi_\eta\rangle|^2
\\
&\ge \sum_{Q_\eta\subset Q_\xi}|\langle (f-\avg_{B_\xi}f)\ONE_{B_\xi}, \psi_\eta\rangle|^2.\nonumber
\end{align}

Let $Q_\eta\subset Q_\xi$. Hence $\eta\in\cX_{n+m}$ for some $m\ge 0$.
Using the fact that $\int_\SS\psi_\eta(x)d\sigma(x)=0$
we have
\begin{align}\label{f-psi}
|\langle f, \psi_\eta\rangle|^2
&= |\langle f-\avg_{B_\xi}f, \psi_\eta\rangle|^2\nonumber
\\
&\le 2|\langle (f-\avg_{B_\xi}f)\ONE_{B_\xi}, \psi_\eta\rangle|^2 + 2|\langle (f-\avg_{B_\xi}f)\ONE_{\SS\setminus B_\xi}, \psi_\eta\rangle|^2
\\
&\le 2|\langle (f-\avg_{B_\xi}f)\ONE_{B_\xi}, \psi_\eta\rangle|^2 \nonumber
\\
&\qquad\qquad\qquad+ 2\Big(\int_{\SS\setminus B_\xi}|f(x)-\avg_{B_\xi}f||\psi_\eta(x)|d\sigma(x)\Big)^2. \nonumber
\end{align}
To estimate the integral above we shall use that for any $M>d-1$
\begin{equation}\label{local-psi}
|\psi_\eta(x)| \le cb^{(n+m)(d-1)/2}(1+b^{n+m}\rho(x, \eta))^{-M},
\quad \eta\in \cX_{n+m},
\end{equation}
which is \eqref{local-needlet-0}.
We use \eqref{local-psi} and \eqref{Q-B} to obtain
\begin{align}\label{f-avg}
\int_{\SS\setminus B_\xi}&|f(x)-\avg_{B_\xi}f||\psi_\eta(x)|d\sigma(x) \nonumber
\\
&
= \sum_{\nu\ge 1} \int_{B(\xi, b^{\nu+1}R_\xi)\setminus B(\xi, b^{\nu}R_\xi)}|f(x)-\avg_{B_\xi}f||\psi_\eta(x)|d\sigma(x)
\\
&
\le c\sum_{\nu\ge 1} \frac{b^{(n+m)(d-1)/2}}{b^{(m+\nu)M}}\int_{B(\xi, b^{\nu+1}R_\xi)\setminus B(\xi, b^{\nu}R_\xi)}|f(x)-\avg_{B_\xi}f|d\sigma(x). \nonumber
\end{align}
Here we used that for $x\in B(\xi, b^{\nu+1}R_\xi)\setminus B(\xi, b^{\nu}R_\xi)$ and $\eta\in Q_\eta\subset Q_\xi$ we have
$$
b^{n+m}\rho(x, \eta) \ge cb^{n+m+\nu} R_\xi \ge cb^{m+\nu}.
$$
Clearly,
\begin{multline}\label{f-avg-2}
\int_{B(\xi, b^{\nu+1}R_\xi)\setminus B(\xi, b^{\nu}R_\xi)}|f(x)-\avg_{B_\xi}f|d\sigma(x)
\\
\le \int_{B(\xi, b^{\nu+1}R_\xi)}\big(|f(x)-\avg_{B(\xi, b^{\nu+1}R_\xi)}f|
\\
+ |\avg_{B(\xi, b^{\nu+1}R_\xi)}f-\avg_{B(\xi, bR_\xi)}f|\big)d\sigma(x).
\end{multline}
Further,
\begin{align}\label{f-avg-3}
\int_{B(\xi, b^{\nu+1}R_\xi)}|f(x)-\avg_{B(\xi, b^{\nu+1}R_\xi)}f| d\sigma(x)
&\le |B(\xi, b^{\nu+1}R_\xi)| \|f\|_{\BMO}
\\
&\le cb^{(\nu-n)(d-1)} \|f\|_{\BMO}. \nonumber
\end{align}
Also, for any $R>0$
\begin{align*}
|\avg_{B(\xi, bR)}f-\avg_{B(\xi, R)}f|
&\le \frac{1}{|B(\xi, R)|} \int_{B(\xi, R)}|f(x)-\avg_{B(\xi, bR)}f| d\sigma (x)
\\
& \le \frac{c}{|B(\xi, bR)|} \int_{B(\xi, bR)}|f(x)-\avg_{B(\xi, bR)}f| d\sigma (x)
\\
&\le c\|f\|_{\BMO}
\end{align*}
and using a telescopic sum we obtain
\begin{align*}
|\avg_{B(\xi, b^{\nu+1}R_\xi)}f-\avg_{B(\xi, bR_\xi)}f|
\le c\nu \|f\|_{\BMO}.
\end{align*}
Hence,
\begin{align*}
\int_{B(\xi, b^{\nu+1}R_\xi)}|\avg_{B(\xi, b^{\nu+1}R_\xi)}f-\avg_{B(\xi, bR_\xi)}f|d\sigma(x)
\le c\nu b^{(\nu-n)(d-1)} \|f\|_{\BMO}.
\end{align*}
This combined with \eqref{f-avg-2} and \eqref{f-avg-3} implies
\begin{align*}
\int_{B(\xi, b^{\nu+1}R_\xi)\setminus B(\xi, b^{\nu}R_\xi)}|f(x)-\avg_{B_\xi}f|d\sigma(x)
\le c\nu b^{(\nu-n)(d-1)} \|f\|_{\BMO}.
\end{align*}
Using this in \eqref{f-avg} we obtain
\begin{align*}
\int_{\SS\setminus B_\xi}&|f(x)-\avg_{B_\xi}f||\psi_\eta(x)|d\sigma(x)
\\
&\le c\sum_{\nu\ge 1} \frac{b^{(n+m)(d-1)/2}}{b^{(m+\nu)M}}\nu b^{(\nu-n)(d-1)} \|f\|_{\BMO}
\\
&\le cb^{-m(M-\frac{d-1}{2})}b^{-n(d-1)/2}\|f\|_{\BMO}\sum_{\nu\ge 1}\nu b^{-\nu(M-d+1)}
\\
&\le cb^{-m(M-\frac{d-1}{2})}|Q_\xi|^{1/2}\|f\|_{\BMO}.
\end{align*}
This and \eqref{f-psi} yield
\begin{align*}
|\langle f, \psi_\eta\rangle|^2
\le 2|\langle (f-\avg_{B_\xi}f)\ONE_{B_\xi}, \psi_\eta\rangle|^2
+ cb^{-2m(M-\frac{d-1}{2})}|Q_\xi|\|f\|_{\BMO}^2.
\end{align*}
Therefore,
\begin{align}\label{f-BMO}
&\frac{1}{|Q_\xi|}\sum_{Q_\eta\subset Q_\xi} |\langle f, \psi_\eta \rangle|^2
\le \frac{c}{|B_\xi|}\sum_{Q_\eta\subset Q_\xi} |\langle (f-\avg_{B_\xi}f)\ONE_{B_\xi}, \psi_\eta\rangle|^2 \nonumber
\\
& \hspace{2in}+ c\sum_{m\ge 0}\sum_{Q_\eta\subset Q_\xi, \eta\in\cX_{n+m}}b^{-2m(M-\frac{d-1}{2})}\|f\|_{\BMO}^2
\\
&\le \frac{c}{|B_\xi|}\|f-\avg_{B_\xi}f\|_{L^2(B_\xi)}^2
+ c\sum_{m\ge 0}b^{m(d-1)}b^{-2m(M-\frac{d-1}{2})}\|f\|_{\BMO}^2
\le c \|f\|_{\BMO}^2.\nonumber
\end{align}
Above we used \eqref{f-avg-B}, \eqref{number_of_children} and that $M>d-1$.

Second, let $\xi$ be the only element of $\cX_0$. Then $\langle f, \psi_\xi \rangle=|\SS|\avg_{\SS}f$ and \eqref{f-BMO} imply
\begin{equation*}
\frac{1}{|Q_\xi|}\sum_{Q_\eta\subset Q_\xi} |\langle f, \psi_\eta \rangle|^2
=\frac{1}{|Q_\xi|}|\langle f, \psi_\xi \rangle|^2
+\sum_{\zeta\in\cX_1}\frac{|Q_\zeta|}{|Q_\xi|}\frac{1}{|Q_\zeta|}\sum_{Q_\eta\subset Q_\zeta} |\langle f, \psi_\eta \rangle|^2
\le c \|f\|_{\BMO}^2.
\end{equation*}
This, \eqref{f-BMO} and \eqref{rep-F} imply \eqref{est-F-BMO}.

\smallskip

(b) Assume $f\in\cF_\infty^{02}(\SS)$.
We will show that $f\in \BMO(\SS)$ and
\begin{equation}\label{est-BMO-F}
\|f\|_{\BMO} \le c\|f\|_{\cF_\infty^{02}(\SS)}.
\end{equation}
Let $B:=B(y, R)$, $y\in\SS$, $0<R\le \pi$, be a spherical cap on $\SS$.
Choose $n\in\NN$ so that $\pi b^{-n}<R\le \pi b^{-n+1}$.
Denote
\begin{equation*}
\cU_n:=\{\xi\in\cX_n: Q_\xi\cap bB\ne\emptyset\}
\quad\hbox{and}\quad
\cY_n:= \cup_{j\ge n}\cX_j.
\end{equation*}
From \eqref{dyadic_eq:4}, \eqref{sph_cap}, and \eqref{dyadic_eq:6} it follows that
$\# \cU_n\le c$ with a constant $c>0$ independent of $n$.

Consider the following representation of $f$:
\begin{align*}
f &=\sum_{\xi\in \cU_n}\sum_{Q_\eta\subset Q_\xi}\langle f, \psi_\eta\rangle \psi_\eta
+ \sum_{\xi\in \cX_n\backslash\cU_n}\sum_{Q_\eta\subset Q_\xi}\langle f, \psi_\eta\rangle \psi_\eta
+ \sum_{\eta\in \cX\setminus \cY_n}\langle f, \psi_\eta\rangle \psi_\eta
\\
&=: F_1+ F_2 +F_3,
\end{align*}
where the convergence is in $L^2(\SS)$.
In the following we shall show that $F_1\in L^2(B)$ and $F_2\in L^\infty(B)$.
Of course, $F_3\in C^\infty(B)$ because $F_3$ is a spherical polynomial.

We first estimate $\frac{1}{|B|}\int_B |F_1(x)|^2d\sigma(x)$.
For this we need the following

\begin{lem}\label{lem:L2-norm}
If $\{a_\eta\}\in \ell^2(\cX)$, then
\begin{equation}\label{est-L2}
\Big\|\sum_{\eta\in\cX} a_\eta\psi_\eta\Big\|_2
\le c\Big(\sum_{\eta\in\cX} |a_\eta|^2\Big)^{1/2}.
\end{equation}
\end{lem}

Inequality \eqref{est-L2} follows immediately from the boundedness of the synthesis operator
$T_\psi: \ff_2^{02} \to \cF_2^{02}$ established in \thmref{thm:F-Bnorm-equiv}~(b)
and the identifications $\ff_2^{02}=\ell^2$ and $\cF_2^{02}(\SS)=L^2(\SS)$ (see \cite[Proposition 4.3]{NPW2}) with equivalent norms.

In light of Lemma~\ref{lem:L2-norm}, \eqref{dyadic_eq:6}, and \eqref{rep-F} we have
\begin{align}\label{F1}
\frac{1}{|B|}\int_B |F_1(x)|^2d\sigma(x)
&\le \frac{1}{|B|}\int_\SS |F_1(x)|^2d\sigma(x)
\le \frac{c}{|B|}\sum_{\xi\in\cU_n}\sum_{Q_\eta\subset Q_\xi}|\langle f, \psi_\eta\rangle|^2 \nonumber
\\
&\le c \sum_{\xi\in\cU_n}\frac{1}{|Q_\xi|}\sum_{Q_\eta\subset Q_\xi}|\langle f, \psi_\eta\rangle|^2
\le c \|f\|_{\cF^{02}_\infty}^2.
\end{align}

We now estimate $|F_2(x)|$ for $x\in B$.
For this we shall use that
\begin{equation*}
|\langle f, \psi_\eta\rangle|^2 \le c|Q_\eta|\|f\|_{\cF^{02}_\infty}^2
\quad\hbox{and}\quad
|\psi_\eta(x)| \le \frac{c|Q_\eta|^{-1/2}}{(1+b^j\rho(x, \eta))^M},\quad \eta\in \cX_j,
\end{equation*}
which follow from \eqref{rep-F}, \eqref{local-needlet-0}, and \eqref{dyadic_eq:6}.
We choose $M=d$ and get
\begin{align*}
|F_2(x)| &\le \sum_{\xi\in \cX_n\backslash\cU_n}\sum_{Q_\eta\subset Q_\xi}|\langle f, \psi_\eta\rangle| |\psi_\eta(x)|
\\
&\le c \|f\|_{\cF^{02}_\infty}\sum_{j\ge n}\sum_{\substack{\eta\in\cX_j\\ Q_\eta\cap bB=\emptyset}}\frac{1}{(1+b^j\rho(x, \eta))^M}.
\end{align*}
Using $\rho(x,\zeta)\ge (b-1)R$ for $x\in B$ and $\zeta\in Q_\eta$ with $Q_\eta\cap bB=\emptyset$, \eqref{dyadic_eq:6},
and \eqref{eq:conv_2} with $N=b^j$ and $r=\pi (b-1)b^{-n}$ we have for the last sum above
\begin{multline*}
\sum_{\substack{\eta\in\cX_j\\ Q_\eta\cap bB=\emptyset}}\frac{1}{(1+b^j\rho(x, \eta))^M}
\le c\sum_{\substack{\eta\in\cX_j\\ Q_\eta\cap bB=\emptyset}}\frac{1}{|Q_\eta|}\int_{Q_\eta}\frac{1}{(1+b^j\rho(x, \zeta))^M}d\sigma(\zeta)\\
\le c\int_{\SS\setminus (b-1)B} \frac{b^{j(d-1)}}{(1+b^j\rho(x, \zeta))^M} d\sigma(\zeta)
\le c b^{(n-j)(M-d+1)}
\end{multline*}
and hence
\begin{equation}\label{F2}
|F_2(x)| \le c\|f\|_{\cF^{02}_\infty},\quad x\in B.
\end{equation}

We finally estimate $|F_3(x)-F_3(y)|$ for $x\in B$ using that
$|\langle f, \psi_\eta\rangle|^2 \le |Q_\eta|\|f\|_{\cF^{02}_\infty}^2$
and
\begin{equation*}
|\psi_\eta(x)-\psi_\eta(y)| \le \frac{c|Q_\eta|^{-1/2}b^j\rho(x, y)}{(1+b^j\rho(x, \eta))^M},
\;\; \eta\in \cX_j,
\quad \rho(x, y)\le \pi b^{-j}, \;\; M>d-1,
\end{equation*}
which follows from \eqref{local-Lip-1} in \thmref{thm:localization} with $\Lambda_{b^j}=\Psi_j$, \eqref{def-psi-xi}, and \eqref{cubature_w}.
We obtain for $\rho(x,y)\le R\le \pi b^{-n}$
\begin{align}\label{F3}
|F_3(x)-F_3(y)| &\le \sum_{\eta\in \cX\setminus \cY_n}|\langle f, \psi_\eta\rangle||\psi_\eta(x)-\psi_\eta(y)| \nonumber
\\
&\le c\|f\|_{\cF^{02}_\infty}\sum_{j<n}\sum_{\eta\in\cX_j}\frac{b^j\rho(x, y)}{(1+b^j\rho(x, \eta))^M}
\\
&\le c\|f\|_{\cF^{02}_\infty}\sum_{j<n}b^{j-n}\sum_{\eta\in\cX_j}\frac{1}{(1+b^j\rho(x, \eta))^M}
\le c\|f\|_{\cF^{02}_\infty}, \nonumber
\end{align}
where \lemref{lem:Besov_1} with $m=0$ is used for the last inequality.

We use \eqref{F1}--\eqref{F3} to obtain
\begin{align*}
&\bigg(\frac{1}{|B|}\int_B |f(x)-F_3(y)|^2d\sigma(x)\bigg)^{1/2}
\le \bigg(\frac{1}{|B|}\int_B |F_1(x)|^2d\sigma(x)\bigg)^{1/2}
\\
&\qquad + \bigg(\frac{1}{|B|}\int_B |F_2(x)|^2d\sigma(x)\bigg)^{1/2}
+ \bigg(\frac{1}{|B|}\int_B |F_3(x)-F_3(y)|^2d\sigma(x)\bigg)^{1/2}
\le c\|f\|_{\cF^{02}_\infty}.
\end{align*}
This readily implies that
\begin{equation*}
\bigg(\frac{1}{|B|}\int_B |f(x)-\avg_B f|^2d\sigma(x)\bigg)^{1/2}
\le 2\bigg(\frac{1}{|B|}\int_B |f(x)-F_3(y)|^2d\sigma(x)\bigg)^{1/2} \le c\|f\|_{\cF^{02}_\infty},
\end{equation*}
which together with $|\avg_{\SS} f|^2 \le |\SS|\|f\|_{\cF^{02}_\infty}^2$
yields \eqref{est-BMO-F}.

\smallskip

The identification $\sepTL{\cF}{0}{2}(\SS)=\VMO(\SS)$ follows
from \eqref{equiv-BMO-F}, \propref{prop:BMO}~(b), and \propref{prop:TL-infty}~(b).
The proof of Theorem~\ref{thm:BMO-F} is complete.
\end{proof}

\begin{rem}
The parameter $q$ plays different roles in the definitions of $\BMO(\SS)$ and Triebel-Lizorkin spaces with $p=\infty$.
All $q$'s, $0<q<\infty$, generate the same space $\BMO(\SS)$ according to \propref{prop_BMO}.
On the other hand $\cF^{sq_1}_\infty\ne\cF^{sq_2}_\infty$ if $q_1\ne q_2$.
\end{rem}

\section{Harmonic BMO space on the ball}\label{sec:harmonic-BMO}

Denote by $\HHH=\HHH(B^d)$ the set of all harmonic functions on the unit ball $B^d$ in $\RR^d$.

It is convenient to define the harmonic Besov and Triebel-Lizorkin spaces on $B^d$
by using their expansion in solid spherical harmonics.
As in Subsection~\ref{s1_3} let $\{Y_{\kk j}: j=1, \dots, N(\kk, d)\}$ be a real-valued orthonormal basis for $\cH_\kk$.
The harmonic coefficients of $U\in\HHH(B^d)$ are defined by
\begin{equation*}
\bbb_{\kk \nu}(U):= \frac{1}{a^\kk}\int_\SS U(a\eta)Y_{\kk \nu}(\eta) d\sigma(\eta)
\end{equation*}
for some $0<a<1$. It is an important observation that the coefficients
are independent of $a$ for all $0<a<1$.
This implies the representation
\begin{equation}\label{rep-U-3a}
U(r\xi) = \sum_{\kk=0}^\infty \sum_{\nu=1}^{N(\kk, d)} \bbb_{\kk \nu}(U)r^\kk Y_{\kk \nu}(\xi),
\quad 0\le r<1, ~\xi\in \SS,
\end{equation}
where the convergence is absolute and uniform on every compact subset of $B^d$.

In view of $(\ref{rep-U-3a})$ to any $U\in \HHH(B^d)$
with coefficients $\{\bbb_{\kk \nu}(U)\}$ with polynomial growth
there corresponds a boundary value distribution $f_U\in \cS'(\SS)$
defined by
\begin{equation}\label{rep-f-33}
f_U(\xi):= \sum_{\kk=0}^\infty \sum_{\nu=1}^{N(\kk, d)} \bbb_{\kk \nu}(U) Y_{\kk \nu}(\xi),\quad \xi\in\SS,
\quad \hbox{$($convergence in $\cS'$$)$}.
\end{equation}
Conversely, to any distribution $f\in \cS'(\SS)$ with coefficients $\bbb_{\kk \nu}(f):=\langle f, Y_{\kk \nu}\rangle$
there corresponds a harmonic function $U_f\in \HHH(B^d)$
$($the harmonic extension of $f$ to $B^d$$)$, defined by
\begin{equation}\label{rep-U-33}
U_f(x) = \sum_{\kk=0}^\infty \sum_{\nu=1}^{N(\kk, d)} \bbb_{\kk \nu}(f)|x|^\kk Y_{\kk \nu}\Big(\frac{x}{|x|}\Big)=\left\langle f, P(\cdot,x) \right\rangle,
\quad |x|<1,
\end{equation}
where the series converges uniformly on every compact subset of $B^d$.
Here $P(y,x)$ is the Poisson kernel, see \eqref{Poisson} and \eqref{Pk}.

In \cite{IP} we define and study the harmonic Besov spaces $B^{sq}_p(\HHH)$ ($s\in\RR$, $0<p,q\le \infty$),
Triebel-Lizorkin spaces $F^{sq}_p(\HHH)$ ($s\in\RR$, $0<q\le \infty$, $0<p<\infty$) and Hardy spaces $\HHH^p(B^d)$ ($0<p\le \infty$).
Some identification results from \cite{IP} are given in the following theorems.

\begin{thm}\label{thm:equiv-norms-F}
Let $s\in \RR$, $0<p < \infty$, $0<q\le \infty$.
A harmonic function $U\in F^{s q}_p(\HHH)$
if and only if its boundary value distribution $f_U$
defined by \eqref{rep-f-33} belongs to $\cF^{sq}_p(\SS)$,
moreover,
$\|U\|_{F^{s q}_p(\HHH)} \sim \|f_U\|_{\cF^{s q}_p(\SS)}$.
Also, a distribution $f\in \cF^{sq}_p(\SS)$ if and only if its harmonic extension $U_f$ defined by \eqref{rep-U-33} belongs to $F^{s q}_p(\HHH)$,
moreover,
$\|f\|_{\cF^{s q}_p(\SS)} \sim \|U_f\|_{F^{s q}_p(\HHH)}$.
\end{thm}

\begin{thm}\label{thm:equiv-norms-B}
Let $s\in \RR$, $0<p, q\le \infty$.
A harmonic function $U\in B^{s q}_p(\HHH)$
if and only if its boundary value distribution $f_U$
defined by \eqref{rep-f-33} belongs to $\cB^{sq}_p(\SS)$,
moreover,
$\|U\|_{B^{s q}_p(\HHH)} \sim \|f_U\|_{\cB^{s q}_p(\SS)}$.
Also, a distribution $f\in \cB^{sq}_p(\SS)$ if and only if its harmonic extension $U_f$ defined by \eqref{rep-U-33} belongs to $B^{s q}_p(\HHH)$,
moreover,
$\|f\|_{\cB^{s q}_p(\SS)} \sim \|U_f\|_{B^{s q}_p(\HHH)}$.
\end{thm}

We now introduce the harmonic BMO space $\BMOH$ on $B^d$.
We will use the following notation: For any spherical cap $B\subset \SS$ of radius $r$ we set
\begin{equation}\label{def:TB}
T(B):= \{x\in B^d: x=ty, y\in B, (1-r)_+\le t<1\}.
\end{equation}

\begin{defn}\label{def:BMOH}
The space $\BMOH:= \BMOH(B^d)$ is defined as the set of all harmonic functions $U\in \HHH(B^d)$
such that
\begin{equation}\label{BMOH}
\|U\|_{\BMOH}:=|U(0)|+\sup_{B\subset \SS}\bigg(\frac{1}{|B|}\int_{T(B)}|\nabla U(x)|^2(1-|x|)dx\bigg)^{1/2} <\infty,
\end{equation}
where the $\sup$ is over all balls $B$ on $\SS$.

Furthermore, the harmonic $\VMO$ space $\VMOH:= \VMOH(B^d)$ is defined
as the set of all harmonic functions $U\in \BMOH(B^d)$ such that
\begin{equation}\label{VMOH}
\sup_{\substack{B=B(y,r)\\y\in\SS,~0<r\le \delta}}\bigg(\frac{1}{|B|}\int_{T(B)}|\nabla U(x)|^2(1-|x|)dx\bigg)^{1/2} \to 0
\quad\hbox{as}\quad \delta\to 0.
\end{equation}
\end{defn}

This definition is justified by the following

\begin{thm}\label{thm:BMOH}
A function $f\in \BMO(\SS)$ if and only if its harmonic extension $U_f\in \BMOH(B^d)$ $($see \eqref{rep-U-33}$)$ and
\begin{equation}\label{BMO-BMOH}
\|f\|_{\BMO} \sim \|U_f\|_{\BMOH}.
\end{equation}
Moreover, $f\in \VMO(\SS)$ if and only $U_f\in \VMOH(B^d)$.
\end{thm}

This theorem is classical by now in the theory of $H^1$ and $\BMO$ spaces.
It simply asserts that $f\in \BMO(\SS)$ if and only if
\begin{equation*}
d\mu_f:= |\nabla U_f(x)|^2(1-|x|)dx
\end{equation*}
is a Carleson measure on $B^d$ with norm $\|d\mu_f\|_{\cC}$ equivalent to $\|f\|_{\BMO}$.

In dimension $d=2$ the proof of Theorem~\ref{thm:BMOH} is contained in the proof of Theorem~3.4, Chapter VI, in \cite{Garnett}.
In the case of $\RR^{n+1}_+$ (instead of $B^d$) this is Theorem~3, Sec. II, in~\cite{FS}.
We do not present here a proof of \thmref{thm:BMOH} in the case $d\ge 3$,
since it does not differ essentially from the corresponding proof for $\RR^{n+1}_+$ in \cite{FS}.
We omit the further details.

\section{Nonlinear approximation in $\BMO$}\label{sec:approximation}

In this section we shall utilize the spaces and frames developed in previous sections to obtain our main results.
In particular, the machinery of highly localized frames for Besov spaces and $\BMO$
whose elements are linear combinations of a fixed number of shifts of the Newtonian kernel
will play a critical role.

\subsection{Main results on approximation from shifts of the Newtonian kernel}

We begin by introducing our approximation tool.
For any $n\ge 1$ define
\begin{equation}\label{def:N-n-d}
\NNN_n:=\Big\{G: G(x) =a_0+\sum_{\nu=1}^{n} \frac{a_\nu}{|x-y_\nu|^{d-2}}, \; |y_\nu|>1, a_\nu\in \CC\Big\},
\;\;\hbox{if $d>2$},
\end{equation}
and
\begin{equation}\label{def:N-n-2}
\NNN_n:=\Big\{G: G(x) =a_0+\sum_{\nu=1}^{n} a_\nu\ln \frac{1}{|x-y_\nu|}, \; |y_\nu|>1, a_\nu\in \CC\Big\},
\;\;\hbox{if $d=2$}.
\end{equation}
Observe that the points $\{y_\nu\}$ above may vary with $G$ and hence $\NNN_n$ is nonlinear.

Given $U\in \VMOH$ we define
\begin{equation}\label{def-En-H}
E_n(U)_{\BMOH}:=\inf_{G\in\NNN_n}\|U-G\|_{\BMOH},
\end{equation}
which we term
{\em best $n$-term approximation of $U$ from shifts of the Newtonian kernel in the harmonic $\BMO$ space}.

The approximation of harmonic functions $U$ on the ball $B^d\subset \RR^d$ is intimately connected
to the approximation of their boundary values $f_U$ on $\SS$.
For any $f \in \VMO(\SS)$ we set
\begin{equation*}
E_n(f)_{\BMO}:=\inf_{G\in\NNN_n}\|f-G|_{\SS}\|_{\BMO}.
\end{equation*}

Here we come to our main result. We adhere to the notation from the previous sections.

\begin{thm}[Jackson estimate]\label{thm:Jackson_VMOH}
If $U\in B^{s\tau}_{\tau}(\HHH)$, $1/\tau=s/(d-1)$, $s>0$,
then $U\in \VMOH(B^d)$ and for $n\ge 1$
\begin{equation}\label{jackson-VMOH}
E_n(U)_{\BMOH} \le cn^{-s/(d-1)}\|U\|_{B^{s\tau}_{\tau}(\HHH)},
\end{equation}
where the constant $c>0$ depends only on $s$, $d$.
\end{thm}

This theorem follows at once by harmonic extension (see \thmref{thm:BMOH}, \thmref{thm:equiv-norms-B})
from the following

\begin{thm}[Jackson estimate]\label{thm:Jackson_VMO}
If $f\in \cB^{s\tau}_{\tau}(\SS)$, $1/\tau=s/(d-1)$, $s>0$,
then $f\in \VMO(\SS)$ and for $n\ge 1$
\begin{equation}\label{jackson}
E_n(f)_{\BMO} \le cn^{-s/(d-1)}\|f\|_{\cB^{s\tau}_{\tau}(\SS)},
\end{equation}
where the constant $c>0$ depends only on $s$, $d$.
\end{thm}

Our development in this article depends heavily on the fact that
the space $\VMO$ fits in the Littlewood-Paley theory and
as shown in Theorem~\ref{thm:BMO-F} coincides with the separable Triebel-Lizorkin space $\sepTL{\cF}{0}{2}(\SS)$
and has a frame decomposition.
To put it in perspective, we consider next the nonlinear $n$-term approximation from shifts of the Newtonian kernel
in the more general Triebel-Lizorkin spaces $\cF_\infty^{0q}(\SS)$, $0<q\le\infty$.
Given $f\in \cF_\infty^{0q}(\SS)$ we define
\begin{equation}\label{def:best-app}
E_n(f)_{\cF_\infty^{0q}}:=\inf_{G\in\NNN_n}\|f-G|_{\SS}\|_{\cF_\infty^{0q}}.
\end{equation}
This sort of approximation makes sense only if $\lim_{n\to\infty}E_n(f)_{\cF_\infty^{0q}}=0$.
Therefore, we are interested in the approximation defined in \eqref{def:best-app} only of 
functions $f$ from the separable space $\sepTL{\cF}{0}{q}(\SS)$.

The following result coupled with Theorem~\ref{thm:BMO-F} implies readily Theorem~\ref{thm:Jackson_VMO}.

\begin{thm}[Jackson estimate]\label{thm:Jackson_TL}
If $f\in \cB^{s\tau}_{\tau}(\SS)$, $1/\tau=s/(d-1)$, $s>0$, $0<q\le\infty$,
then $f\in \sepTL{\cF}{0}{q}(\SS)$ and for $n\ge 1$
\begin{equation}\label{jackson_TL}
E_n(f)_{\cF_\infty^{0q}} \le cn^{-s/(d-1)}\|f\|_{\cB^{s\tau}_{\tau}},
\end{equation}
where the constant $c>0$ depends only on $s$, $d$, $q$.
\end{thm}

The fundamental fact that the spaces $\cB^{s\tau}_{\tau}$ and $\sepTL{\cF}{0}{q}$ (in particular $\VMO$)
have frame decompositions allows to reduce the proof of the above theorem to approximation in sequence spaces
that we consider next. The final step in the proof of Theorem~\ref{thm:Jackson_TL} will be given thereafter.

\subsection{Jackson and Bernstein estimates for nonlinear approximation in $\ff^{0 q}_\infty$}

Here we consider nonlinear $n$-term approximation from finitely supported sequences
in the spaces $\ff^{0 q}_\infty$ in a general setting.

\begin{defn}\label{nested_structure}
Let $\cX$ be a countable multilevel index set, where either $\cX=\cup_{n=0}^\infty\cX_n$ or $\cX=\cup_{n=-\infty}^\infty\cX_n$.
Let $\fD$ be a measurable domain. To each $\xi\in\cX$ we associate an open set $Q_\xi\subset\fD$.
We call $\{Q_\xi : \xi\in\cX\}$ a \emph{nested structure associate with $\cX$} if there exists constant $\kappa\ge 1$ such that
\begin{align}
&|\fD\backslash\bigcup_{\xi\in\cX_n}Q_\xi|=0\quad\forall n\in\NN_0 \; (n\in\ZZ);\label{dyad_eq:1}\\
&\mbox{If}~\eta\in\cX_m, \xi\in\cX_n~\mbox{and}~m\ge n~\mbox{then either}~Q_\eta\subset Q_\xi~\mbox{or}~Q_\eta\cap Q_\xi=\emptyset;\label{dyad_eq:2}\\
&\mbox{For every}~\xi\in\cX_n,~m< n~\mbox{there exists unique}~\eta\in\cX_m~\mbox{such that}~Q_\xi\subset Q_\eta;\label{dyad_eq:3}\\
&|Q_\eta|\le \kappa|Q_\zeta| \quad\forall\eta,\zeta\in\cX_n, \forall n\in\NN_0 (n\in\ZZ);\label{dyad_eq:4a}\\
&\mbox{Every $Q_\xi$, $\xi\in\cX_n$, has at least two children, i.e. there exist $\eta,\zeta\in\cX_{n+1}$} \label{dyad_eq:6a}\\
&\mbox{such that $Q_\eta\subset Q_\xi$, $Q_\zeta\subset Q_\xi$, $\eta\ne\zeta$.}\nonumber
\end{align}
In the case when $\cX=\cup_{n=0}^\infty\cX_n$ we also assume that $\#\cX_0=1, \xi_0\in\cX_0$ and $Q_{\xi_0}$ 
coincides with the interior of $\fD$.
If $\cX$ is of the form $\cX=\cup_{n=-\infty}^\infty\cX_n$, then we assume that there exist
finitely many mutually disjoint open sets $\fD^1,\dots,\fD^J\subset\fD$
with the properties:
$(i)$ $|\fD\setminus\cup_{j=1}^J\fD^j|=0$;
$(ii)$ for every $\xi\in\cX$ there is $j, 1\le j\le J$, such that $Q_\xi\subset \fD^j$;
$(iii)$ if the sets $Q_\xi,Q_\eta$ are contained in $\fD^j$ for some $j=1,\dots,J$,
then there exists $Q_\zeta\subset\fD^j$ such that $Q_\xi\subset Q_\zeta$, $Q_\eta\subset Q_\zeta$.
\end{defn}

\begin{rem}
Conditions \eqref{dyad_eq:1}--\eqref{dyad_eq:3} imply that for every $\xi\in\cX_n$ we have
\begin{equation}\label{dyad_eq:7}
|Q_\xi|=\sum_{\substack{\eta\in\cX_{n+j}\\ Q_\eta\subset Q_\xi}}|Q_\eta|,
\quad \forall j\in\NN.
\end{equation}

Conditions \eqref{dyad_eq:4a}--\eqref{dyad_eq:6a} imply that there exists $\lambda\in(0,1)$, namely $\lambda=\kappa/(\kappa+1)$, such that
\begin{equation}\label{dyad_eq:6}
|Q_\eta|\le \lambda |Q_\xi|\quad\forall n\in\NN_0 (n\in\ZZ), \xi\in\cX_n, \eta\in\cX_{n+1}, Q_\eta\subset Q_\xi.
\end{equation}
\end{rem}

\begin{rem}
The open dyadic cubs in $\fD=\RR^d$ form such a nested structure with $\cX=\cup_{n=-\infty}^\infty\cX_n$.
The open dyadic sub-cubs of a fixed dyadic cub $\fD$, $\fD\subset\RR^d$, also form such a nested structure with $\cX=\cup_{n=0}^\infty\cX_n$.
In both cases $\kappa=1$, $\lambda=2^{-d}$.

For the open dyadic cubs structure in $\fD=\RR^d$ we have $J=2^d$.
For example, for $d=1$ we have $\fD^1=(0,\infty), \fD^2=(-\infty,0)$ and for $d=2$ we have
$\fD^1=(0,\infty)\times(0,\infty), \fD^2=(-\infty,0)\times(0,\infty), \fD^3=(-\infty,0)\times(-\infty,0), \fD^4=(0,\infty)\times(-\infty,0)$.

The construction from \thmref{nested_sets} is another example of 
a nested structure associated with a given index set $\cX$ consisting of maximal $\delta_n$-nets on $\SS$.
\end{rem}

Recall that by definition the sequence spaces $\ell^\tau=\ell^\tau(\cX)$, $0<\tau\le\infty$,
consists of all sequences $\{h_\xi : \xi\in\cX\}$ such that
\begin{equation*}
\|h\|_{\ell^\tau}
:= \bigg(\sum_{\xi\in \cX}|h_\xi|^\tau\bigg)^{1/\tau}<\infty.
\end{equation*}
We define the sequence spaces $\fg^q=\fg^q(\cX)$, $0<q\le\infty$, as the set of all sequences $\{h_\xi : \xi\in\cX\}$ such that
\begin{equation}\label{def:g_q}
\|h\|_{\fg^q}
:= \sup_{\xi\in\cX}\bigg(\sum_{Q_\eta\subset Q_\xi}|h_\eta|^q \frac{|Q_\eta|}{|Q_\xi|}\bigg)^{1/q}<\infty.
\end{equation}
Note that the definition of $\ell^\tau(\cX)$ does not need a nested structure associate with $\cX$,
while the definition of $\fg^q(\cX)$, $q<\infty$, requires such kind of structure.
Of course, for $q=\tau=\infty$ we have $\fg^\infty=\ell^\infty$.

The best nonlinear approximation of $h\in\fg^q$ from sequences with at most $n$ non-zero elements is given by
\begin{equation*}
\sigma_n(h)_{\fg^q}:=\inf_{|\supp~\bar{h}|\le n}\|h-\bar{h}\|_{\fg^q}
=\inf_{\substack{\Lambda_n\subset\cX\\ \#\Lambda_n\le n}}\sup_{\xi\in\cX}\Bigg(
\sum_{\substack{Q_\eta\subset Q_\xi\\ \eta\notin\Lambda_n}}|h_\eta|^q \frac{|Q_\eta|}{|Q_\xi|}\Bigg)^{1/q}.
\end{equation*}

\begin{thm}[Jackson estimate]\label{thm:Jackson}
Let $0<\tau<\infty$ and $0<q\le\infty$.
Assume $\{Q_\xi : \xi\in\cX\}$ is a nested structure associate with $\cX$.
There exists a constant $c=c(\tau,q,\kappa)$
such that for any $h\in\ell^{\tau}$ we have $h\in\fg^q$ and
\begin{equation}\label{Jackson}
\sigma_n(h)_{\fg^q}\le c n^{-1/\tau}\|h\|_{\ell^\tau}, \quad n\in\NN.
\end{equation}
\end{thm}

\begin{thm}[Bernstein estimate]\label{thm:Bernstein}
Under the assumptions of Theorem~\ref{thm:Jackson} for any sequence $h:=\{h_\xi\}_{\xi\in \cX}$
with at most $n\in \NN$ non-zero elements we have
\begin{equation}\label{Bernstein}
\|h\|_{\ell^\tau}\le n^{1/\tau}\|h\|_{\fg^q}.
\end{equation}
\end{thm}

\begin{rem}
The matching Jackson and Bernstein estimates \eqref{Jackson} and \eqref{Bernstein} allow for characterization 
of the rate of approximation (approximation spaces) generated by $\{\sigma_n(h)_{\fg^q}\}$.
An interesting phenomenon is that the parameters  $0<\tau<\infty$ and $0<q\le\infty$
in Theorems~\ref{thm:Jackson}, \ref{thm:Bernstein} are not related.
\end{rem}

The proof of Theorem~\ref{thm:Bernstein} is trivial because obviously $\|h\|_{\ell^\infty}\le \|h\|_{\fg^q}$,
while the one of Theorem~\ref{thm:Jackson} is nontrivial.

In the proof of Theorem \ref{thm:Jackson} we shall use several lemmas. The first one is equivalent to the embedding $\ell^\tau\subset\fg^q$.
\begin{lem}\label{lem1}
If $0<\tau<\infty$, $0<q\le\infty$, and $\sA\subset\cX$, then there exists a constant $c_1=c_1(q,\tau,\kappa)$ such that
for any  $\xi\in\cX$ and any sequence $\{h_\eta\}_{\eta\in\sA}$ we have
\begin{equation}\label{lem1:1}
\Bigg(\sum_{\substack{Q_\eta\subset Q_\xi\\ \eta\in\sA}}|h_\eta|^q\frac{|Q_\eta|}{|Q_\xi|} \Bigg)^{1/q}
\le c_1 \Bigg(\sum_{\substack{Q_\eta\subset Q_\xi\\ \eta\in\sA}}|h_\eta|^\tau\Bigg)^{1/\tau}.
\end{equation}
\end{lem}
\begin{proof}
If $\tau\le q$ then \eqref{lem1:1} follows with $c_1=1$ from $|Q_\eta|\le |Q_\xi|$ and the concavity of the function $t^{\tau/q}$.
Now, let $\tau >q$. Applying H\"{o}lder's inequality with $p=\tau/q$ and $p'=\tau/(\tau-q)>1$ we obtain
\begin{align*}
\Bigg(\sum_{\substack{Q_\eta\subset Q_\xi\\ \eta\in\sA}}|h_\eta|^q \frac{|Q_\eta|}{|Q_\xi|}\Bigg)^{1/q}
\le
\Bigg(\sum_{\substack{Q_\eta\subset Q_\xi\\ \eta\in\sA}}|h_\eta|^\tau\Bigg)^{1/\tau}
\Bigg(\sum_{Q_\eta\subset Q_\xi}\left(\frac{|Q_\eta|}{|Q_\xi|}\right)^{p'}\Bigg)^{1/(qp')}.
\end{align*}
For $\xi\in\cX_n$ and $j\in\NN_0$ set $m_j=m_j(\xi):=\#\{\zeta\in\cX_{n+j} : Q_\zeta\subset Q_\xi\}$. From \eqref{dyad_eq:6} we have $m_j\ge 2^j$. 
For every $\eta\in\cX_{n+j}$ we obtain from \eqref{dyad_eq:4a} and \eqref{dyad_eq:7} $|Q_\eta|\le\kappa|Q_\zeta|$ for every $\zeta\in\cX_{n+j}$
and hence $|Q_\eta|\le\kappa |Q_\xi|/m_j$. Therefore
\begin{equation*}
\sum_{Q_\eta\subset Q_\xi}\left(\frac{|Q_\eta|}{|Q_\xi|}\right)^{p'}
\le \sum_{j\ge 0}m_j\kappa^{p'}m_j^{-p'}
\le \kappa^{p'}\sum_{j\ge 0}2^{-j (p'-1)}=c_1 < \infty. 
\end{equation*}
Estimate \eqref{lem1:1} follows readily from above.
\end{proof}

The second lemma is obvious.
\begin{lem}\label{lem2}
If $0<\tau<\infty$ and $a_1\ge a_2\ge\dots\ge 0$ then for every $m\in \NN$ we have
\begin{equation}\label{lem2:1}
a_m\le m^{-1/\tau} \bigg(\sum_{k=1}^\infty a_k^\tau\bigg)^{1/\tau}.
\end{equation}
\end{lem}

The third lemma shows that the sets from a finite tree cannot concentrate a lot of area/volume outside the branching nodes.
\begin{lem}\label{lem3}
Let $\cZ\subset\cX$ be a finite tree with root $\zeta\in \cX$.
Denote by $\cZ_{(b)}$ the set of branching nodes in $\cZ$ and by $\cZ_{(m)}$ the set of minimal nodes in $\cZ$, 
i.e. nodes whose children in $\cX$ are not in $\cZ$.
Then
\begin{equation}\label{lem3:1}
\sum_{\eta\in\cZ\backslash\cZ_{(b)}}|Q_\eta|+\frac{\lambda}{1-\lambda}\sum_{\eta\in\cZ_{(m)}}|Q_\eta|\le\frac{1}{1-\lambda}|Q_\zeta|
\end{equation}
with $\lambda$ from \eqref{dyad_eq:6}.
\end{lem}
\begin{proof}
We shall carry out the proof of this lemma by induction on $\#\cZ_{(b)}$. 
First, we prove \eqref{lem3:1} in the case $\#\cZ_{(b)}=0$, i.e. when $\cZ$ has no branching nodes,
by a straight application of \eqref{dyad_eq:6}. Take $n$ such that $\zeta\in\cX_n$, set $\eta_0=\zeta$
and denote the other nodes of $\cZ$ by $\eta_\nu\in\cX_{n+\nu}$, $\nu=1,\dots,N$. Then
\begin{align*}
\sum_{\eta\in\cZ\backslash\cZ_{(b)}}|Q_\eta|+\frac{\lambda}{1-\lambda}\sum_{\eta\in\cZ_{(m)}}|Q_\eta|
&=\sum_{\nu=0}^N |Q_{\eta_\nu}|+\frac{\lambda}{1-\lambda}|Q_{\eta_N}|
\\
&\le\bigg(\sum_{\nu=0}^N\lambda^\nu+\frac{\lambda}{1-\lambda}\lambda^N\bigg)|Q_\zeta|
=\frac{1}{1-\lambda}|Q_\zeta|.
\end{align*}
Assume \eqref{lem3:1} is established for all finite trees $\cZ\subset \cX$ with $\#\cZ_{(b)}\le k$.
Let $\tilde\cZ\subset \cX$ be a finite tree with $\#\tilde\cZ_{(b)}=k+1$.
Choose a branching node $\zeta_0$ 
whose descendants in $\tilde\cZ$ are not branching nodes.
Then the tree $\cZ:=\{\eta\in\tilde{\cZ} : Q_\eta\not\subset Q_{\zeta_0}\}\cup\{\zeta_0\}$ will satisfy $\#\cZ_{(b)}= k$ and $\zeta_0\in\cZ_{(m)}$.
Denote by $\zeta_1,\dots,\zeta_r$ the children of $\zeta_0$ in $\tilde\cZ$. Then \eqref{lem3:1} with $\#\cZ_{(b)}=0$ implies
\begin{equation*}
\sum_{\substack{\eta\in\cZ\backslash\cZ_{(b)}\\ Q_\eta\subset Q_{\zeta_0}}}|Q_\eta|
+\frac{\lambda}{1-\lambda}\sum_{\substack{\eta\in\cZ_{(m)}\\ Q_\eta\subset Q_{\zeta_0}}}|Q_\eta|
\le\sum_{i=1}^r \frac{1}{1-\lambda}|Q_{\zeta_i}|
\le\frac{1}{1-\lambda}|Q_{\zeta_0}|.
\end{equation*}
Now, the above inequality and the inductive assumption \eqref{lem3:1} applied to $\cZ$ with $\#\cZ_{(b)}= k$ give
\begin{multline*}
\sum_{\eta\in\tilde{\cZ}\backslash\tilde{\cZ}_{(b)}}|Q_\eta|+\frac{\lambda}{1-\lambda}\sum_{\eta\in\tilde{\cZ}_{(m)}}|Q_\eta|\\
\le\sum_{\eta\in\cZ\backslash\cZ_{(b)}}|Q_\eta|+\frac{\lambda}{1-\lambda}\sum_{\eta\in\cZ_{(m)}}|Q_\eta|\le\frac{1}{1-\lambda}|Q_\zeta|.
\end{multline*}
This completes the proof.
\end{proof}

\begin{proof}[Proof of Theorem \ref{thm:Jackson}]
For $q=\infty$ \eqref{Jackson} follows immediately from \lemref{lem2}.
So, in what follows we assume $q<\infty$.

First, consider the case when $\cX=\cup_{n=0}^\infty\cX_n$, 
where all sets $\{Q_\eta\}_{\eta\in\cX}$ are contained in the set $Q_{\xi_0}$, $\xi_0\in\cX_0$.
For a given sequence $h\in\ell^\tau$ we set $a_\xi=|h_\xi|$ and rewrite inequality \eqref{Jackson} as
\begin{equation}\label{Jackson_b}
\inf_{\substack{\Lambda_n\subset\cX\\ \#\Lambda_n\le n}}\sup_{\xi\in\cX}
\Bigg(\sum_{\substack{\eta\in\cX\backslash\Lambda_n\\ Q_\eta\subset Q_\xi}}\frac{|Q_\eta|}{|Q_\xi|}a_\eta^q \Bigg)^{1/q}
\le c n^{-1/\tau}\bigg(\sum_{\eta\in \cX}a_\eta^\tau\bigg)^{1/\tau}.
\end{equation}

Set $m=\left\lfloor (n+3)/2\right\rfloor\ge 2$.

Given $\{a_\eta\}$ such that $a_\eta\ge 0$ and $\sum_{\eta\in \cX}a_\eta^\tau<\infty$, 
we set $A_\xi=\sum_{Q_\eta\subset Q_\xi}a_\eta^\tau$ and $A=A_{\xi_0}$.
Then for a given $m\in\NN$ we split $\cX$ into the union of two disjoint sets $\cZ^1$ and $\cZ^2$, $\cX=\cZ^1\cup\cZ^2$,  
defined by
\begin{equation}\label{thm:eq:1}
\cZ^1=\{\xi\in\cX : A_\xi>m^{-1}A\},\quad\quad \cZ^2=\{\xi\in\cX : A_\xi\le m^{-1}A\}.
\end{equation}

Observe that if $Q_\eta\subset Q_\xi$, then: (i) if $\eta\in\cZ^1$ then $\xi\in\cZ^1$;
(ii) if $\xi\in\cZ^2$ then $\eta\in\cZ^2$. Hence $\cZ^1$ is a tree with root at $\xi_0$.
If we order $a_\eta$ by their magnitude $a_{\eta_1}\ge a_{\eta_2}\ge\dots\ge0$ then there exists $K$ such that $\sum_{k>K}a_{\eta_k}^\tau\le m^{-1}A$.
Let $j_k$ be such that $\eta_k\in\cX_{j_k}$. Set $M=\max\{j_k : 1\le k\le K\}$. Then $\cZ^1\subset \cup_{j=0}^M\cX_j$ and hence $\cZ^1$ is a finite tree.

Denote by $\cZ^2_0$ the set of all $\eta\in\cZ^2$ with parents in $\cZ^1$.
Then $\cZ^2_0$ is finite (because $\cZ^1$ is finite) and contains all maximal (with respect to inclusion) elements of $\cZ^2$.
Moreover, if $\eta$ and $\zeta$ are two different elements of $\cZ^2_0$, then $Q_\eta\cap Q_\zeta=\emptyset$
and $\cZ^2=\cup_{\eta\in\cZ^2_0} \{\zeta\in\cX : Q_\zeta\subset Q_\eta\}$.

Denote by $\cZ^1_{(m)}$ the elements of $\cZ^1$ with no children in $\cZ^1$
and by $\cZ^1_{(b)}$ the branching nodes of $\cZ^1$, i.e. the elements of $\cZ^1$ with more than one child in $\cZ^1$.
From $A_\xi>m^{-1}A$ for all $\xi\in \cZ^1_{(m)}$  we conclude that $\#\cZ^1_{(m)}\le m-1$
and now $\#\cZ^1_{(b)}\le \#\cZ^1_{(m)}-1$ implies $\#\cZ^1_{(b)}\le m-2$.
Observe that $\#\cZ^1$ can be arbitrary large.

We include in $\Lambda_n$ the indices $\eta_1,\dots,\eta_{m-1}$ of the $m-1$ largest elements $a_{\eta_1}\ge a_{\eta_2}\ge\dots\ge a_{\eta_{m-1}}$
and all elements of $\cZ^1_{(b)}$.
Thus the total number of elements in $\Lambda_n$ is at most $2m-3\le n$ and $\Lambda_n\subset\cZ^1$.

Now, we are ready to establish \eqref{Jackson_b} with the above selection of $\Lambda_n$.
If $\xi\in\cZ^2$, then the fact that $\Lambda_n\subset\cZ^1$, \lemref{lem1}, and \eqref{thm:eq:1} imply
\begin{align}\label{thm:eq:2}
\Bigg(\sum_{\substack{\eta\in\cX\backslash\Lambda_n\\ Q_\eta\subset Q_\xi}}\frac{|Q_\eta|}{|Q_\xi|}a_\eta^q \Bigg)^{1/q}
&= \Bigg(\sum_{Q_\eta\subset Q_\xi}\frac{|Q_\eta|}{|Q_\xi|}a_\eta^q \Bigg)^{1/q}
\\
&\le c_1 \Bigg(\sum_{Q_\eta\subset Q_\xi}a_\eta^\tau\Bigg)^{1/\tau}= c_1 A_\xi^{1/\tau}\le c_1 m^{-1/\tau}A^{1/\tau}. \nonumber
\end{align}

In the case $\xi\in\cZ^1$ we write
\begin{equation}\label{thm:eq:3}
\sum_{\substack{\eta\in\cX\backslash\Lambda_n\\ Q_\eta\subset Q_\xi}}\frac{|Q_\eta|}{|Q_\xi|}a_\eta^q
=\sum_{\substack{\eta\in\cZ^1\backslash\Lambda_n\\ Q_\eta\subset Q_\xi}}\frac{|Q_\eta|}{|Q_\xi|}a_\eta^q
+\sum_{\substack{\eta\in\cZ^2\\ Q_\eta\subset Q_\xi}}\frac{|Q_\eta|}{|Q_\xi|}a_\eta^q  =:S_1+S_2.
\end{equation}
For the estimation of $S_1$ we first observe that \lemref{lem2} implies
\begin{equation}\label{thm:eq:4}
a_\eta\le \max_{\zeta\in\cZ^1\backslash\Lambda_n}a_\zeta\le m^{-1/\tau}A^{1/\tau},\quad \forall \eta\in\cZ^1\backslash\Lambda_n.
\end{equation}
From \eqref{thm:eq:4}, the fact that $\cZ^1_{(b)}\subset\Lambda_n$, and \lemref{lem3} we obtain
\begin{equation}\label{thm:eq:5}
S_1= \sum_{\substack{\eta\in\cZ^1\backslash\Lambda_n\\ Q_\eta\subset Q_\xi}}\frac{|Q_\eta|}{|Q_\xi|}a_\eta^q
\le \sum_{\substack{\eta\in\cZ^1\backslash\cZ^1_{(b)}\\ Q_\eta\subset Q_\xi}}\frac{|Q_\eta|}{|Q_\xi|}(m^{-1/\tau}A^{1/\tau})^q
\le\frac{1}{1-\lambda}(m^{-1/\tau}A^{1/\tau})^q.
\end{equation}
The desired estimate of $S_2$ follows readily from \eqref{thm:eq:2} with $\xi$ replaced by $\zeta\in\cZ^2_0$
\begin{equation}\label{thm:eq:6}
S_2=\sum_{\substack{\zeta\in\cZ^2_0\\ Q_\zeta\subset Q_\xi}}\frac{|Q_\zeta|}{|Q_\xi|}\sum_{Q_\eta\subset Q_\zeta}\frac{|Q_\eta|}{|Q_\zeta|}a_\eta^q
\le \sum_{\substack{\zeta\in\cZ^2_0\\ Q_\zeta\subset Q_\xi}}\frac{|Q_\zeta|}{|Q_\xi|}(c_1 m^{-1/\tau}A^{1/\tau})^q\le (c_1 m^{-1/\tau}A^{1/\tau})^q,
\end{equation}
where we used that the sets $\{Q_\zeta : \zeta\in\cZ^2_0\}$ are disjoint.
Now, combining \eqref{thm:eq:5}, \eqref{thm:eq:6}, and \eqref{thm:eq:3} we obtain for $\xi\in\cZ^1$
\begin{equation}\label{thm:eq:7}
\Bigg(\sum_{\substack{\eta\in\cX\backslash\Lambda_n\\ Q_\eta\subset Q_\xi}}\frac{|Q_\eta|}{|Q_\xi|}a_\eta^q\Bigg)^{1/q}
\le ((1-\lambda)^{-1}+c_1^q)^{1/q}m^{-1/\tau}A^{1/\tau}.
\end{equation}
Finally, \eqref{thm:eq:7} and \eqref{thm:eq:2} give \eqref{Jackson_b} with $c=((1-\lambda)^{-1}+c_1^q)^{1/q}2^{1/\tau}$
and prove the theorem in the case $\cX=\cup_{n=0}^\infty\cX_n$.

Now, we consider the case $\cX=\cup_{n=-\infty}^\infty\cX_n$. Here we work separately in every $\fD^1,\dots,\fD^J$.
Set $\cY_j=\{\xi\in\cX : Q_\xi\subset\fD^j\}$, $j=1,\dots,J$.
Then $\cX=\cup_{j=1}^J \cY_j$ and $Q_\xi\cap Q_\eta=\emptyset$ if $\eta\in \cY_i, \xi\in \cY_j, i\ne j$.

Fix $j, 1\le j\le J$
and set $A=\sum_{\eta\in\cY_j}a_\eta^\tau<\infty$.
If we order $a_\eta$, $\eta\in\cY_j$, by their magnitude $a_{\eta_1}\ge a_{\eta_2}\ge\dots\ge0$,
 then there exists $K$ such that $\sum_{k>K}a_{\eta_k}^\tau\le n^{-1}A$.
Also, there exists $\xi_0\in\cY_j$ such that $Q_{\eta_k}\subset Q_{\xi_0}$ for $1\le k\le K$.
Now we work with this $\xi_0$ as in the case $\cX=\cup_{n=0}^\infty\cX_n$ with $\cY_j$ and $n/J$ instead of $\cX$ and $n$.
Let $\cZ^1, \cZ^2, \Lambda_n$ be defined as above.

Fix $\xi\in\cY_j$  and write
\begin{equation}\label{thm:eq:8}
\sum_{\substack{\eta\in\cY_j\backslash\Lambda_n\\ Q_\eta\subset Q_\xi}}\frac{|Q_\eta|}{|Q_\xi|}a_\eta^q
=\sum_{\substack{\eta\in\cY_j\backslash\Lambda_n\\ Q_\eta\subset Q_\xi\cap Q_{\xi_0}}}\frac{|Q_\eta|}{|Q_\xi|}a_\eta^q
+\sum_{\substack{\eta\in\cY_j\\ Q_\eta\subset Q_\xi\\ Q_\eta\not\subset Q_{\xi_0}}}\frac{|Q_\eta|}{|Q_\xi|}a_\eta^q=:S_1+S_2.
\end{equation}
In \eqref{thm:eq:8} $S_1$ is an empty sum if $Q_\xi\cap Q_{\xi_0}=\emptyset$ and $S_2$ is an empty sum if $Q_\xi\subset Q_{\xi_0}$.
If $Q_\xi\subset Q_{\xi_0}$ then we have
\begin{equation}\label{thm:eq:9}
S_1\le \bigg(c n^{-1/\tau}\Big(\sum_{\eta\in \cY_j}a_\eta^\tau\Big)^{1/\tau}\bigg)^q.
\end{equation}
in view of \eqref{Jackson_b}, which has already been established in this case.
If $Q_{\xi_0}\subset Q_\xi$ then we have
\begin{equation}\label{thm:eq:10}
S_1=\frac{|Q_{\xi_0}|}{|Q_\xi|}\sum_{\substack{\eta\in\cY_j\backslash\Lambda_n\\ Q_\eta\subset Q_{\xi_0}}}\frac{|Q_\eta|}{|Q_{\xi_0}|}a_\eta^q
\le \sum_{\substack{\eta\in\cY_j\backslash\Lambda_n\\ Q_\eta\subset Q_{\xi_0}}}\frac{|Q_\eta|}{|Q_{\xi_0}|}a_\eta^q
\le \bigg(c n^{-1/\tau}\Big(\sum_{\eta\in \cY_j}a_\eta^\tau\Big)^{1/\tau}\bigg)^q.
\end{equation}
in view of \eqref{thm:eq:9} with $\xi=\xi_0$. 
\lemref{lem1} with $\sA=\{\eta\in\cY_j : Q_\eta\subset Q_\xi, Q_\eta\not\subset Q_{\xi_0}\}$ yields
\begin{equation}\label{thm:eq:11}
S_2\le \Bigg(c \bigg(\sum_{\substack{\eta\in\cY_j\\ Q_\eta\subset Q_\xi, Q_\eta\not\subset Q_{\xi_0}}}a_\eta^\tau\bigg)^{1/\tau}\Bigg)^q
\le \bigg(c n^{-1/\tau}\Big(\sum_{\eta\in \cY_j}a_\eta^\tau\Big)^{1/\tau}\bigg)^q.
\end{equation}
Finally, using \eqref{thm:eq:9} (or \eqref{thm:eq:10}) and \eqref{thm:eq:11} in \eqref{thm:eq:8}
and summing over $j=1,\dots,J$ we complete the proof of \eqref{Jackson}.
\end{proof}

Theorem~\ref{thm:Jackson} readily implies the following
\begin{cor}\label{cor:vanish}
Let $0<\tau<\infty$ and $0<q\le\infty$. If $h\in\ell^\tau$, then $h$ belongs to the closure of the finitely supported sequences in the norm of $\fg^q$.
\end{cor}

\subsection{Proof of Theorem~\ref{thm:Jackson_TL}}

We shall obtain \thmref{thm:Jackson_TL} as a consequence of \thmref{thm:Jackson}
and the frame decompositions of $\sepTL{\cF}{0}{q}(\SS)$ and $\cB^{s\tau}_{\tau}(\SS)$.
\thmref{thm:Jackson} uses the sequence space $\fg^q$ defined with the help of a nested structure given in Definition~\ref{nested_structure}.
For $\SS$ we shall use the index set $\cX=\cup_{n=0}^\infty\cX_n$ as defined in \S\ref{s5_3} and the nested structure from \thmref{nested_sets}.
Note that conditions \eqref{dyad_eq:1}--\eqref{dyad_eq:3} in Definition~\ref{nested_structure} follow from \eqref{dyadic_eq:1}--\eqref{dyadic_eq:3};
\eqref{dyad_eq:4a} follow from \eqref{dyadic_eq:6}.
Finally, \eqref{dyad_eq:6a} follows from \propref{prop:nested_sets}~(2) by choosing $b\ge 4$.
Thus, \thmref{nested_sets} with $b=4$ and $\bar{\beta}=1/12$ gives a nested structure
in the sense of Definition~\ref{nested_structure}.

Let $\Theta=\{\theta_\xi\}_{\xi\in\cX}$ be the system 
from Theorems~\ref{thm:prop-frame} and \ref{thm:frame-infty}.
It is important that $\Theta$ provides a frame decomposition for $\sepTL{\cF}{0}{q}(\SS)$ and $\cB^{s\tau}_{\tau}(\SS)$ simultaneously.
Recall that each element $\theta_\xi$ is a linear combination of at most $\tilde{n}$ shifts of the Newtonian kernel.
Using \thmref{thm:prop-frame}~(c) with $p=q=\tau$ we obtain that for any $f\in \cB^{s\tau}_{\tau}$ 
we have $\{\langle f, \tilde\theta_\eta\rangle\}\in\fb^{s\tau}_\tau$,
where $\tilde\theta_\eta$ is defined in \eqref{def-f-dual-f}.
Set $h_\eta= |Q_\eta|^{-1/2}\langle f, \tilde\theta_\eta\rangle$. Then in view of \eqref{def-b-space} we have $h\in\ell^\tau$.
Now, according to \thmref{thm:Jackson} $h\in\fg^q$, i.e. $\{\langle f, \tilde\theta_\eta\rangle\}\in\ff^{0q}_\infty$
(see \eqref{def-f-space-infty2} and \eqref{def:g_q}) and \corref{cor:vanish} implies $\{\langle f, \tilde\theta_\eta\rangle\}\in\sepTL{\ff}{0}{q}$.
Applying \thmref{thm:frame-infty} we conclude that $f\in \sepTL{\cF}{0}{q}(\SS)$ and
\begin{equation}\label{eq:Jackson_TL1}
\|f\|_{\cF_\infty^{0q}}\le c \|h\|_{\fg^q} \le c \|h\|_{\ell^\tau} \le c \|f\|_{\cB^{s\tau}_{\tau}}.
\end{equation}

If $n< \tilde{n}$, then \eqref{jackson_TL} follows from $E_n(f)_{\cF_\infty^{0q}}\le\|f\|_{\cF_\infty^{0q}}$ and \eqref{eq:Jackson_TL1}.
Now, let $n\ge \tilde{n}$. Set $m=\left\lfloor n/ \tilde{n}\right\rfloor\ge 1$.
\thmref{thm:Jackson} implies the existence of an index set $\Lambda\subset\cX$ such that $\#\Lambda\le m$ and
\begin{equation}\label{eq:Jackson_TL2}
\sup_{\xi\in\cX}\Bigg(\sum_{Q_\eta\subset Q_\xi, \eta\notin\Lambda}|h_\eta|^q \frac{|Q_\eta|}{|Q_\xi|}\Bigg)^{1/q}
\le c m^{-1/\tau}\|h\|_{\ell^\tau}\le c n^{-1/\tau}\|h\|_{\ell^\tau}.
\end{equation}
Set $g=\sum_{\eta\in\Lambda}\langle f, \tilde\theta_\eta\rangle \theta_\eta$ and $\bar{h}=\sum_{\eta\in\Lambda}h_\eta$.
Then, using \thmref{thm:frame-infty}, \eqref{eq:Jackson_TL2}, and \thmref{thm:prop-frame}~(c) we obtain
\begin{equation*}
E_n(f)_{\cF_\infty^{0q}}\le\|f-g\|_{\cF_\infty^{0q}}\le c \|h-\bar{h}\|_{\fg^q}
\le c n^{-1/\tau}\|h\|_{\ell^\tau}
\le c n^{-s/(d-1)}\|f\|_{\cB^{s\tau}_{\tau}},
\end{equation*}
which completes the proof of Theorem~\ref{thm:Jackson_TL}.
\qed

\end{document}